\documentclass[preprint,3p,sort&compress]{elsarticle}

\usepackage{amsmath, amssymb}
\usepackage{subfigure}
\usepackage{IEEEtrantools}
\usepackage{placeins}
\usepackage{framed}

\numberwithin{equation}{section}

\newcommand{\kdifform}[2]{#1^{(#2)}} 
\newcommand{\kdifformh}[2]{#1^{(#2)}_{h}} 
\newcommand{\reconstruction}{\mathcal{I}} 
\newcommand{\reduction}{\mathcal{R}} 
\newcommand{\kformspace}[1]{\Lambda^{#1}} 
\newcommand{\kformspacedomainh}[2]{\Lambda_h^{#1}(#2)} 
\newcommand{\incidenceboundary}[2]{\mathsf{E}_{(#1,#2)}} 
\newcommand{\incidencederivative}[2]{\mathsf{E}^{(#1,#2)}} 
\newcommand{\innerspace}[3]{\left(#1,#2\right)_{#3}} 
\newcommand{\matrixoperator}[1]{#1}

\newcommand{\projection}{\pi_{h}} 

\newcommand{\manifold}[1]{\mathcal{#1}} 
\newcommand{\tangentspace}[1]{T_{p}\manifold{#1}} 
\newcommand{\vectorfieldspace}[1]{\Xi_{#1}} 

\newcommand{\ederiv}{\mathrm{d}} 

\renewcommand{\eqref}[1]{(\ref{#1})} 
\newcommand{\figref}[1]{Figure~\ref{#1}} 
\newcommand{\secref}[1]{Section~\ref{#1}} 
\newcommand{\theoremref}[1]{Theorem~\ref{#1}} 
\newcommand{\exampleref}[1]{Example~\ref{#1}} 

\newtheorem{theorem}{Theorem}
\newtheorem{proof}{Proof}

\newtheorem{definition}{Definition}
\newtheorem{example}{Example}

\newtheorem{remark}{Remark}

\begin{document}

\begin{frontmatter}


\author[tue]{Artur Palha\corref{cor1}}
\ead{A.Palha@tue.nl}

\author[tudelft]{Marc Gerritsma}
\ead{M.I.Gerritsma@tudelft.nl}

\cortext[cor1]{Corresponding author}

\address[tue]{Eindhoven University of Technology, Department of Mechanical Engineering, section of Control Systems Technology P.O. Box 513 5600 MB Eindhoven, The Netherlands}
\address[tudelft]{Delft University of Technology, Faculty of Aerospace Engineering, Aerodynamics Group P.O. Box 5058, 2600 GB Delft, The Netherlands}


\title{Mimetic spectral element method for Hamiltonian systems}

\begin{abstract}
There is a growing interest in the conservation of invariants when numerically solving a system of ordinary differential equations. Methods that exactly preserve these quantities in time are known as \emph{geometric integrators}. In this paper we apply the recently developed mimetic framework \cite{Kreeft2011,Palha2014} to the solution of a system of first order ordinary differential equations. Depending on the discrete Hodge-$\star$ employed, two classes of arbitrary order time integrators are derived. It is shown that  the one based on a canonical Hodge-$\star$ results in a symplectic integrator, whereas the one based on a Galerkin Hodge-$\star$ results in an energy preserving integrator. A set of numerical tests confirms these theoretical results.
\end{abstract}


\end{frontmatter}


\section{INTRODUCTION}
\label{Section::Introduction}

In the last decades there has been a significant development of discretization methods for the solution of ordinary differential equations (\emph{ODE}s). These equations are ubiquitous, appearing in particle accelerators, molecular dynamics, vortex particle methods, stock market modeling, celestial dynamics, etc. From a mathematical point of view, these problems can be generalized as a system of first order ODEs:
\begin{equation}
\frac{\ederiv \boldsymbol{y}}{\ederiv t} = \boldsymbol{h}(\boldsymbol{y},t),\quad \boldsymbol{y}(t_{0}) = \boldsymbol{y}^{0},\quad \boldsymbol{y}\in\mathcal{M}\subset\mathbb{R}^{n}\quad \text{and} \quad t\in [t_{0},t_{f}]\subset\mathbb{R}\;.
\label{eq::TestCases_ode}
\end{equation}
Where $\boldsymbol{h}(\boldsymbol{y},t) \in\mathbb{R}^{n}$ and the region $\mathcal{M}$ is the phase space of the system. For the case of Hamiltonian systems, $\boldsymbol{y} = (\boldsymbol{p},\boldsymbol{q})$, $\boldsymbol{q}\in\mathbb{R}^{m}$, $\boldsymbol{q}\in\mathbb{R}^{m}$ and \eqref{eq::TestCases_ode} takes the form:

\begin{equation}
	\frac{\ederiv \boldsymbol{y}}{\ederiv t} = \mathsf{J}^{-1}\nabla H(\boldsymbol{y}),\quad \boldsymbol{y}(t_{0}) = \boldsymbol{y}^{0},\quad \boldsymbol{y}\in\mathcal{M}\subset\mathbb{R}^{2m}\quad \text{and} \quad t\in I\subset\mathbb{R} \quad \text{with}\quad \mathsf{J} = \left[\begin{array}{cc} \mathsf{0} & \mathsf{I}_{m} \\ -\mathsf{I}_{m}  & \mathsf{0}\end{array}\right]\;, \label{eq::hamiltonian_system}
\end{equation}
where $\mathsf{I}_{m}$ is the $m\times m$ identity matrix and $h = \mathsf{J}^{-1}\nabla H(\boldsymbol{y})$.

Traditionally, the methods developed for the discrete solution of these problems derive from finite difference principles and Taylor series expansions. This approach focuses on minimizing the local truncation error associated with the discretization. No matter how small this truncation error is, it is always finite, which might lead  to a continuously growing global error. Over a sufficiently long time interval, the global error will inevitably reach unacceptable values and eventually the numerical solution will no longer resemble the analytical one. Due to this intrinsic local truncation error of the discretization process, one wishes to develop methods that are able to preserve the same qualitative properties as the continuous system. This will enable acceptable long time integration.

By \emph{qualitative properties} one refers, for example, to the undermentioned properties, \cite{budd2003}:
\begin{description}
	\item[Geometric structure:] The evolution equations preserve geometric properties, such as orthogonality of lines or surfaces, volumes in phase space, etc.
	\item[Conservation laws:] While evolving in time, certain well known quantities as mass, momentum, energy, or other less known quantities such as the radius of the trajectory or functions of the dynamical variables are preserved.
	\item[Symmetries:] Many systems remain invariant under specific transformations, e.g.: Galilean transformations (translations, reflexions and rotations), reversal symmetries (of which time reversal is an example), scaling symmetries (invariance under rescalings in either time or space).
\end{description}

The key word to retain from qualitative properties is \emph{conservation}.  When we refer to \emph{qualitative properties} we mean the invariance of certain quantities or geometric relations, during time evolution. These qualitative properties are explicit constraints that restrict the universe of allowable solution trajectories, discarding solutions which are not physically compatible with the system of equations one wishes to solve.

A simple yet illustrative example for the relevance of geometric integration is presented, among others, in \cite{budd2003} for the discretisation of the harmonic oscillator with a forward Euler scheme. 
\begin{example}\label{ex:harmonic_oscillator}
Consider the two coupled ODEs that describe the evolution of an harmonic oscillator:
\begin{equation}
	\left\{ \,
		\begin{IEEEeqnarraybox}[\IEEEeqnarraystrutmode\IEEEeqnarraystrutsizeadd{7pt}{7pt}][c]{rCl}
			\frac{\ederiv q}{\ederiv t} & = & p\\
			\frac{\ederiv p}{\ederiv t} & = & -q
		\end{IEEEeqnarraybox}
	\right.\;,\quad\mathrm{with}\,\, q(t_{0}) = q_{0},\, p(t_{0}) = p_{0}\,\,\mathrm{and}\,\, t\in[t_{0},t_{f}]\;.
	\label{eq:harmonic_oscillator}
\end{equation}
It is possible to verify that all solutions are periodic and bounded and that $\left(q^{2}+p^{2}\right)$ is a conserved quantity of the evolution. If a forward Euler discretization is applied to the system, the following discrete time evolution expressions are obtained:
\begin{equation}
	\left\{ \,
		\begin{IEEEeqnarraybox}[\IEEEeqnarraystrutmode\IEEEeqnarraystrutsizeadd{7pt}{7pt}][c]{rCl}
			q_{n+1} & = & q_{n} + \Delta t p_{n}\\
			p_{n+1} & = & p_{n} - \Delta t q_{n}
		\end{IEEEeqnarraybox}
	\right.\;,\quad\mathrm{with}\,\, q(t_{0}) = q_{0},\, p(t_{0}) = p_{0}\;.
	\label{eq:harmonic_oscillator_euler}
\end{equation}
It is straightforward to see that with this time stepping scheme the quantity $\left(q_{n}^{2}+p_{n}^{2}\right)$ grows geometrically at each time step:
\[
	q^{2}_{n+1} + p^{2}_{n+1} = (1+\Delta t^{2})(q^{2}_{n}+p^{2}_{n})\;.
\]
Due to this feature of the forward Euler discretization, the discrete solutions of the harmonic oscillator will no longer be periodic and bounded. Moreover, since up to a multiplicative constant  $\left(q^{2}+p^{2}\right)$ is the energy of the system, conservation of energy is no longer satisfied. Another relevant point is that this discretization renders the discrete system time-irreversible, in contrast to the time-reversibility of the continuous system. This can be seen by using \eqref{eq:harmonic_oscillator_euler} with a negative time step starting from $(x_{n+1},q_{n+1})$:
\begin{equation}
	\left\{ \,
		\begin{IEEEeqnarraybox}[\IEEEeqnarraystrutmode\IEEEeqnarraystrutsizeadd{7pt}{7pt}][c]{rCl}
			\tilde{q}_{n} & = & q_{n+1} - \Delta t p_{n+1} = q_{n} + \Delta t p_{n} - \Delta t (p_{n} - \Delta t q_{n}) =  q_{n}(1+\Delta t^{2}) \\
			\tilde{p}_{n} & = & p_{n+1} + \Delta t q_{n+1} = p_{n} - \Delta t (q_{n} + \Delta t p_{n})+ \Delta t q_{n+1}  = p_{n}(1+\Delta t^{2}) \,.
		\end{IEEEeqnarraybox}
	\right. 
	\label{eq:harmonic_oscillator_euler_negative_delta_t}
\end{equation}
Where we have used \eqref{eq:harmonic_oscillator_euler} to replace $q_{n+1}$ and $p_{n+1}$. Since $\tilde{q}_{n} \neq q_{n}$ and  $\tilde{p}_{n} \neq p_{n}$, this method is not time reversible.
 
On the other hand, if the implicit midpoint rule is applied to this system, the alternative discrete time evolution expressions become:
\begin{equation}
	\left\{ \,
		\begin{IEEEeqnarraybox}[\IEEEeqnarraystrutmode\IEEEeqnarraystrutsizeadd{7pt}{7pt}][c]{rCl}
			q_{n+1} & = & q_{n} \frac{4-\Delta t^{2}}{4+\Delta t^{2}} + p_{n}\frac{4\Delta t}{4+\Delta t^{2}}\\
			p_{n+1} & = & p_{n} \frac{4-\Delta t^{2}}{4+\Delta t^{2}} - q_{n}\frac{4\Delta t}{4+\Delta t^{2}}
		\end{IEEEeqnarraybox}
	\right.\;,\quad\mathrm{with}\,\, q(t_{0}) = q_{0},\, p(t_{0}) = p_{0}\;.
	\label{eq:harmonic_oscillator_euler_midpoint}
\end{equation}
For this discretization, the quantity $\left(q^{2}+p^{2}\right)$ is conserved, regardless of the time step used:
\[
	q^{2}_{n+1} + p^{2}_{n+1} = q^{2}_{n}+p^{2}_{n}\;.
\]
Contrary to the forward Euler discretization, the midpoint rule is capable of preserving the same qualitative properties of the continuous system, namely: the discrete solutions are periodic and bounded and the quantity $\left(q^{2}+p^{2}\right)$ is conserved. Additionally, it is possible to show that the discrete system is time reversible, just like the original dynamics.
\end{example}

This simple example expresses the relevance of judiciously selecting a numerical integrator. If correctly chosen, it is capable of preserving the qualitative properties of the continuous system. The approach that incorporates global characteristics of the problem being solved to construct discretization schemes with the same properties of the continuous one, gives rise to a class of time integrators referred in literature as \emph{geometric time integrators}.

In this paper, the focus is on \emph{autonomous Hamiltonian} problems, i.e. ones such that $\boldsymbol{h}(\boldsymbol{y},t) = \mathsf{J}^{-1} \nabla H \left( \boldsymbol{y} \right)$. The scalar function $ H \left( \boldsymbol{y} \right)$ is the Hamiltonian of the problem and its value is independent of time,
\begin{equation}
   H \left( \boldsymbol{y}(t) \right) \equiv  H \left( \boldsymbol{y}_{0} \right), \quad \forall t \geq t_{0}\;.
\end{equation}
In the case of isolated mechanical systems, the Hamiltonian and energy are associated. Another characteristic of Hamiltonian systems is that they are symplectic, see \cite{Hairer2006} for example. The exact flow map, $\varphi_{\tau}$, associated to a system of ODEs is defined as $\varphi_{\tau}(\boldsymbol{y}_{0}):=\boldsymbol{y}(\tau)$, where $\boldsymbol{y}(\tau)$ is the solution of the initial value problem at time $t=\tau$. A system of ODEs is symplectic if the associated Jacobian of the flow map, $\varphi^{\prime}_{\tau}$, satisties:
\begin{align}
 \varphi_{\tau}^{\prime}\left( y \right)^{T} J \varphi_{\tau}^{\prime}\left( y \right) = J.
\end{align}
Symplecticity implies, for instance, that the flow map preserves phase space volume, which imposes strong constraints on the admissible solutions. This can be stated mathematically as:
\begin{equation}
	\ederiv q(t) \wedge \ederiv p(t) = \ederiv q(t_{0}) \wedge \ederiv p(t_{0}) \label{eq:symplecticity}\;,
\end{equation} 
where $t > t_{0}$.

If we apply equation \eqref{eq:symplecticity} to the Euler discretization, \eqref{eq:harmonic_oscillator_euler}, used in Example~\ref{ex:harmonic_oscillator} we see that:
\[
	\ederiv q_{n+1}\wedge \ederiv p_{n+1} = \left(\ederiv q_{n} + \Delta t\,\ederiv p_{n}\right)\wedge\left(\ederiv p_{n} - \Delta t\, \ederiv q_{n}\right) = \ederiv q_{n}\wedge \ederiv p_{n} + \Delta t^{2}\left(\ederiv q_{n}\wedge \ederiv p_{n}\right)\,, 
\]
showing that the Euler discretization is not symplectic. On the other hand, if we apply equation \eqref{eq:symplecticity} to the midpoint rule discretization, \eqref{eq:harmonic_oscillator_euler_midpoint}, used in the same example we obtain:
\[
	\ederiv q_{n+1}\wedge \ederiv p_{n+1} = \frac{4-\Delta t^{2}}{4+\Delta t ^{2}}\ederiv q_{n}\wedge \ederiv p_{n} + \frac{4\Delta t^{2}}{4+\Delta t ^{2}}\ederiv q_{n}\wedge \ederiv p_{n} = \ederiv q_{n}\wedge \ederiv p_{n}\,,
\]
confirming the symplecticity of the midpoint rule.

Since Hamiltonian systems are fundamental in the study of many physical systems, several discrete time integrators have been developed. As said before, energy conservation and symplecticity are key properties of Hamiltonian dynamics to satisfy at a discrete level. Ideally, a numerical integrator with discrete flow map $\tilde{\varphi}_{\Delta t}$ such that $\boldsymbol{y}_{n+1} = \tilde{\varphi}_{\Delta t}\left(\boldsymbol{y}_{n} \right)$ should be symplectic and exact energy preserving. Although highly desirable, it is very difficult to have a numerical integrator that preserves all geometrical properties. In this work we take as starting point the mimetic framework \cite{Kreeft2011,Palha2014} and show that two numerical integrators of arbitrary order, one symplectic and the other exactly energy preserving, can be derived.

\subsection{Literature review}

Structure-preserving methods that aim to preserve at a discrete level the geometrical, topological and homological properties of the systems of equations being solved have gained increased popularity in the last decades, see \cite{christiansen2011} for a brief overview. Two main areas of development have emerged: one focussing on spatial discretization (mimetic/compatible discretization) and another focussing on time discretization (geometric integration).

Regarding spatial discretisations, mimetic/compatible discretizations have been proposed by several authors for the numerical solution of partial differential equations. Tonti's work \cite{tonti1975formal} was one of the first to establish the relation between differential geometry and algebraic topology in physical theories and its implications on the development of structure preserving discretizations. Tonti employed differential forms and cochains as the building blocks of his method. The relation between differential forms and cochains was established by the Whitney map ($k$-cochains $\rightarrow$ $k$-forms) and the de Rham map ($k$-forms $\rightarrow$ $k$-cochains). These operators were initially introduced by Whitney in \cite{Whitney57} and de Rham in \cite{derham1955}, respectively. The linear interpolation of cochains to differential forms on a triangular grid was established in \cite{Whitney57} employing what is now known as {\em Whitney forms}. On quadrilaterals, Robidoux, \cite{robidoux-polynomial}, and Gerritsma, \cite{gerritsma::edge_basis}, derived the edge basis functions, capable of arbitrary order interpolation and histopolation of differential forms. Subsequently, among others, Robidoux, \cite{RobidouxThesis,RobidouxAdjointGradients1996}, Hyman, \cite{HymanShashkovSteinberg97,HymanShashkovSteinberg2002,HYmanSteinberg2004}, Steinberg, \cite{Steinberg1996,SteibergZingano2009}, Shashkov, \cite{bookShashkov}, Brezzi, \cite{BrezziBuffaLipnikov2009,brezzi2010}, Desbrun, \cite{desbrun2005discrete,ElcottTongetal2007,MullenCraneetal2009}, Lipnikov, \cite{Lipnikov2014}, Perot, \cite{Perot2000}, and Pavlov, \cite{PavlovMullenetal2010}, have proposed approaches in a finite difference/volume context. We highlight the `Japanese papers' by Bossavit, \cite{bossavit:japanese_01,bossavit:japanese_02}, which serve as an excellent introduction and motivation for the use of differential forms in the description of physics and its use in numerical modelling. In a series of papers by Arnold, Falk and Winther, \cite{arnold:Quads,arnold2006finite,arnold2010finite}, a \emph{finite element exterior calculus} framework is developed. Bochev, \cite{bochev2006principles}, presents general principles of mimetic discretizations, common to finite element, finite volume and finite difference discretizations. Finally, Kreeft et al., \cite{Kreeft2011}, and Palha et al., \cite{Palha2014}, introduced a framework merging the ideas of Tonti and Bossavit, resembling a hybrid finite element/finite volume formulation of arbitrary order on curvilinear elements.

On the topic of time discretization, namely the numerical integration of Hamiltonian systems, two main paths have emerged: \textit{symplectic} and \textit{exact energy-conserving} integrators. For a detailed overview of these methods the reader is referred to \cite{Hairer2006}. Symplectic integrators are widely known and can be traced back to the work of Vogelaere \cite{DeVogelaere1956} and have been extended to Runge-Kutta methods by several authors, for example  \cite{Feng1986, Suris1989, Sanz-Serna1992}. Exact energy-preserving methods can be found in the works of Gonzalez (for Hamiltonian systems with symmetry), \cite{Gonzalez1996}, Quispel et al., \cite{quispel1996}, and McLachlan et al., \cite{Mclachlan1999}, (for general Hamiltonians). This class of methods, named \emph{discrete gradient methods}, defines a discrete counterpart of the gradient operator ensuring that energy conservation is guaranteed in each time step. \emph{Projection methods} are an alternative to discrete gradient methods.  Hairer makes an extensive introduction to  projection methods in \cite{Hairer1996} and later extends this method to include time reversibility in \cite{Hairer2000}. Unlike discrete gradient and projection methods, the \emph{averaged vector field method} (AVF method) introduced by Quispel, \cite{Quispel2008}, does not require the knowledge of the functional form of the conserved integral, requiring only information on the vector field. An extension of AVF methods to Runge-Kutta methods is presented in \cite{Celledoni2009}. Subsequently, a modification of collocation methods has been introduced in \cite{hairerEnergyPreservingCollocation2010}, extending the AVF method to arbitrary order. Following a different route, the \emph{time finite element method}, initially introduced in \cite{Argyris1969,Fried1969,hulme1972}, was solved by a mixed formulation in \cite{Borri1991} and extended to satisfy energy conservation in \cite{Betsch2000}. More recently, \textit{line integral methods}, \cite{Brugnano2012}, (of which \textit{
Hamiltonian boundary value methods}, \cite{Brugnano2012, Brugnano2014}, are a subclass)  have been introduced. These methods exactly preserve the energy of arbitrary order polynomial Hamiltonian systems.

With respect to discretization methods that combine both space and time it is relevant to mention \emph{Box schemes} or \emph{Keller box schemes}, which are a class of numerical methods originally introduced by Wendroff, \cite{wendroff1960}, for hyperbolic problems and later popularized by Keller, \cite{keller1971,keller1978}, for parabolic problems. These methods are face-based and space and time are coupled by introducing a space-time control volume. They are known to be physically accurate and several successful applications have been reported, \cite{Croisille2002,Croisille2005,Gustafsson2006,Ranjan2013}. More recently, this method has been shown to be multisymplectic (Ascher, \cite{Ascher2005}, and Frank, \cite{Frank2006}). Also, Perot, \cite{Perot2007}, established a relation between box schemes and discrete calculus and generalized it to arbitrary meshes.

\subsection{Outline}

Despite the increasing popularity and success of mimetic/compatible methods their extension to the solution of ODEs has received little attention and mainly concerns time-marching schemes for PDEs, \cite{MullenCraneetal2009,christiansen2011}. Therefore, in this work we present how the mimetic framework, \cite{Kreeft2011,Palha2014}, can be applied to the numerical solution of systems of ODEs. As seen in \cite{Kreeft2011,Palha2014}, all approximations lie in the discretisation of the Hodge-$\star$ operator, which may be defined in different ways. In this work we show that a discrete Hodge-$\star$ obtained from a canonical Hodge-$\star$ operator results in a symplectic time integrator, whereas a Galerkin Hodge-$\star$ operator gives rise to an exact energy preserving integrator.

The outline of the paper is as follows. In \secref{section:brief_introduction_differential_geometry} a brief introduction to differential geometry is presented. The focus will be on vector fields, differential forms and the operators that act upon them. In \secref{section:brief_introduction_algebraic_topology} we introduce the required elements of algebraic topology, showing a clear parallel structure to differential geometry. In \secref{section:bridging_continuous_discrete} we establish a connection between the continuous world and the discrete by introducing the mimetic projection operator and the reconstruction of $k$-forms. In Section \ref{sec:time_integrators} the mimetic integrators and their properties are discussed and  in Section \ref{section::numerical_results} the time integrators are applied to test cases. Finally, in Section \ref{Section::Conclusions} the summary of this work and further applications are discussed.


\section{A brief introduction to differential geometry}\label{section:brief_introduction_differential_geometry}
	As mentioned previously, the objective of this work is to apply the mimetic framework developed in \cite{Kreeft2011,Palha2014} to the solution of systems of ordinary differential equations. The mimetic framework is based on differential geometry for the representation of the physical quantities and the operators that act on them. For this reason, a brief introduction to the key concepts needed to understand this approach will be given. For a more detailed presentation of this topic see, for example, \cite{abraham_diff_geom,burke1985applied,flanders::diff_forms,frankel}. Since time, $t$, is the only independent variable in the problems dealt with in this work, we give special attention to manifolds of dimension $n=1$. For this reason, each new concept is first introduced for manifolds of arbitrary dimension and then the particular case of $n=1$ is presented. Throughout the text we will consider Riemannian manifolds, $\manifold{N}$, of arbitrary dimension $n$ and temporal Riemannian manifolds, $\manifold{T}$, of dimension one.
	
	Consider a smooth curve $\gamma(s)$ in $\manifold{N}$ parametrized by $s\in [-\sigma, \sigma]$, $\sigma > 0$, such that $\gamma(0) = P\in\manifold{N}$. The derivative $\dot{\gamma}(0)$ is a \emph{tangent vector} at the point $P\in\manifold{N}$. For $n$-dimensional manifolds it is possible to find a set of $n$ linearly independent tangent vectors at the point $P$, $\{\left.\vec{e}_{1}\right|_{P},\dots,\left.\vec{e}_{n}\right|_{P}\}$, see for example \cite{abraham_diff_geom,burke1985applied}. This set spans a linear vector space called \emph{tangent space} at the point $P$ and denoted by $T_{P}\manifold{N}$. Any vector at $P$ can be written as a unique linear combination of these basis vectors, i.e.,
	\begin{equation}
		\left.\vec{v}\,\right|_{P} = \sum_{i}^{n}v^{i}\left.\vec{e}_{i}\right|_{P}\,. \label{eq:tangent_vector_general}
	\end{equation}
	Where the coefficients $v^{i}$ are associated to the particular basis elements $\left.\vec{e}_{i}\right|_{P}$. If in the vicinity of the point $P$ we define a local \emph{coordinate system}, $\{x^{1},\dots,x^{n}\}$, we specify the basis of the tangent space and in this case the basis vectors are generally denoted by $\left.\frac{\partial}{\partial x^{i}}\right|_{P}$ and called \emph{coordinate basis vectors}. We can then write \eqref{eq:tangent_vector_general} for a local coordinate system as:
	\[
		\left.\vec{v}\,\right|_{P} = \sum_{i}^{n}v^{i}\left.\frac{\partial}{\partial x^{i}}\right|_{P}\,.
	\]
	It is possible to smoothly define a tangent vector at each point $P\in\manifold{N}$. This construction generates \emph{vector fields}. To simplify the notation, in what follows we will suppress the explicit reference to the point $P$.
	
	\begin{framed}
	\noindent For a temporal one dimensional Riemannian manifold, $\manifold{T}$, the tangent space at each time instant $\tau$, $T_{\tau}\manifold{T}$, has dimension one and can therefore be spanned by a single coordinate basis vector $\{\frac{\partial}{\partial t}\}$. Any tangent vector, $\vec{v}\in T_{\tau}\manifold{T}$, can be represented as:
	\[
		\vec{v} = v\,\frac{\partial}{\partial t}\;.
	\]
	where the coefficient $v$ is associated to the coordinate basis.
	\end{framed}
	
	For a Riemannian manifold $\manifold{N}$, there exists a \emph{metric} $g_{P}$, or simply $g$, that assigns to every point $P\in\manifold{N}$ an inner product $g = \left(\cdot,\cdot\right)_{P}$ on $T_{P}\manifold{N}$ which depends smoothly on $P$, i.e.
	\[
		g: T_{P}\manifold{N}\times T_{P}\manifold{N}\rightarrow \mathbb{R}\,. \label{eq:metric_definition}
	\] 
	For the coordinate basis vectors we define the \emph{metric coefficients} as:
	\begin{equation}
		g_{ij} := \left(\frac{\partial}{\partial x^{i}},\frac{\partial}{\partial x^{j}}\right)\,. \label{eq:metric_coefficients}
	\end{equation}
	Due to the bilinearity, the inner product between any two real tangent vectors $\vec{v}=\sum_{i}v^{i}\frac{\partial}{\partial x^{i}}$ and $\vec{u}=\sum_{i}u^{i}\frac{\partial}{\partial x^{i}}$ can be written as:
	\begin{align}
		\left(\vec{v},\vec{u}\right) &= \left(\sum_{i}v^{i}\frac{\partial}{\partial x^{i}},\sum_{j}u^{j}\frac{\partial}{\partial x^{j}}\right)  \\
	 	                                        &=\sum_{i,j}v^{i}u^{j}\left(\frac{\partial}{\partial x^{i}},\frac{\partial}{\partial x^{j}}\right) \\
		                                        &= \sum_{i,j}v^{i}u^{j} g_{ij}\,.
	\end{align}
	
	\begin{framed}
		\noindent If we consider a one dimensional Riemannian manifold, $\manifold{T}$, there exists only one metric coefficient:
		\begin{equation}
			g_{11} :=  \left(\frac{\ederiv}{\ederiv t},\frac{\ederiv}{\ederiv t}\right)\,. \label{eq:metric_coefficient_1d}
		\end{equation}
		To simplify the notation, instead of $g_{11}$ we will use $g$ to denote the metric coefficient for a one dimensional manifold.
	\end{framed}
	
	With any linear vector space $V$ we can associate the dual space $V^{*}$ of linear functionals acting on $V$,
	\[
		\forall \alpha \in V^{*}, \quad \alpha:V\rightarrow\mathbb{R}\,.
	\]
	In the same way, with $T_{P}\manifold{N}$ we can associate the dual space $T^{*}_{P}\manifold{N}$ of linear functionals acting on the tangent space:
	\[
	\alpha^{(1)}: T_{P}\manifold{N}\rightarrow \mathbb{R}.
	\]
	We say that $\alpha^{(1)}\in T^{\star}_{P}\manifold{N}$, where $T^{\star}_{P}\manifold{N}$ is the \emph{cotangent space}.  The elements of the cotangent space are called \emph{covectors} or alternatively \emph{exterior $1$-forms}. As was done for the tangent space, for a particular local coordinate system we can define a basis $\{\ederiv x^{1},\dots,\ederiv x^{n}\}$ and represent these exterior 1-forms as:	
	\[
		\alpha^{(1)} = \alpha_{1} \ederiv x^{1} + \dots + \alpha_{n}\ederiv x^{n}\,.
	\]
	These basis $1$-forms are such that:
	\[
		\ederiv x^{i}\left(\frac{\partial}{\partial x^{j}}\right) := \delta^{i}_{j}\,,
	\]
	where $\delta^{i}_{j}$ is the Kronecker-$\delta$.
	
	A \emph{differential 1-form}, $\alpha^{1}$, is a smooth assignment of an exterior 1-form,
	\[
		\alpha^{(1)}(x_{1},\dots,x_{n}) = \alpha_{1}(x_{1},\dots,x_{n}) \ederiv x^{1} + \dots + \alpha_{n}(x_{1},\dots,x_{n}) \ederiv x^{n}\,\in T^{\star}_{P}\manifold{N}\,, 
	\]
	to each point, $P\in\manifold{N}$. We write $\alpha^{(1)}\in\Lambda^{1}(\manifold{N})$ or simply $\alpha^{(1)}\in\Lambda^{1}$.
	
	\begin{framed}
	\noindent For a one dimensional manifold $\tau$, with $T_{t}\manifold{T}$ we can associate the dual space $T^{\star}_{t}\manifold{T}$ of linear functionals acting on the tangent space. In this case the basis consists of only one element $\{\ederiv t\}$ and the exterior 1-forms are represented as:
	\begin{equation}
		\alpha^{(1)} = \alpha\,\mathrm{d}t\;, \label{eq:1_form_1d}
	\end{equation}
	where $\alpha$ is the $t$-component of the exterior 1-form. In this coordinate system we have that:
	\[
		\ederiv t \left(\frac{\ederiv}{\ederiv t}\right) = 1\,.
	\]
	And therefore:
	\[
		\alpha^{(1)}\left(\vec{v}\right) = \left(\alpha\ederiv t\right)\left(v \frac{\ederiv}{\ederiv t}\right) = \alpha v\,.
	\]
	\end{framed}

The \emph{exterior product} or \emph{wedge product}, $\wedge$, allows the construction of higher rank $k$-forms, $1<k\leq n$, from 1-forms. Let $\manifold{N}$ be a manifold of dimension $n$, $\Lambda^{k}\left(\manifold{N}\right)$ and $\Lambda^{l}\left(\manifold{N}\right)$ be the space of $k$-forms and $l$-forms, respectively, with $k+l\leq n$, then the wedge product, $\wedge$, is a mapping:

\[
	\wedge : \Lambda^{k}\left(\manifold{N}\right)\times\Lambda^{l}\left(\manifold{N}\right)\rightarrow \Lambda^{k+l}\left(\manifold{N}\right), \quad k+l\leq n
\]
which satisfies the following properties:

\begin{subequations}
	\begin{align}
		&\left(\alpha^{\left(k\right)}+\beta^{\left(l\right)}\right)\wedge\gamma^{\left(m\right)} = \alpha^{\left(k\right)}\wedge\gamma^{\left(m\right)} + \beta^{\left(l\right)}\wedge\gamma^{\left(m\right)} &\quad \text{(Distributivity)}\\
		&\left(\alpha^{\left(k\right)}\wedge\beta^{\left(l\right)}\right)\wedge\gamma^{\left(m\right)} = \alpha^{\left(k\right)}\wedge\left(\beta^{\left(l\right)}\wedge\gamma^{\left(m\right)}\right) = \alpha^{\left(k\right)}\wedge\beta^{\left(l\right)}\wedge\gamma^{\left(m\right)} & \quad \text{(Associativity)} \\
		& a\alpha^{(k)}\wedge\beta^{(l)} = \alpha^{(k)}\wedge a\beta^{(l)} = a\left(\alpha^{(k)}\wedge\beta^{(l)}\right) & \quad \text{(Multiplication by functions)} \\
		& \alpha^{(k)}\wedge\beta^{(l)} = \left(-1\right)^{kl}\beta^{(l)}\wedge\alpha^{(k)} & \quad \text{(Skew symmetry)} 
	\end{align}
\end{subequations}
where $\alpha^{(k)}\in\Lambda^{(k)}\left(\manifold{N}\right)$, $\beta^{(l)}\in\Lambda^{(l)}\left(\manifold{N}\right)$ and $\gamma^{(m)}\in\Lambda^{(m)}\left(\manifold{N}\right)$. A differential $k$-form, $\alpha^{(k)}\in\Lambda^{k}$, is then an alternating $k$-tensor:
\[
	\alpha^{(k)}:\underbrace{T_{P}\manifold{N}\times\dots\times T_{P}\manifold{N}}_{k \text{ times }}\rightarrow\mathbb{R}\,,
\]
and we can write:
\begin{equation}
	\alpha^{(k)} = \sum_{i_{1},\dots, i_{k}}\alpha_{i_{1},\dots, i_{k}}(x_{1},\dots,x_{n})\,\ederiv x^{i_{1}}\wedge \dots \wedge \ederiv x^{i_{k}}\,, \quad 1\leq i_{1} < \dots < i_{k}\leq n\,. \label{eq:expression_diff_form}
\end{equation}
The space of differential 0-forms, $\Lambda^{(0)}(\manifold{N})$, is simply the space of smooth functions and we write:
\begin{equation}
	\beta^{(0)}(x_{1},\dots,x_{n}) = \beta(x_{1},\dots,x_{n})\,. \label{eq:0_form}
\end{equation}

\begin{framed}
\noindent For the particular case $n=1$ the wedge product takes the simpler form:
\[
	\alpha^{(1)} \wedge \beta^{(0)} = \beta^{(0)} \wedge \alpha^{(1)} = \alpha \beta\,\ederiv t\,, \quad \beta^{(0)}\wedge\gamma^{(0)} = \gamma^{(0)}\wedge\beta^{(0)} = (\gamma\beta)^{(0)} \quad\text{and}\quad \alpha^{(1)} \wedge \sigma^{(1)} = \sigma^{(1)} \wedge \alpha^{(1)} = 0\,,
\]
where $\beta^{(0)},\gamma^{(0)}\in\Lambda^{0}(\manifold{T})$ and $\alpha^{(1)},\sigma^{(1)}\in\Lambda^{1}(\manifold{T})$. 
\end{framed}

Differential $k$-forms naturally integrate over $k$-dimensional manifolds. Let $\alpha^{(k)}\in\Lambda^{k}\left(\manifold{N}\right)$ and $\manifold{N}_{k}\subset\manifold{N}\subset\mathbb{R}^{n}$, with $k=\mathrm{dim}\left(\manifold{N}_{k}\right)$, $n=\mathrm{dim}\left(\manifold{N}\right)$ and $k\leq n$,

\begin{equation}
	\langle\alpha^{(k)},\manifold{N}_{k}\rangle:=\int_{\manifold{N}_{k}}\alpha^{(k)},\label{eq:duality_pairing}
\end{equation}
this represents a metric-free duality pairing. 

\begin{framed}
\noindent For the case $n=1$ considered here we have:
\[
	\langle\beta^{(0)},\manifold{T}_{0}\rangle:=\int_{\manifold{T}_{0}}\beta^{(0)}\quad\text{and}\quad\langle\alpha^{(1)},\manifold{T}_{1}\rangle:=\int_{\manifold{T}_{1}}\alpha^{(1)}\,.\label{eq:duality_pairing_1d}
\]
Note that $\manifold{T}_{0}$ is a 0-dimensional manifold, therefore a point. By definition, the integral of a scalar function over a point is just the evaluation of the function at the point: $\int_{\manifold{T}_{0}}\beta^{(0)} = \beta(\manifold{T}_{0})$.
\end{framed}

%
%
%

For Riemannian manifolds, $\manifold{N}$, we can define a point-wise inner product of $k$-forms, $\left(\cdot,\cdot\right)$:
\[
	\left(\cdot,\cdot\right):\Lambda^{k}\left(\manifold{N}\right)\times\Lambda^{k}\left(\manifold{N}\right)\rightarrow\Lambda^{0}\left(\manifold{N}\right)\,.
\]
If $\alpha^{(k)}$ and $\beta^{(k)}$ are two $k$-forms expressed in the form \eqref{eq:expression_diff_form}, then the point-wise inner product is given by:
\[
	\left(\alpha^{(k)},\beta^{(k)}\right) = g^{i_{1}j_{1}}\dots g^{i_{k}j_{k}}\alpha_{i_{1},\dots,i_{k}}\beta_{j_{1},\dots,j_{k}}\,.
\]
Where $g^{ij}:=(\ederiv x^{i},\ederiv x^{j})$ are the components of the inverse of the metric tensor, \eqref{eq:metric_coefficients}.

\begin{framed}
\noindent For 1-dimensional manifolds, $\manifold{T}$, the point-wise inner products between 0-forms, $\alpha^{(0)},\beta^{(0)}\in\Lambda^{0}(\manifold{T})$, written in the form \eqref{eq:0_form} is given by
	\begin{equation}
		\left(\alpha^{(0)},\beta^{(0)}\right) = \alpha \, \beta \;,\label{eq::definition_point_wise_inner_product_zero_forms}
	\end{equation}
	and the inner product between 1-forms, $\alpha^{(1)},\beta^{(1)}\in\Lambda^{1}(\manifold{T})$, written in the form \eqref{eq:1_form_1d} is given by
	\begin{equation}
		\left(\alpha^{(1)},\beta^{(1)}\right) = \alpha \, \beta \, \frac{1}{g}\;. \label{eq::definition_point_wise_inner_product_one_forms}
	\end{equation}
	Note that we used $g:=g_{11}$, with $g_{11}$ as in \eqref{eq:metric_coefficient_1d}. Therefore $g^{11} = \frac{1}{g_{11}}=\frac{1}{g}$.	
\end{framed}

The Hodge-$\star$ operator is a mapping:
\[
	\star:\Lambda^{k}(\manifold{N})\rightarrow\Lambda^{n-1}(\manifold{N})\,,
\]
 that is induced by a combination of the wedge product and the inner product,
 \begin{equation}
 	\alpha^{(k)}\wedge\star\beta^{(k)} := \left(\alpha^{(k)},\beta^{(k)}\right)\sigma^{(n)}\,. \label{eq:hodge_star}
 \end{equation}
Where $\sigma^{(n)}$ is the $n$-form, called \emph{volume form}, such that $\int_{\manifold{N}}\sigma^{(n)}$ is the volume of the manifold $\manifold{N}$ and $\star 1 = \sigma^{(n)}$.

The $L^{2}$ inner product, $\left(\cdot,\cdot\right)_{L^{2}(\manifold{N})}$, is defined as a mapping,
\[
	\left(\cdot,\cdot\right)_{L^{2}(\manifold{N})}: \Lambda^{k}(\manifold{N})\times\Lambda^{k}(\manifold{N}) \rightarrow\mathbb{R}\,
\]
such that:
\begin{equation}
	\left(\alpha^{(k)},\beta^{(k)}\right)_{L^{2}(\manifold{N})} := \int_{\manifold{N}}(\alpha^{(k)},\beta^{(k)})\,\sigma^{(n)} = \int_{\manifold{N}}\alpha^{(k)}\wedge\star\beta^{(k)}\,. \label{eq:L2_inner_product}
\end{equation}

\begin{framed}
	\noindent The Hodge-$\star$ operator, for the case of one dimensional Riemannian manifolds, $\manifold{T}$, takes the form:
	\begin{equation}
		\star 1 = \sqrt{g}\,\ederiv t \quad\text{and}\quad \star\ederiv t = \frac{1}{\sqrt{g}}\,. \label{eq:hodge_star_1d}
	\end{equation}
	The $L^{2}$ inner product between 0-forms, $\alpha^{(0)},\beta^{(0)}\in\Lambda^{0}(\manifold{T})$, written in the form \eqref{eq:0_form} becomes
	\begin{equation}
		\left(\alpha^{(0)},\beta^{(0)}\right)_{L^{2}(\manifold{N})} = \int_{\manifold{T}}\alpha\beta\sqrt{g}\,\ederiv t\,, \label{eq:L2_1d_0form}
	\end{equation}
	and the  $L^{2}$ inner product between 1-forms, $\alpha^{(1)},\beta^{(1)}\in\Lambda^{1}(\manifold{T})$, written in the form \eqref{eq:1_form_1d} is
	\begin{equation}
		\left(\alpha^{(1)},\beta^{(1)}\right)_{L^{2}(\manifold{N})} = \int_{\manifold{T}} \frac{\alpha\beta}{\sqrt{g}}\ederiv t\,. \label{eq:L2_1d_1form}
	\end{equation}
\end{framed}

Differentiation in differential geometry is represented by the \emph{exterior derivative}, $\ederiv$, which is defined as a mapping,
\[
	\ederiv : \Lambda^{k}(\manifold{N})\rightarrow\Lambda^{k+1}(\manifold{N})\,,
\]
that satisfies the generalized Stokes Theorem: 
\begin{equation}
	\int_{\manifold{N}}\ederiv\alpha^{(k)} = \int_{\partial\manifold{N}}\alpha^{(k)}\;. \label{eq:stokes_theorem}
\end{equation}
Where $\manifold{N}$ is a $(k+1)$-dimensional manifold and $\partial\manifold{N}$ is its $k$-dimensional boundary manifold, and $\alpha^{(k)}\in\Lambda^{k}(\partial\manifold{N})$. This equation contains, as special cases, the Newton-Leibniz, Stokes and Gauss theorems and hence generalizes the vector calculus $\nabla$, $\nabla\times$ and $\nabla\cdot$ operators to arbitrary manifolds. Moreover, this operator is completely topological (does not depend on the metric) and satisfies the Leibniz rule $\ederiv(\alpha^{(k)}\wedge\beta^{(l)}) = \ederiv\alpha^{(k)}\wedge\beta^{(l)}+(-1)^{k}\alpha^{(k)}\wedge\ederiv\beta^{(l)}$. An important point to note is that this theorem can be expressed using the duality pairing, \eqref{eq:duality_pairing}, as:
\begin{equation}
		\langle\mathrm{d}\alpha^{(k)},\manifold{N}\rangle = \langle\alpha^{(k)},\partial\manifold{N}\rangle\;. \label{eq:stokes_theorem_duality_pairing}
\end{equation}
This duality property will play a fundamental role in the construction of the \emph{discrete exterior derivative}.

\begin{framed}
\noindent For $1$-dimensional manifolds, the \emph{exterior derivative}, $\mathrm{d}$, is a mapping $\mathrm{d}:\Lambda^{0}\mapsto\Lambda^{1}$, such that on a coordinate system it is defined as:
\begin{equation}
	\mathrm{d} \alpha^{(0)} := \frac{\mathrm{d}\alpha}{\mathrm{d}t}\,\mathrm{d}t,\quad \forall \alpha^{(0)}\in\Lambda^{(0)}(\manifold{T})\;.\label{eq::def_exterior_derivative}
\end{equation}
In this case, since $n=1$, we have that $\mathrm{d}\alpha^{(1)} = 0$, $\forall\alpha^{(1)}\in\Lambda^{1}(\manifold{T})$. Higher order differentiation can be defined by combining the exterior derivative and the Hodge-$\star$ operators, for example:
\[
	\ederiv\star\ederiv \alpha^{(0)} = \ederiv\star \frac{\ederiv \alpha}{\ederiv t}\,\ederiv t = \ederiv \frac{1}{\sqrt{g}}\frac{\ederiv \alpha}{\ederiv t} =  \frac{\ederiv}{\ederiv t} \frac{1}{\sqrt{g}}\frac{\ederiv \alpha}{\ederiv t} \,\ederiv t\,.
\]
	
In this way one can construct a double de Rham complex associated to differential forms in 1-dimensional Riemannian manifolds, \eqref{eq::double_de_rham_complex}:
	\begin{equation}
		\includegraphics{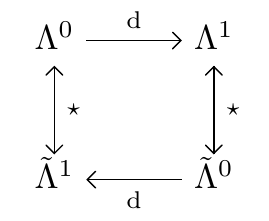} \label{eq::double_de_rham_complex}
	\end{equation}

\end{framed}
	
	There are two fundamental results to extract from the differential geometric formulation. First, from the generalized Stokes' theorem in one dimension, it is possible to recognize that 1-forms are intrinsically related to integrals over time intervals whereas 0-forms are associated  to evaluations at time instants. Moreover exact differentiation is possible if the points where 0-forms are evaluated are the boundary of the line segments where 1-forms are integrated. Second, from the double de Rham complex, \eqref{eq::double_de_rham_complex}, one can see that there exist two distinct complexes, the top one and the bottom one, both being related to each other by the Hodge-$\star$ operator. It states, for example, that $\alpha^{(1)}$ is associated to time intervals and that $\star\alpha^{(1)}$ is associated to time instants. Conventionally, these two quantities are considered the same quantity and are discretized in the same way. It is the authors' opinion that this formal distinction is essential in the construction of a numerical discretization since, at the discrete level, the Hodge-$\star$ cannot be performed exactly. Therefore one needs to make use of dual grids to represent this dual complex. In this way all approximations lie in the discretisation of the Hodge-$\star$, see \cite{Kreeft2011,Palha2014,desbrun2005discrete,perot43discrete,tonti1975formal,mattiussi2000finite}.

\section{A brief introduction to algebraic topology} \label{section:brief_introduction_algebraic_topology}
	Following the mimetic framework presented in \cite{Kreeft2011,Palha2014}, consider a temporal 1-dimensional domain, $\manifold{T}$, and its \emph{primal}, $D$, and \emph{dual}, $\tilde{D}$, grids, see \figref{fig:primal_dual_grid}.

\begin{figure}[!ht]
	\centering
	\includegraphics{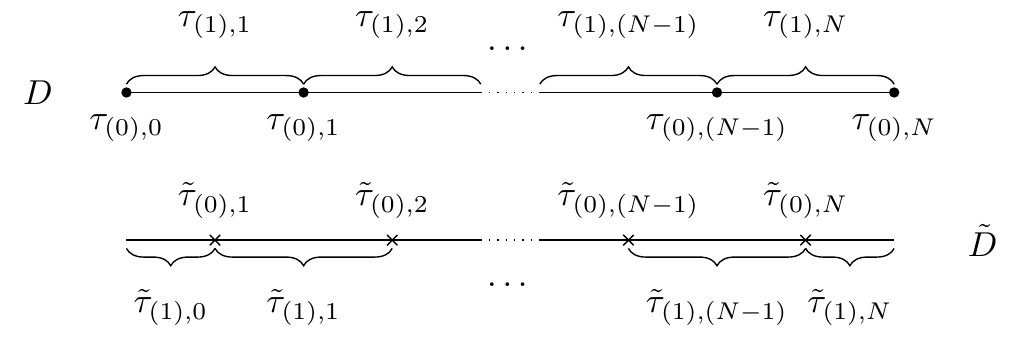}
	\caption{Example of a primal, $D$, and dual, $\tilde{D}$, grids covering a 1-dimensional manifold. Notice that the number of nodes, $\tau_{(0),i}$, in the primal mesh is equal to the number of edges (line segments), $\tilde{\tau}_{(1),i}$, in the dual mesh, and vice-versa.}
	\label{fig:primal_dual_grid}
\end{figure}

	These grids consist not only of the points,
	\[
		\left\{\tau_{(0),i}: i=0,\dots,N\right\} \qquad \text{and} \qquad \left\{\tilde{\tau}_{(0),i}: i=1,\dots,N\right\},
	\]
	as is common in many numerical discretizations, but also of the edges (line segments) connecting them,
	\[
		\left\{\tau_{(1),i}: i=1,\dots,N\right\} \qquad \text{and} \qquad \left\{\tilde{\tau}_{(1),i}: i=0,\dots,N\right\}.
	\]
	One important point to note is that the number of nodes, $\tau_{(0),i}$, in the primal mesh is identical to the number of edges, $\tilde{\tau}_{(1),i}$, in the dual mesh, and vice-versa.
	
	The $k$-dimensional objects in $D$ and $\tilde{D}$ are called \emph{$k$-cells} and we represent them by $\tau_{(k),i}$ where $k$ denotes the dimension of the object (in our case $k=0,1$) and $i$ is a label that distinguishes different objects. 
	
	It is possible to \emph{inner orient} a time interval (a line) by defining a direction that goes from the preceding time instant to the following one (as in the natural time sequence), see top left of the \emph{Inner orientation} block in \figref{fig:orientation}. One can also inner orient a time interval in the opposite direction (as in time reversal), see top right of the \emph{Inner orientation} block in \figref{fig:orientation}. Time instants can be inner oriented as either sources (see bottom left of the \emph{Inner orientation} block in \figref{fig:orientation}) or sinks (see bootom right of the \emph{Inner orientation} block in \figref{fig:orientation}). The \emph{outer orientation} is induced by inner orientation, see the \emph{Outer orientation} block in \figref{fig:orientation}. For a more detailed discussion of orientation see \cite{tonti1975formal,bossavit:japanese_01,Palha2014} and especially the more recent work by Tonti \cite{Tonti2014} where the particular case of time is addressed.
	
	If we endow each of the geometric objects with a default orientation then we call the grid an \emph{oriented grid}, see \figref{fig:orientation} for examples of how the geometric objects can be oriented.

	\begin{figure}[!ht]
	\centering
	\includegraphics{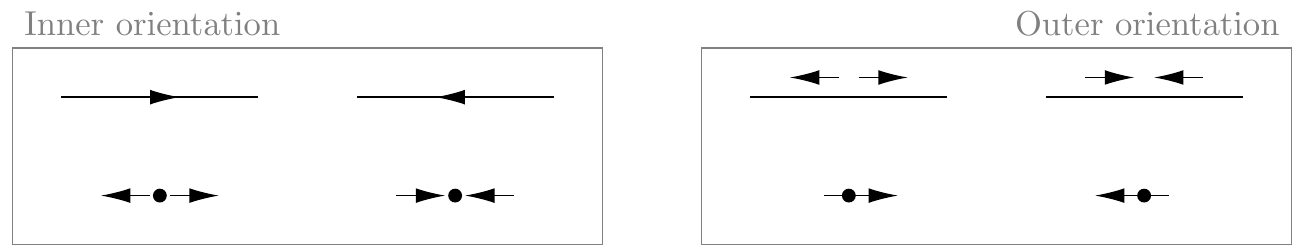}
	\caption{Example of possible orientations of points and edges. See \cite{tonti1975formal,bossavit:japanese_01,Palha2014} for a more detailed discussion of orientation.}
	\label{fig:orientation}
	\end{figure}

	Given a grid, e.g. $D$, the space of $k$-chains, $C_{k}(D)$, is the collection of weighted $k$-cells. A $k$-chain, $\boldsymbol{c}_{(k)}\in C_{k}(D)$, is a formal sum of $k$-cells, $\tau_{(k),i} \in D$,

	\[
		\boldsymbol{c}_{(k)} = \sum_{i}\tau_{(k),i}c^{i}.
	\]
	By formal sum we mean a collection of cells and weights, $\{\tau_{(k),i},c^{i}\}$, where the $c^{i}$ denote the weights.

	The \emph{boundary operator}, $\partial$, on $k$-chains, with $k\geq 1$, is a homomorphism, $\partial : C_{k}(D)\rightarrow C_{k-1}(D)$, such that:

	\[
		\partial \boldsymbol{c}_{(k)} = \partial\sum_{i}\tau_{(k),i}c^{i} := \sum_{i}\partial\tau_{(k),i}\,c^{i}.
	\]
	For 0-chains, $\partial \boldsymbol{c}_{(0)} = \boldsymbol{0}$.
	The boundary of a $k$-cell, $\tau_{(k),i}$, is a $(k-1)$-chain formed by its oriented faces. The coefficients of this $(k-1)$-chain are associated to each of the faces and are given by the orientations:
	\[
		\partial\tau_{(k),i} = \sum_{j}\tau_{(k-1),j}\,e^{j}_{i}\,,
	\]
	with:
	\[
		\begin{cases}
			e^{j}_{i} = 1, & \text{if the orientation of } \tau_{(k-1),j} \text{ equals the default orientation}\\
			e^{j}_{i} = -1, & \text{if the orientation of } \tau_{(k-1),j} \text{ is opposite to the default orientation}\\
			e^{j}_{i} = 0, & \text{if } \tau_{(k-1),j} \text{ is not a face of } \tau_{(k),i},
		\end{cases}
	\]
	and $\partial\partial \boldsymbol{c}_{(k)} = 0$.

	It is possible to establish an isomorphism between the $k$-chains and the corresponding column vectors of its coefficients and we use $c_{(k)}$ to represent the column vector of the coefficients of the $k$-chain $\boldsymbol{c}_{(k)}$, see \cite{Palha2014} for more details.

	\begin{example}\label{ex:chain_boundary}
	As an example of $k$-chains and the boundary operator, consider the grid represented in \figref{fig:chain_boundary_example} where the default orientation of the geometric objects is depicted. 
	\begin{figure}[!ht]
	\centering
	\includegraphics{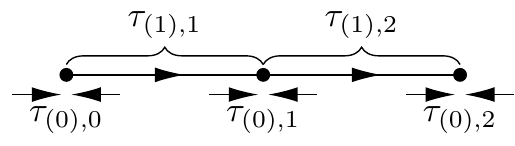}
	\caption{Example of a 1-dimensional grid, $D$, with the representation of the default orientation of the geometric objects.}
	\label{fig:chain_boundary_example}
	\end{figure}

	Then we can construct the following $1$-chains:
	\[
		\begin{array}{ccccc}
		\boldsymbol{c}_{(1)} = \tau_{(1),1} + \tau_{(1),2} & \longleftrightarrow & c_{(1)} = \left[\begin{array}{c} 1 \\ 1\end{array}\right] & \longleftrightarrow & \includegraphics{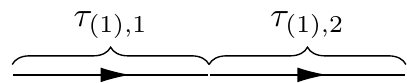}
		 \\
		\boldsymbol{a}_{(1)} = - \tau_{(1),2} & \longleftrightarrow & a_{(1)} = \left[\begin{array}{c} 0 \\ -1\end{array}\right] & \longleftrightarrow & \includegraphics{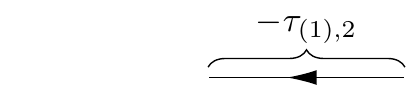}
		\end{array}\,,
	\]
	and the following $0$-chains:
	\[
		\begin{array}{ccccc}
		\boldsymbol{c}_{(0)} = \tau_{(0),0} + \tau_{(0),1} + \tau_{(0),2} & \longleftrightarrow & c_{(0)} = \left[\begin{array}{c} 1 \\ 1 \\ 1 \end{array}\right] & \longleftrightarrow & \includegraphics{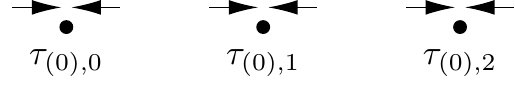}
		 \\
		\boldsymbol{a}_{(0)} = - \tau_{(0),1} & \longleftrightarrow & a_{(0)} = \left[\begin{array}{c} 0 \\ -1 \\ 0\end{array}\right] & \longleftrightarrow & \includegraphics{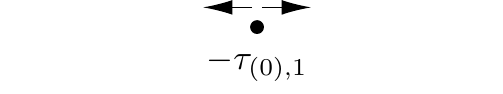}
		\end{array}\,,
	\]
	
	The boundary operator, $\partial$, applied to the 1-cochain, $\boldsymbol{c}_{(1)}$, results in:
	\[
		\begin{array}{ccc}
			{\setlength\arraycolsep{1pt}
			\begin{array}{rcl}
				\partial\boldsymbol{c}_{(1)} &=& \partial\tau_{(1),1} + \partial\tau_{(1),2} \\
				                             &=& -\tau_{(0),0} + \tau_{(0),1} - \tau_{(0),1} + \tau_{(0),2}\\
				                             &=& -\tau_{(0),0} + \tau_{(0),2}
			\end{array}}
			& \longleftrightarrow & 
				\partial c_{(1)} = \left[\begin{array}{cc} -1 & 0 \\ 1 & -1 \\ 0 & 1 \end{array}\right]
				\left[\begin{array}{c} 1 \\ 1\end{array}\right] =  \mathsf{E}_{(0,1)}\left[\begin{array}{c} 1 \\ 1\end{array}\right]
				=\left[\begin{array}{c} -1 \\ 0 \\ 1\end{array}\right]
		\end{array}\,,
	\]

	and to the 1-cochain $\boldsymbol{a}_{(1)}$ results in:
	\[
		\begin{array}{ccc}
		{\setlength\arraycolsep{1pt}
		\begin{array}{rcl}
			\partial\boldsymbol{a}_{(1)} &=& - \partial\tau_{(1),2} \\
			                             &=& -\left(-\tau_{(0),1} + \tau_{(0),2}\right) \\
			                             &=& \tau_{(0),1} - \tau_{(0),2}
		\end{array}}
		& \longleftrightarrow & \partial a_{(1)} = \left[\begin{array}{cc} -1 & 0 \\ 1 & -1 \\ 0 & 1 \end{array}\right]
				\left[\begin{array}{c} 0 \\ -1\end{array}\right] =  \mathsf{E}_{(0,1)}\left[\begin{array}{c} 0 \\ -1\end{array}\right] 
				=\left[\begin{array}{c} 0 \\ 1 \\ -1\end{array}\right]
		\end{array}\,.
	\]
	\end{example}

	\begin{remark}
		The matrices $\mathsf{E_{(0,1)}}$ present in \exampleref{ex:chain_boundary}, are an example of \emph{incidence matrices}, $\mathsf{E_{(k,k+1)}}$, which are matrix representations of the boundary operator.
	\end{remark}

	For a grid to be considered  a \emph{cell complex} we impose that if a $k$-dimensional geometric object, $\tau_{(k),i}$, is in the grid, its boundary, $\partial\tau_{(k),i}$, is also in the grid. We see that the primal grid, $D$ in \figref{fig:primal_dual_grid}, is a cell complex. The dual grid, $\tilde{D}$ in \figref{fig:primal_dual_grid}, is only a cell complex if boundary nodes are added. For a detailed discussion of this topic the reader is directed to \cite{Kreeft2011,Palha2014}.

	Dual to the space of $k$-chains, $C_{k}(D)$, is the space of \emph{$k$-cochains}, $C^{k}(D)$, defined as the set of homomorphisms, $\boldsymbol{c}^{(k)}:C_{k}\rightarrow \mathbb{R}$,

	\[
		\langle\boldsymbol{c}^{(k)},\boldsymbol{c}_{(k)}\rangle := \boldsymbol{c}^{(k)}\left(\boldsymbol{c}_{(k)}\right).
	\]

	With the duality pairing between cochains and chains, it is possible to define the formal adjoint of the boundary operator, the \emph{coboundary operator}, $\delta : C^{k}(D)\rightarrow C^{k+1}(D)$,
	\begin{equation}
		\langle\delta\boldsymbol{c}^{(k)},\boldsymbol{c}_{(k+1)}\rangle := \langle\boldsymbol{c}^{(k)},\partial\boldsymbol{c}_{(k+1)}\rangle .\label{eq:definition_coboundary_operator}
	\end{equation}
	This equation is the discrete analogue of the generalized Stokes' equation, \eqref{eq:stokes_theorem_duality_pairing}.

	Analogous to the exterior derivative, the coboundary operator is nilpotent: $\delta\delta\boldsymbol{c}^{(k)} = 0$ for all $\boldsymbol{c}^{(k)}\in C^{k}(D)$. This property follows directly from its definition, \eqref{eq:definition_coboundary_operator}, and from the nilpotency of the boundary operator $\partial\partial\boldsymbol{c}_{(k)} := 0$ for all $\boldsymbol{c}_{(k)}\in C_{k}(D)$.

	Let $C_{k}(D)$ be the space of $k$-chains with basis $\{\tau_{(k),j}\}$, then the basis of the dual space of $k$-cochains, $C_{k}(D)$, is given by $\{\tau^{(k),i}\}$, such that $\tau^{(k),i}\left(\tau_{(k),j}\right) = \delta^{i}_{j}$ (here $\delta^{i}_{j}$ are the coefficients of the Kronecker delta). In this way, all $k$-cochains can be represented as linear combinations of these basis elements,
	\[
		\boldsymbol{c}^{(k)} = \sum_{i}\tau^{(k),i} c_{i}.
	\]
	Similar to $k$-chains, there exists an isomorphism between $k$-cochains an the corresponding row vector of its coefficients and we use $c^{(k)}$ to represent the row vector of the coefficients of the $k$-cochain $\boldsymbol{c}^{(k)}$, see \cite{Palha2014} for more details. With this isomorphism that identifies $k$-chains and $k$-cochains with their coefficients we can establish a matrix representation for the coboundary operator, $\delta$:
	\begin{align}
		\langle\boldsymbol{c}^{(k)},\partial\boldsymbol{c}_{(k+1)}\rangle &= \sum_{i=1}^{\mathrm{rank}(C^{k}(D))} \sum_{j=1}^{\mathrm{rank}(C_{k+1}(D))} c_{i}\left(\mathsf{E}^{i,j}_{(k,k+1)}c^{j}\right) \label{eq:definition_incidence_matrix}\\
		&= \sum_{i=1}^{\mathrm{rank}(C^{k}(D))} \sum_{j=1}^{\mathrm{rank}(C_{k+1}(D))} \left(c_{i}\mathsf{E}^{i,j}_{(k,k+1)}\right)c^{j} \\
		&= \langle\delta\boldsymbol{c}^{(k)},\boldsymbol{c}_{(k+1)}\rangle .
	\end{align}
	Therefore, when the row vector $c^{(k)}$ contains the expansion coefficients, $c_{i}$, for the cochain $\boldsymbol{c}^{(k)}$, then the row vector $c^{(k)}\mathsf{E}_{(k,k+1)}$ contains the expansion coefficients for $\delta\boldsymbol{c}^{(k)}$.

\section{Bridging the continuous and the discrete} \label{section:bridging_continuous_discrete}
	In \secref{section:brief_introduction_differential_geometry} we have introduced the key concepts of differential geometry, integration over manifolds, differential operators and established fundamental integral relations. In \secref{section:brief_introduction_algebraic_topology} we have presented an analogous mathematical structure but at the discrete level, using algebraic topology. We now have two distinct worlds, one continuous and one discrete, with paralleled mathematical structures but still no connection between them. What we propose to do in this section is to explain how to convert continuous forms into discrete cochains and vice versa. In this way, we establish a bridge between the continuous and the discrete.

	\subsection{From continuous to discrete} \label{section:from_continuous_to_discrete}
		The \emph{reduction operator}, $\reduction:\Lambda^{k}\left(\manifold{N}\right)\rightarrow C^{k}(D)$, maps $k$-differential forms onto $k$-cochains by
		\begin{equation}
			\left.\reduction\alpha^{(k)}\right|_{\tau_{(k),i}} := \langle\reduction\alpha^{(k)},\tau_{(k),i}\rangle := \int_{\tau_{(k),i}}\alpha^{(k)} = \langle\alpha^{(k)},\tau_{(k),i}\rangle, \qquad \forall \tau_{(k),i}\in C_{k}(D). \label{eq:definition_reduction_operator}
		\end{equation}

		Therefore, for all $k$-chains, $\boldsymbol{c}_{(k)}\in C_{k}(D)$, the reduction of a $k$-form, $\alpha^{(k)}\in\Lambda^{k}(\manifold{N})$, to a $k$-cochain, $\boldsymbol{a}^{(k)}\in C^{k}(D)$, is
		\begin{equation}
			\boldsymbol{a}^{(k)}(\boldsymbol{c}_{(k)}) := \langle\reduction\alpha^{(k)},\boldsymbol{c}_{(k)}\rangle = \sum_{i}c^{i}\langle\reduction\alpha^{(k)},\tau_{(k),i}\rangle = \sum_{i}c^{i}\int_{\tau_{(k),i}}\alpha^{(k)} = \int_{\boldsymbol{c}_{(k)}}\alpha^{(k)}. \label{eq:reduction_applied_to_differential_form}
		\end{equation}

		By defining the reduction operator in this manner an important commuting property relating the exterior derivative and the coboundary operator is satisfied:
		\begin{equation}
			\reduction\ederiv = \delta\reduction . \label{eq:reduction_communting_derivative}
		\end{equation}

		Commuting relations are one of the key points of the mimetic framework, see \cite{Kreeft2011,Palha2014} for a more extensive discussion. In this case, this commuting relation states that taking the exterior derivative and then discretizing lead to the same result as discretizing first and then taking the discrete derivative.

	\subsection{From discrete to continuous} \label{section:from_discrete_to_continuous}
		The operator that establishes the relation in the opposite direction, mapping $k$-cochains into $k$-differential forms, is the \emph{reconstruction operator}, $\reconstruction: C^{k}(D)\rightarrow \Lambda_{h}^{k}(\manifold{N})$, where $\Lambda_{h}^{k}(\manifold{N})$ is the range of $\reconstruction$. This operator can be implemented in different ways, see \cite{Palha2014} for a more detailed discussion and section \secref{subsection:basis_forms} for its use in this work. Despite the different possibilities to define this operator, the following commuting property must be satisfied,
		\begin{equation}
			\ederiv\reconstruction = \reconstruction\delta .\label{eq:reconstruction_commuting_property}
		\end{equation}

		The reconstruction operator is essentially an interpolation operator which interpolates the discrete $k$-cochains to continuous $k$-forms. The commuting property \eqref{eq:reconstruction_commuting_property} states that reconstructing and taking the exterior derivative is identical to first taking the discrete derivative and then reconstructing. This condition poses a severe restriction on the admissible reconstruction forms. For triangular elements these reconstruction forms are known as Whitney forms, \cite{Whitney57,bossavit:japanese_01}. The polynomial reconstruction forms on quadrilaterals are known as edge forms, \cite{robidoux-polynomial,gerritsma::edge_basis}, and are the ones used in this work, see \secref{subsection:basis_forms}.

		Furthermore, the reconstruction operator, $\reconstruction$, must be the right inverse of the reduction operator, $\reduction$, so $\reduction\reconstruction = \text{Id}$ on $C^{k}(D)$ and it should approximate the left inverse of $\reduction$, so $\reconstruction\reduction = \text{Id} + \mathcal{O}(h^{p})$, where $p$ is denotes the order of the approximation, for smooth differential forms.

		The first property, $\reduction\reconstruction=\text{Id}$, is a \emph{consistency condition}. The second property, $\reconstruction\reduction=\text{Id} + \mathcal{O}(h^{p})$, is an \emph{approximability condition}.

	\subsection{Mimetic projection} \label{subsection:mimetic_projection}
		The \emph{mimetic projection operator} is obtained by combining the reduction and reconstruction operators:
		\begin{equation} 
			\projection := \reconstruction\circ\reduction ,\label{eq:projection_operator}
		\end{equation}
		\[
			\includegraphics{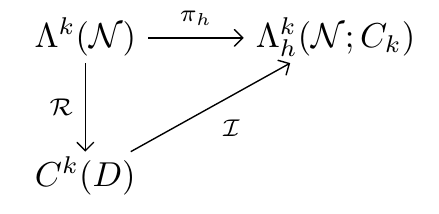}
		\]
		where $\Lambda_{h}^{k}(\manifold{N};C_{k})$ denotes the range of the reconstruction operator, $\reconstruction$.

		Due to the commuting properties of the reduction and reconstruction operators we have:
		\[
			\projection\ederiv = \reconstruction\reduction\ederiv = \reconstruction\delta\reduction = \ederiv\reconstruction\reduction = \ederiv\projection .
		\]

		Having defined the reduction, reconstruction and projection operators for the primal grid $D$, equivalent operators can be defined for the dual grid, $\tilde{D}$, \figref{fig:primal_dual_grid}. The space $\Lambda^{k}_{h}(\manifold{N};\tilde{C}_{k})$ is the space of discrete $k$-forms, with $k$-cochains associated with the dual $k$-cells,
		\[
			\Lambda_{h}^{k}(\manifold{N};\tilde{C}_{k}) = \tilde{\pi}_{h}\Lambda^{k}(\manifold{N}) := \tilde{\reconstruction}\tilde{\reduction}\, \Lambda^{k}(\manifold{N}).
		\]
		Where $\tilde{C}_{k}$ is the space of $k$-chains on the dual complex associated to the dual grid, $\tilde{D}$, $\tilde{\reduction}$ is the reduction of differential forms on the dual chains and $\tilde{\reconstruction}$ constitutes the reconstruction of differential forms from cochains defined on the dual complex.
		
		There are essentially two different ways in which the discrete Hodge-$\star$ operator can be implemented. Either we can use the definition of reconstruction of differential forms or we can use the definition of the inner product of differential forms. 
		
		For the first approach, we apply directly the definition of Hodge-$\star$ operator, \eqref{eq:hodge_star} or its one dimensional version \eqref{eq:hodge_star_1d}, to the reconstruction of a $k$-cochain and subsequently reduce the result on the topological dual grid, i.e. $\tilde{\reduction}\star\reconstruction$ or $\reduction\star\tilde{\reconstruction}$. Since here we use directly the reduction and reconstruction operators we refer to it as \emph{canonical Hodge}.
		
		For the second approach we use \eqref{eq:L2_inner_product}:
		\[
			\left(\alpha^{(k)}_{h},\beta^{(k)}_{h}\right)_{L^{2}\left(\manifold{N}\right)} := \int_{\manifold{N}}\alpha^{(k)}_{h}\wedge\star\beta^{(k)}_{h}\,,
		\]
		and note that in this definition the Hodge-$\star$ is applied to the second argument, therefore the inner product induces the Hodge-$\star$ operator. Since the underlying principle of this construction of the Hodge-$\star$ operator is an inner product we refer to it as \emph{Galerkin Hodge}.

		See \figref{fig:ex_hodge} for an example of the action of the two different Hodge-$\star$ operators on a 1-form.
		\begin{figure}[!ht]
			\center
			\includegraphics[width=0.8\textwidth]{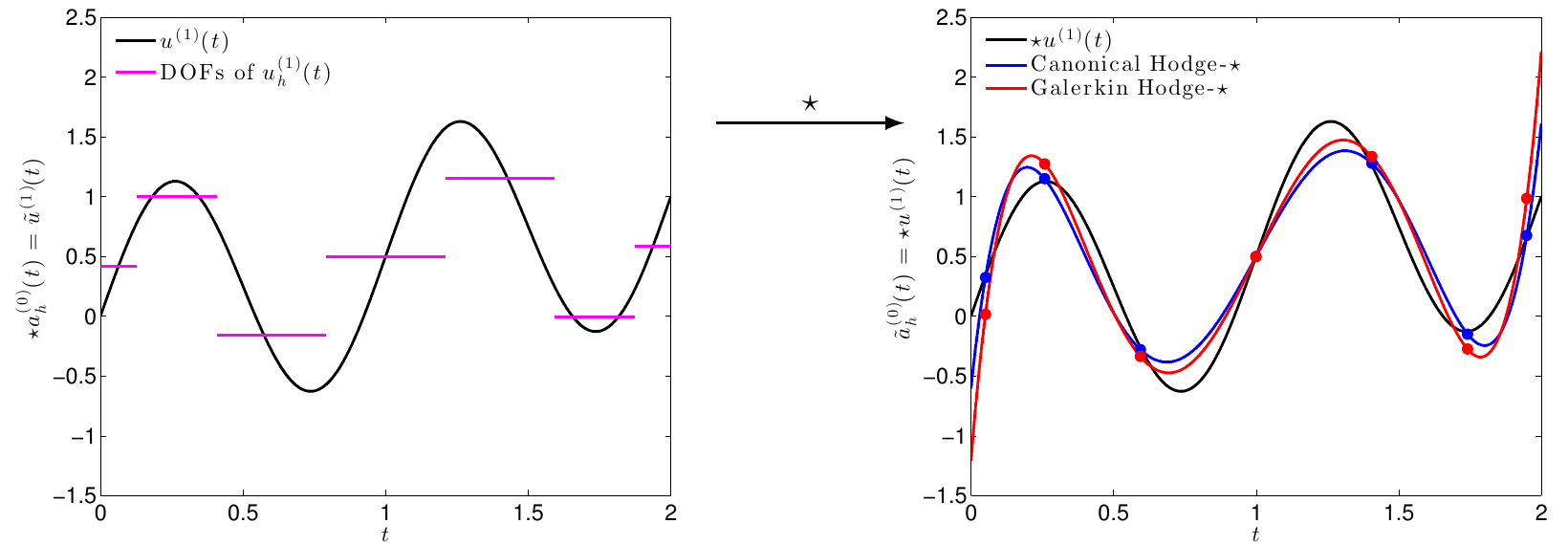}
			\caption{Example of the action of the two different discrete Hodge-$\star$ operators. The two operators result in different approximations (blue line and red line). If we refine the mesh or increase the polynomial order, the two operators will converge. Left: the analytical function and corresponding cochain (lines) associated to the degrees of freedom (dofs) of $\kdifformh{u}{1}\in\Lambda^1(\manifold{T};C_1)$. Right: the reconstruction (lines) and corresponding cochains (dots) associated to the degrees of freedom of $\kdifformh{a}{0}\in\Lambda^0(\manifold{T};\tilde{C}_0)$.}
			\label{fig:ex_hodge}
		\end{figure}
		

	\subsection{Basis forms} \label{subsection:basis_forms}
	Consider a one dimensional domain, $\Omega$, and an associated grid consisting of a collection of points $\tau_{\left(0 \right),i}$ and line segments connecting the points $\tau_{\left(1 \right),i}$.

	Let $\kformspace{k}$ be the space of smooth differentiable forms in $\Omega$. Additionally, let the finite dimensional space of differential forms be defined as $\kformspace{k}_{h} = \mathrm{span} \left( \left\{ \kdifform{\epsilon_{i}}{k} \right\} \right)$, $i = 1, \dots, \mathrm{dim} \left(  \kformspace{k}_{h} \right)$, where $\kdifform{\epsilon_{i}}{k} \in \kformspace{k}$ are basis $k$-forms. Under these conditions it is possible, see \cite{Palha2014}, to define a projection operator $\projection$ which projects elements of $\kformspace{k}$ onto elements of $\kformspace{k}_{h}$ which satisfies,
\begin{align}
 \projection \ederiv = \ederiv \projection.
\end{align}
Furthermore, it is possible to write,
\begin{align}
 \projection \kdifform{\alpha}{k} = \kdifformh{\alpha}{k} = \sum_{i} \alpha_{i} \kdifform{\epsilon_{i}}{k},
\end{align}
where
\begin{align}
 \alpha_{j} = \int_{\tau_{\left(k \right),j}} \kdifform{\alpha}{k} \quad \text{and } \quad \int_{\tau_{\left(k \right),j}} \kdifform{\epsilon_{i}}{k} = \delta^{i}_{j}, \quad k = 0,1. \label{eq:kronecker_basis_forms}
\end{align}

	\subsubsection{Primal grid}
Consider then a 0-form, $\kdifform{\alpha}{0} \in \kformspace{0}\left(Q_{ref} \right)$, where $Q_{ref} := \xi \in \left[ -1,1 \right]$. Define on $Q_{ref}$ a cell complex $D$ of order $p$ consisting of $\left(p+1 \right)$ nodes $\tau_{\left(0 \right),i} = \xi_{i}$, where $-1 \leq \xi_0 < \dots < \xi_{i}< \xi_{p} \leq 1$ are the Gauss-Lobatto-Legendre (GLL) quadrature nodes, and $p$ edges, $\tau_{\left(1 \right),i} = \left[\xi_{i-1},\xi_{i} \right]$, of which the nodes are the boundaries. The projection operator, $\projection$ reads,
\begin{align}
\projection \alpha^{(0)} = \sum_{i=0}^{p} \alpha_{i} \kdifform{\epsilon_{i}}{0}.
\end{align}
where $\kdifform{\epsilon_{i}}{0} = l_{i}\left(\xi \right)$ are the $p^{th}$ degree Lagrange polynomials and $\alpha_{i} = \kdifform{\alpha}{0}\left( \xi_{i}\right)$. Note that $\kdifform{\epsilon_{i}}{0}\left(\xi_{j} \right) = l_{i}\left(\xi_{j} \right) = \delta^{i}_{j}$, as in \eqref{eq:kronecker_basis_forms}. In \figref{fig:bassisfunctions}, top left, an example of basis 0-forms of polynomial order four is presented.

Similarly, for the projection of 1-forms in one dimension, \cite{gerritsma::edge_basis,robidoux-polynomial} derived 1-form \textit{edge polynomials}, $\kdifform{\epsilon_{i}}{1} \in \kformspacedomainh{1}{Q_{ref}}$,
\begin{align}
\kdifform{\epsilon_{i}}{1} = e_{i} \left( \xi \right) \ederiv \xi, \quad \text{with} \quad e_{i} (\xi)= - \sum_{k=0}^{i-1} \frac{d l_{k}}{\ederiv \xi}, \qquad i=1,\dots,p,
\end{align}
note that we again have,
\begin{align}
 \int_{\xi_{j-1}}^{\xi_{j}} \kdifform{\epsilon_{i}}{1} =  \int_{\xi_{j-1}}^{\xi_{j}} e_{i} \left( \xi \right) \ederiv\xi = \delta^{i}_{j}.
\end{align}
In \figref{fig:bassisfunctions}, top right, an example of the 1-form basis polynomial of order four is presented.

	\subsubsection{Dual grid}
		For the dual grid we again consider a 0-form defined on the reference interval, $\alpha^{(0)}\in\tilde{\Lambda}^{0}(Q_{ref})$. In this case, on the reference interval, $Q_{ref}$, we define a cell complex $\tilde{D}$ consisting of $(p+2)$ nodes $\tilde{\tau}_{\left(0 \right),i} = \tilde{\xi}_{i}$, where $-1 \leq \tilde{\xi}_{0} < \dots < \tilde{\xi}_{i}< \tilde{\xi}_{p+1} \leq 1$, and $(p+1)$ edges, $\tilde{\tau}_{\left(1 \right),i} = \left[\tilde{\xi}_{i-1},\tilde{\xi}_{i} \right]$, of which the nodes are the boundaries.  The nodes used are the $p$ nodes associated to Gauss-Legendre (GL) quadrature together with the two end points $\{-1,1\}$. The projection operator, $\tilde{\pi}_{h}$ reads,
\begin{align}
\tilde{\pi}_{h} \alpha^{(0)}  = \sum_{i=0}^{p} \alpha_{i} \kdifform{\tilde{\epsilon}_{i}}{0}.
\end{align}
where $\kdifform{\tilde{\epsilon}_{i}}{0} = \tilde{l}_{i}\left(\xi \right)$ are the Lagrange polynomials of degree $(p+1)$ and $\alpha_{i} = \kdifform{\alpha}{0}\left( \tilde{\xi}_{i}\right)$. Note that $\kdifform{\tilde{\epsilon}_{i}}{0}\left(\tilde{\xi}_{j} \right) = \tilde{l}_{i}\left(\tilde{\xi}_{j} \right) = \delta^{i}_{j}$, as in \eqref{eq:kronecker_basis_forms}. In \figref{fig:bassisfunctions}, bottom right, an example of basis 0-forms of polynomial order four is presented.

Similarly, the dual 1-form basis polynomials, $\kdifform{\tilde{\epsilon}_{i}}{1} \in \tilde{\Lambda}^{1}(Q_{ref})$, are obtained from
\begin{align}
\kdifform{\tilde{\epsilon}_{i}}{1} = \tilde{e}_{i} \left( \xi \right) \ederiv \xi, \quad \text{with} \quad \tilde{e}_{i} = - \sum_{k=0}^{i-1} \frac{d \tilde{l}_{k}}{\ederiv \xi}, \qquad i=1,\dots,p+1,
\end{align}
note that we again have,
\begin{align}
 \int_{\tilde{\xi}_{j-1}}^{\tilde{\xi}_{j}} \kdifform{\tilde{\epsilon}_{i}}{1} =  \int_{\tilde{\xi}_{j-1}}^{\tilde{\xi}_{j}} \tilde{e}_{i} \left( \xi \right) \ederiv\xi = \delta^{i}_{j}.
\end{align}
In \figref{fig:bassisfunctions}, bottom left, an example of the 1-form basis of polynomial order four is presented.

	\subsubsection{Exterior derivative}
The \textit{exterior derivative} of the primal basis 0-forms is given by,
\begin{align}
 \ederiv \kdifform{\epsilon_{i}}{0} = \frac{d l_{i}}{\ederiv \xi} \, \ederiv\xi = \sum_{k=0}^{i} \frac{d l_{k}}{\ederiv \xi} - \sum_{k=0}^{i-1} \frac{d l_{k}}{\ederiv \xi} = -\kdifform{\epsilon_{i+1}}{1} + \kdifform{\epsilon_{i}}{1}, \quad i = 0, \dots, p.
\end{align}
The derivation of the exterior derivative of the dual basis 0-forms is obtained following the same procedure by substituting the primal functions by the respective dual functions.

\begin{example}
	Consider a 0-form, $\alpha^{(0)}$, and the 1-form, $\beta^{(1)} = \ederiv\alpha^{(0)}$, both defined in $\Omega = [-1,1]$ and take a grid $D$ as the one defined in \figref{fig:chain_boundary_example_3}. If the discrete differential forms are given by:
	\begin{figure}[!ht]
	\centering
	\includegraphics{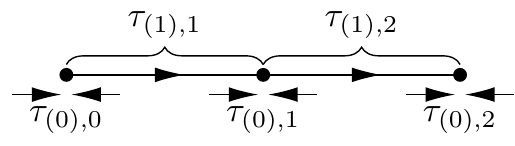}
	\caption{Representation of the 1-dimensional grid, $D$, with the default orientation of the geometric objects.}
	\label{fig:chain_boundary_example_3}
	\end{figure}
	
	\[
		\projection\alpha^{(0)} = \alpha_{0}\epsilon^{(0)}_{0} + \alpha_{1}\epsilon^{(0)}_{1} + \alpha_{2}\epsilon^{(0)}_{2} \qquad \text{and} \qquad \projection\beta^{(1)} = \beta_{1}\epsilon^{(1)}_{1} + \beta_{1}\epsilon^{(1)}_{2},
	\]
	then
	\begin{align*}
		\ederiv\projection\alpha^{(0)} &= \ederiv\left(\alpha_{0}\epsilon^{(0)}_{0} + \alpha_{1}\epsilon^{(0)}_{1} + \alpha_{2}\epsilon^{(0)}_{2}\right), \\
		&= \alpha_{0}\ederiv\epsilon^{(0)}_{0} + \alpha_{1}\ederiv\epsilon^{(0)}_{1} + \alpha_{2}\ederiv\epsilon^{(0)}_{2} \\
		&= \alpha_{0}\left(-\epsilon^{(1)}_{1}\right) + \alpha_{1}\left(\epsilon^{(1)}_{1} - \epsilon^{(1)}_{2}\right) + \alpha_{2}\left(\epsilon^{(1)}_{2}\right) \\
		& = \left(\alpha_{1}-\alpha_{0}\right)\epsilon^{(1)}_{1} + \left(\alpha_{2}-\alpha_{1}\right)\epsilon^{(1)}_{2}\\
		&= \beta_{1}\epsilon^{(1)}_{1} + \beta_{1}\epsilon^{(1)}_{2} = \beta^{(1)}.
	\end{align*}
	This means then that $\beta_{1} = \left(\alpha_{1}-\alpha_{0}\right)$ and $\beta_{2} = \left(\alpha_{2}-\alpha_{1}\right)$. This corresponds exactly to the coefficients obtained from the projection of $\beta^{(1)}$, since, recall \eqref{eq:definition_reduction_operator}:
	\[
		\left.\reduction\beta^{(1)}\right|_{\tau_{(1),i}} := \int_{\tau_{(1),i}}\beta^{(1)} = \int_{\tau_{(1),i}}\ederiv\alpha^{(0)} = \alpha(\tau_{(0),i}) - \alpha(\tau_{(0),i-1}),
	\]
	by Stokes' theorem, \eqref{eq:stokes_theorem}.
\end{example}

The exterior derivative of a discrete 0-form can then be written as,
\begin{align}
 \ederiv \kdifformh{\alpha}{0} = \ederiv \sum_{i=0}^{p} \alpha_{i} \kdifform{\epsilon_{i}}{0} = \sum_{i=1}^{p} \sum_{j=0}^{p} \mathsf{E}_{(0,1)}^{j,i} \alpha_{j} \kdifform{\epsilon_{i}}{1}, \label{eq:exterior_derivative_incidence_matrix}
\end{align}
where, $\mathsf{E}_{(0,1)}^{j,i}$ are the coefficients of the matrix representation of the coboundary operator, \eqref{eq:definition_incidence_matrix}, see \cite{Kreeft2011,Palha2014} for more details.

		\begin{figure}[!ht]
			\center
			\includegraphics{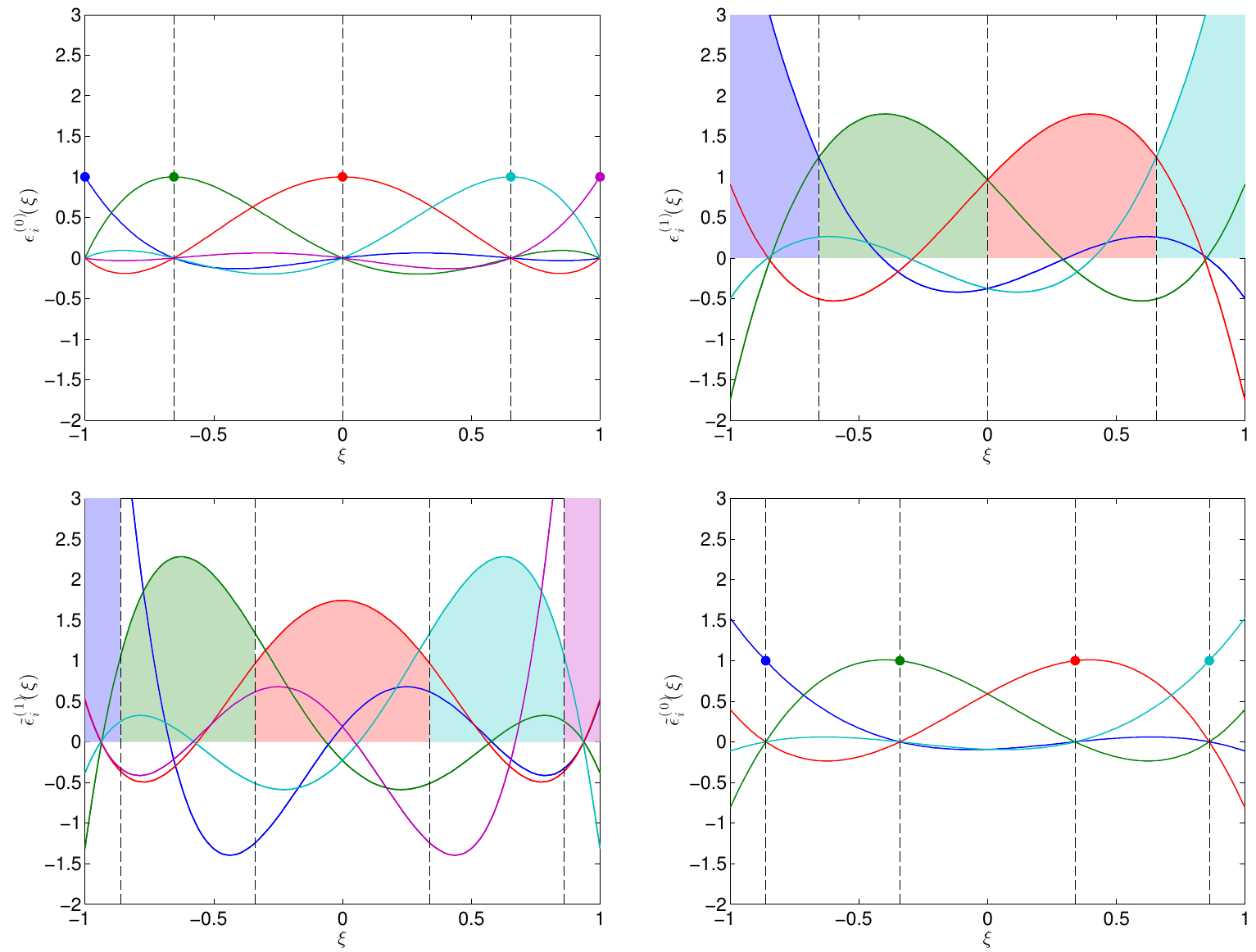}
			\caption{Basis functions associated to a $p=4$ spectral finite element. Top left: primal grid 0-form basis functions $\epsilon^{(0)}_{i}\left( \xi \right)$. Top right: primal grid 1-form basis functions $\epsilon^{(1)}_{i}\left( \xi \right)$. Bottom left: dual grid 1-form basis functions $\tilde{\epsilon}^{(1)}_{i}\left( \xi \right)$. Bottom right: dual grid 0-form basis functions $\tilde{\epsilon}^{(0)}_{i}\left( \xi \right)$. }
\label{fig:bassisfunctions}
		\end{figure}

\section{Time integrators}\label{sec:time_integrators}

\subsection{Introductory concepts}
As said before, the field of geometric integration is vast. Therefore, in what follows next, only key concepts essential to the understanding of the basic properties (such as symplecticity) of the numerical time integrators presented will be introduced. For a more detailed discussion and elaboration on these concepts, the authors suggest \cite{reich2005,Hairer2006}.
			
			\begin{definition}[\textbf{Symplectic form}]
				\cite{abraham_diff_geom} Let $\manifold{M}$ be a manifold of dimension $2n$. A symplectic form $\kdifform{\omega}{2}$ is a 2-form on $\manifold{M}$ such that:
				\begin{enumerate}
					\item $\kdifform{\omega}{2}$ is closed, that is $\ederiv\kdifform{\omega}{2}=0$;
					\item For each $p\in\manifold{M}$ and $\vec{v},\vec{u}\in\tangentspace{\manifold{M}}$ the identity $\kdifform{\omega}{2}(\vec{v},\vec{u})=0$ is satisfied if and only if $\vec{v}=0$ or $\vec{u}=0$ ($\kdifform{\omega}{2}$ is non-degenerate).
				\end{enumerate}
			\end{definition}
			
			For a Hamiltonian system, which involves a pair of $n$-dimensional variables $q$ (configuration variables) and $p$ (momenta), the symplectic form takes the form:
			\[
				\kdifform{\omega}{2} = \sum_{i}\ederiv q_{i}\wedge\ederiv p_{i}
			\]
			
			\begin{definition}[\textbf{Symplectic map}]
				\cite{abraham_diff_geom, Hairer2006} Given a $2n$-dimensional manifold $\manifold{M}$, a diffeomorphism $\Phi:\manifold{M}\mapsto\mathbb{R}^{2n}$ is a symplectic map if it preserves the symplectic form, that is:
				\[
					\Phi^{*}\kdifform{\omega}{2} = \kdifform{\omega}{2}
				\]
			\end{definition}
			
			\begin{definition}[\textbf{Hamiltonian vector field}]
				\cite{abraham_diff_geom} On a $2n$-dimensional manifold $\manifold{M}$ a vector field $\vec{v}_{H}\in\vectorfieldspace{\manifold{M}}$ is called Hamoltonian if there is a 0-form $\kdifform{H}{0}\in\Lambda^{0}(\manifold{M})$ such that:
				\[
					\iota_{\vec{v}_{H}}\kdifform{\omega}{2} = \ederiv \kdifform{H}{0}\;,
				\]
			\end{definition}
			and $\kdifform{\omega}{2}$ is the symplectic form. Moreover, in local coordinates, $\vec{v}_{H}$ takes the form:
			\[
				\vec{v}_{H} = \sum_{i}^{n} \left(\frac{\partial H}{\partial p_{i}}\frac{\partial}{\partial q_{i}} - \frac{\partial H}{\partial q_{i}}\frac{\partial}{\partial p_{i}}\right)\;.
			\]
			Additionally, it is possible to show, \cite{abraham_diff_geom, goldsteinClassicalMechanics}, that the integral curves generated by $\vec{v}_{H}$ satisfy Hamilton's equations:
			\[
				\frac{\ederiv q_{i}}{\ederiv t} = \frac{\partial H}{\partial p_{i}}\;,\quad \frac{\ederiv p_{i}}{\ederiv t} = - \frac{\partial H}{\partial q_{i}}\;,
			\]
			which is simply a restatement of \eqref{eq::hamiltonian_system}.
			
			From these definitions it follows that:
			
			\begin{theorem}[\textbf{Properties of Hamiltonian systems}]
				\cite{abraham_diff_geom,Hairer2006} Given a $2n$-dimensional manifold $\manifold{M}$ and an Hamiltonian vector field $\vec{v}_{H}$ associated to the Hamiltonian $\kdifform{H}{0}$, then the flow generated by the vector field, $\Phi_{t}$, has the following properties:
				\begin{enumerate}
					\item $\Phi_{t}$ is symplectic.
					\item The Hamiltonian $\kdifform{H}{0}$ is conserved along integral lines: $\Phi^{*}_{t}\kdifform{H}{0} = \kdifform{H}{0}$. 
				\end{enumerate} 
			\end{theorem}
			\begin{proof}
				See \cite{abraham_diff_geom}.
			\end{proof}
			
			Hence, integral lines are isocontours of the Hamiltonian. Therefore, given the symplecticity and energy conservation properties of Hamiltonian systems it is natural to search for numerical integrators that are symplectic or exact energy preserving, leading to numerical integral lines that remain in a bounded region close to the real integral curve (symplectic integrator) or lie on it (energy conserving integrator).
			
			\begin{definition}[\textbf{Symplectic numerical one-step method}]
				\cite{Hairer2006} A one-step numerical method, $\Phi_{h}:(q_{t_{0}},p_{t_{0}})\in\mathbb{R}^{2n}\mapsto(q_{t_{1}},p_{t_{1}})\in\mathbb{R}^{2n}$, is symplectic if the one step map:
				\[
					\left[\begin{array}{c}
						q_{t_{1}} \\
						p_{t_{1}}
					\end{array}\right]
					= 
					\Phi_{h}\left(
					\left[
					\begin{array}{c}
						q_{t_{0}} \\
						p_{t_{0}}
					\end{array}
					\right]\right)\;,
				\]
				is symplectic whenever the method is applied to a smooth Hamiltonian system.
			\end{definition}
			
			In this work we restrict our study to either symplectic or exact energy preserving time integrators. The choice between an integrator that is exact energy preserving or symplectic depends on the problem at hand, see for example \cite{Simo1992}. 
			
			The extension of the mimetic framework to time dependent problems, and in particular to the solution of first order ordinary differential equations, is done by noticing that the system of first order ordinary differential equations, \eqref{eq::TestCases_ode}:
			\[
				\frac{\ederiv \boldsymbol{y}}{\ederiv t} = \boldsymbol{h}(\boldsymbol{y}),\quad \boldsymbol{y}(t_{0}) = \boldsymbol{y}^{0},\quad \boldsymbol{y}\in\mathcal{M}\subset\mathbb{R}^{n}\quad \text{and} \quad t\in [t_{0},t_{f}]\subset\mathbb{R}\;,
			\]
			can be rewritten in differential geometric terms as:
			
			\begin{equation}
			   \star\ederiv \kdifform{\boldsymbol{y}}{0} = \kdifform{\boldsymbol{h}}{0}(\kdifform{\boldsymbol{y}}{0}), \quad\boldsymbol{y}^{(0)}(t_{0})=\boldsymbol{y}^{0},\quad \kdifform{y_{i}}{0}\in\Lambda^{0}([t_{0},t_{f}]), \quad  \kdifform{h_{i}}{0}\in\tilde{\Lambda}^{0}([t_{0},t_{f}])\;. \label{eq::TestCases_ode_primal_dual}
			\end{equation}
			That is, each coordinate of an integral curve is a 0-form in time and its time derivative is a 0-form on the dual grid, obtained by exterior differentiation in time. 
			
			Based on the formulation \eqref{eq::TestCases_ode_primal_dual}, two arbitrary order discrete integrators related to Gauss collocation methods can be constructed: one based on the canonical Hodge-$\star$ and another one based on the Galerkin Hodge-$\star$. Each discrete integrator has different properties, as will be seen.

			 Although we focus on Hamiltonian systems, the time integrators we will present can be applied to systems of any dimension (Hamiltonian and non-Hamiltonian) but, as will be shown, when applied to autonomous Hamiltonian systems, they display important properties. Therefore, for the sake of clarity, and without loss of generality, we present initially a first order ordinary differential equation only in one variable and then extend it to systems. For a system with one variable, the formulation presented in \eqref{eq::TestCases_ode_primal_dual}, becomes
			\begin{equation}
			   \star\ederiv \kdifform{y}{0} = \kdifform{h}{0}(\kdifform{y}{0}), \quad \kdifform{y}{0}\in\Lambda^{0}([t_{0},t_{f}]), \quad  \kdifform{h}{0}\in\tilde{\Lambda}^{0}( [t_{0},t_{f}])\;. \label{eq::TestCases_ode_primal_dual_one_variable}
			\end{equation}
			
			\subsection{Mimetic canonical integrator}
				As was done in  \secref{subsection:basis_forms}, the physical quantities are discretized in the following way:
				\[
					\kdifform{y}{0} \in \Lambda^{0}(\manifold{T}) \Rightarrow \kdifform{y_{h}}{0} \in \Lambda_{h}^{0}(\manifold{T}) \quad\textrm{and}\quad \kdifform{h}{0} \in \tilde{\Lambda}^{0}(\manifold{T}) \Rightarrow \kdifform{h_{h}}{0} \in \tilde{\Lambda}_{h}^{0}(\manifold{T})\;,
				\]
				where $\manifold{T} = [t_{0},t_{f}]$. The discrete physical quantities are represented as:
				\begin{equation}
					 \kdifform{y_{h}}{0} = \sum_{j=0}^{p} y^{j} \kdifform{\epsilon_{j}}{0}(\tau) \quad \textrm{and}\quad \kdifform{h_{h}}{0} = \sum_{j=1}^{p} h^{j} \kdifform{\tilde{\epsilon}_{j}}{0}(\tau)\;, \label{eq::TestCases_primal_dual_discretization_one_variable}
				\end{equation}
				where $\tau = \left(\frac{2(t-t_{0})}{t_{f}-t_{0}}-1\right)$ and $\kdifform{\epsilon_{j}}{0}(\tau)$, $\kdifform{\tilde{\epsilon}_{j}}{0}(\tau)$ are the 0-form basis functions on the primal and dual grids, respectively, as described in \secref{subsection:basis_forms}. The only difference here being that for the dual grid basis, $\tilde{\epsilon}_{i}^{(0)}$, we consider only the Gauss-Legendre nodes:
				\[
					\epsilon_{j}^{(0)}(\tau) = l_{j}(\tau) \quad j=0,\dots,p, \quad \text{and} \quad \tilde{\epsilon}_{j}^{(0)}(\tau) = \tilde{l}^{g}_{j}(\tau)\quad j=1,\dots,p\;.
				\]
				Note also that:
				\begin{equation}
					y^{j} = \kdifform{y}{0}(t^{j}) \quad\mathrm{and}\quad h^{j} = \kdifform{h}{0}\left(\kdifform{y}{0}(\tilde{\tau}^{j})\right) =  \kdifform{h}{0}\left(\sum_{j=1}^{p} y^{j} \kdifform{\epsilon_{j}}{0}(\tilde{\tau}^{j})\right)\;, \label{eq::TestCases_primal_dual_discreet_DOF_one_variable}
				\end{equation}
				where $\tau^{j}$ are the $(p+1)$ intermediate time instants associated to the nodes of the primal grid, $\tilde{\tau}^{j}$ are the $p$ intermediate time instants associated to the nodes of the dual grid and $t^{j} = (\tau^{j}+1)\frac{t_{f}-t_{0}}{2}+t_{0}$. It is important to remark that the upper index corresponds to degrees of freedom in time. A lower index will later correspond to distinct physical variables that appear in a system of ordinary differential equations. The choice to display an upper index instead of another lower index was done simply to make a distinction between space and time degrees of freedom. With this discretization the degrees of freedom appear staggered in time, as can be seen in \figref{fig:mimetic_integrator}.
				\begin{figure}[ht]
					\centering
						\includegraphics[width=0.3\textwidth]{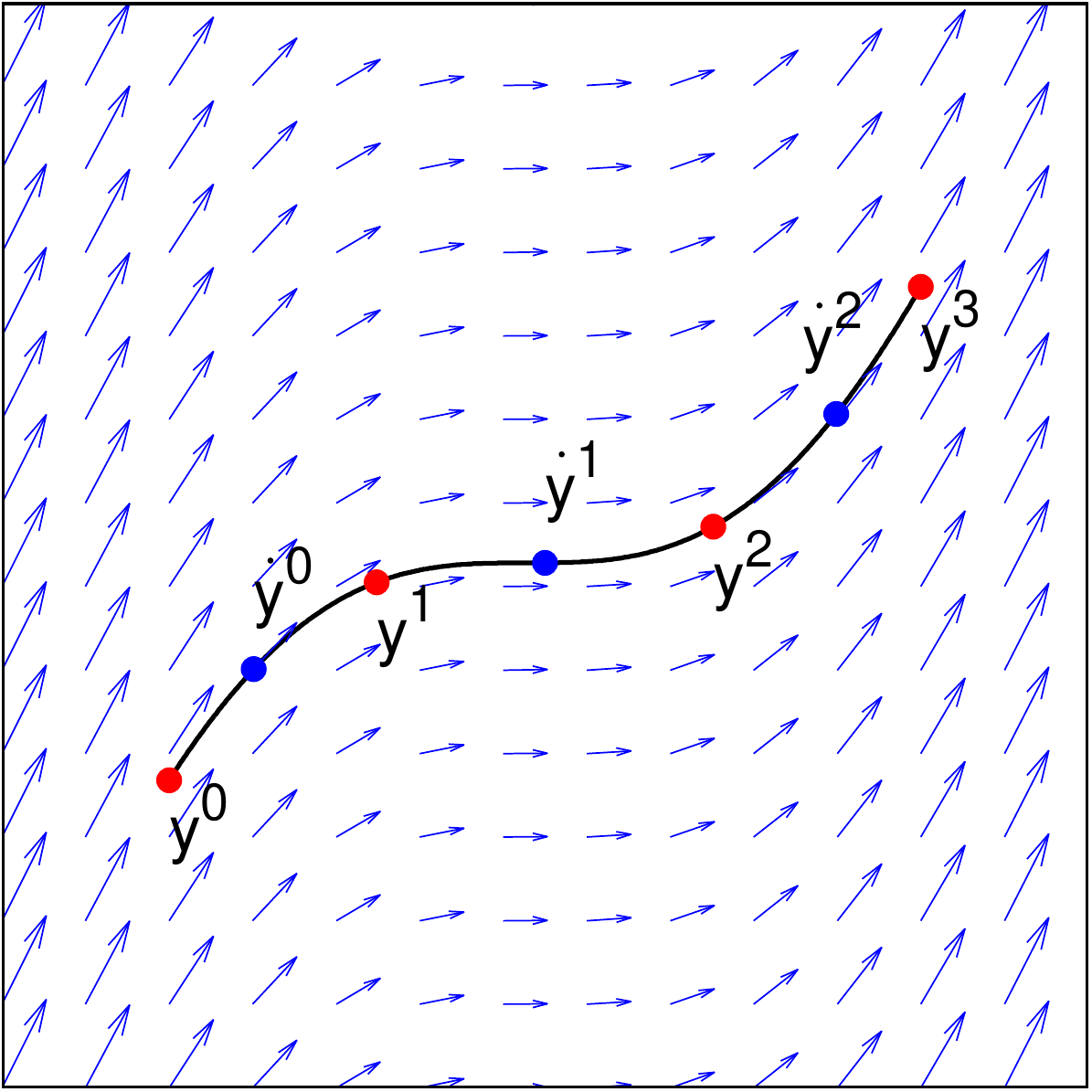}
						\caption{Geometric interpretation of the solution of \eqref{eq::TestCases_ode_primal_dual_one_variable} obtained with the mimetic discretization \eqref{eq::TestCases_primal_dual_integrator}: $(t,y^{(0)}(t))$. The nodes of the primal grid where the trajectory is discretized are represented in red (corresponding to Gauss-Lobatto-Legendre quadrature nodes). The nodes associated to the dual grid where the time derivative is discretized are depicted in blue (corresponding to Gauss-Legendre quadrature nodes). The flow field, represented by arrows, is tangent to the curve at the Gauss nodes. That is, the derivative of the approximate trajectory is exactly equal to the flow field at the Gauss nodes.}
						\label{fig:mimetic_integrator}
				\end{figure}
				
				Taking the discretization presented in \eqref{eq::TestCases_primal_dual_discretization_one_variable} and substituting in \eqref{eq::TestCases_ode_primal_dual_one_variable} we obtain:
				\begin{equation}
					\star\ederiv \left(\sum_{j=0}^{p}y^{j} \kdifform{\epsilon_{j}}{0}(\tau)\right) =  \sum_{j=1}^{p} h^{j} \kdifform{\tilde{\epsilon}_{j}}{0}(\tau)\;. \label{eq::TestCases_primal_dual_discrete_step_1}
				\end{equation}
				Given this discretization and using \eqref{eq:exterior_derivative_incidence_matrix}, it is possible to rewrite \eqref{eq::TestCases_primal_dual_discrete_step_1} as:
				\begin{equation}
					\star\sum_{j=1,k=0}^{p} \incidenceboundary{0}{1}^{k,j}y^{k} \kdifform{\epsilon_{j}}{1}(\tau) =  \sum_{j=1}^{p} h^{j} \kdifform{\tilde{\epsilon}_{j}}{0}(\tau)\;.\label{eq::ode_hodge_one_variable}
				\end{equation}
				Following the canonical projection procedure as  presented in \secref{subsection:mimetic_projection}, the Hodge-$\star$ of the term on the left can be computed, yielding:
				\begin{equation}
					\sum_{j=1,l=1,k=0}^{p} \incidenceboundary{0}{1}^{k,l}y^{k} e_{l}(\tilde{\tau}^{j})\, \kdifform{\tilde{\epsilon}_{j}}{0}(\tau) = \sum_{j=1}^{p} h^{j} \kdifform{\tilde{\epsilon}_{j}}{0}(\tau)\;. \label{eq::TestCases_canonical_projection}
				\end{equation} 
				Once more, since the basis 0-forms are the same, this equality can be stated as:
				\begin{equation}
					\sum_{l=1,k=0}^{p} \incidenceboundary{0}{1}^{k,l}y^{k} e_{l}(\tilde{\tau}^{j}) = h^{j}\;. \label{eq::TestCases_primal_dual_canonical_projection}
				\end{equation}
				Explicitly expanding $h^{j}$,  \eqref{eq::TestCases_primal_dual_discreet_DOF_one_variable}, one obtains:
				\begin{equation}
					\sum_{l=1,k=0}^{p} \incidenceboundary{0}{1}^{k,l}y^{k} e_{l}(\tilde{\tau}^{j}) = \kdifform{h}{0}\left(\sum_{k=0}^{p} y^{k} \kdifform{\epsilon_{k}}{0}(\tilde{\tau}^{j})\right), \quad j=1,\dots,p\;. \label{eq::TestCases_primal_dual_integrator_one_variable}
				\end{equation}
				This is a non-linear set of equations for the unknown values $y^{j}$ which can be solved, for instance, with a Newton-Raphson iterative solver, for known initial values $y^{0}=y(t_{0})$.
				
				For a system of  ordinary differential equations, as in \eqref{eq::TestCases_ode_primal_dual}, instead of a single time dependent unknown $y(t)$, we now have $M$ distinct time dependent unknowns $y_{i}$, with $i=1,\dots,M$. Following the same steps as presented previously, but now for multiple $y$'s, one arrives at the following algebraic system:
				\begin{equation}
					\sum_{l=1,k=0}^{p} \incidenceboundary{0}{1}^{k,l}y_{i}^{k} e_{l}(\tilde{\tau}^{j}) = \kdifform{h_{i}}{0}\left(\sum_{k=0}^{p} y_{1}^{k} \kdifform{\epsilon_{k}}{0}(\tilde{\tau}^{j}),\dots, \sum_{k=0}^{p} y_{M}^{k} \kdifform{\epsilon_{k}}{0}(\tilde{\tau}^{j})\right), \quad j=1,\dots,p \quad\text{and}\quad i=1,\dots,M\;.\label{eq::TestCases_primal_dual_integrator}
				\end{equation}
				The numerical time integrator resulting from this discretization will be referred to as mimetic canonical integrator, or MCI, since it is obtained from a canonical projection $\pi_{h}$. 
								
				\begin{remark}
					It is important to note that this formulation, although different from Gauss-Legendre collocation methods, yields identical results (up to machine accuracy). Gauss-Legendre collocation method constructs an interpolating polynomial with Gauss-Legendre nodes for both the trajectory and its derivative, whereas here we use an interpolating polynomial with Gauss-Lobatto-Legendre nodes for the trajectory (primal grid) and another interpolating polynomial with Gauss-Legendre nodes for its derivative (dual grid).
				\end{remark}

			\subsection{Mimetic Galerkin integrator}
				As was done for the mimetic canonical integrator, the physical quantities are defined in the following way:
				\[
					\kdifform{y}{0} \in \Lambda^{0}(\manifold{T}) \Rightarrow \kdifform{y_{h}}{0} \in \Lambda_{h}^{0}(\manifold{T}) \quad\textrm{and}\quad \kdifform{h}{0} \in \tilde{\Lambda}^{0}(\manifold{T})\;,
				\]
				where $\manifold{T} = [t_{0},t_{f}]$. The discrete physical quantities are represented as:
				\begin{equation}
					 \kdifform{y_{h}}{0} = \sum_{j=0}^{p} y^{j} \kdifform{\epsilon_{j}}{0}(\tau)\;, \label{eq::TestCases_primal_dual_discretization_one_variable_galerkin}
				\end{equation}
				where $\tau = \left(\frac{2(t-t_{0})}{t_{f}-t_{0}}-1\right)$ and $\kdifform{\epsilon_{j}}{0}(\tau)$ are the 0-form basis functions on the primal grid, as described in \secref{subsection:basis_forms}:
				\[
					\epsilon_{j}^{(0)}(\tau) = l_{j}(\tau) \quad j=0,\dots,p, \,.
				\]
				Note also that:
				\begin{equation}
					y^{j} = \kdifform{y}{0}(t^{j}) \;, \label{eq::TestCases_primal_dual_discreet_DOF_one_variable_galerkin}
				\end{equation}
				where $t^{j}=(\tau^{j}+1)\frac{t_{f}-t_{0}}{2}+t_{0}$ and $\tau^{j}$ are again the $(p+1)$ intermediate time instants associated with the nodes of the primal grid.				
				
				For the Galerkin Hodge formulation, instead of using a canonical projection to go from \eqref{eq::ode_hodge_one_variable} to  \eqref{eq::TestCases_canonical_projection} and to discretize $h^{(0)}$, one starts with \eqref{eq::TestCases_ode_primal_dual_one_variable}, discretizes the left hand side and applies \eqref{eq:L2_inner_product}:
			\begin{equation}
				\sum_{j=1,l=1,k=0}^{p} \incidenceboundary{0}{1}^{k,l}y^{k} e_{l}(\tilde{\tau}^{j})\, \int_{\mathcal{T}}\kdifform{\tilde{\epsilon}_{j}}{0}(\tau)\wedge\star\kdifform{\tilde{\epsilon}_{m}}{0}(\tau) =\innerspace{\kdifform{h}{0}\left(\sum_{k=0}^{p}y^{k}\epsilon_{k}^{(0)}(\tau)\right)}{\kdifform{\tilde{\epsilon}_{m}}{0}(\tau)}{\manifold{T}}\;, \quad m=1,\dots,p \;,\label{eq::TestCases_primal_primal_galerkin_projection_one_variable}
			\end{equation}
			where $\manifold{T}=[-1,1]$. Using \eqref{eq:L2_1d_0form} together with \eqref{eq:hodge_star_1d} the previous expression reduces to:
			\begin{equation}
				\sum_{j=1,l=1,k=0}^{p} \incidenceboundary{0}{1}^{k,l}y^{k} e_{l}(\tilde{\tau}^{j})\, \int_{-1}^{1}\tilde{l}_{j}^{g}(\tau)\tilde{l}_{m}^{g}(\tau)\sqrt{g}\,\ederiv \tau = \int_{-1}^{1}h\left(\sum_{k=0}^{p}y^{k}\epsilon_{k}^{(0)}(\tau)\right)\tilde{l}_{m}^{g}(\tau)\sqrt{g}\,\ederiv \tau\quad m=1,\dots,p\;,\label{eq::TestCases_primal_primal_galerkin_integrator_one_variable}
			\end{equation}
			where, for constant time steps, $\sqrt{g} = \frac{2}{t_{f}-t_{0}}$. This is again a non-linear set of equations for the unknown values $y^{j}$.
			
			As before, for a system of  ordinary differential equations, as in \eqref{eq::TestCases_ode_primal_dual}, there exists $M$ distinct time dependent unknowns $y_{i}$, with $i=1,\dots,M$. Following the same steps as presented previously, but now for multiple $y$'s, one arrives at the following algebraic system:
				\begin{equation}
				\sum_{j=1,l=1,k=0}^{p} \incidenceboundary{0}{1}^{k,l}y_{i}^{k} e_{l}(\tilde{\tau}^{j})\, \int_{-1}^{1}\tilde{l}^{g}_{j}(\tau)\tilde{l}^{g}_{m}(\tau)\sqrt{g}\,\ederiv \tau = \int_{I}\kdifform{h_{i}}{0}\left(\sum_{k=0}^{p}y_{1}^{k}\epsilon_{k}^{(0)}(\tau),\dots,\sum_{k=0}^{p}y_{M}^{k}\epsilon_{k}^{(0)}(\tau)\right)\tilde{l}^{g}_{m}(\tau)\,\ederiv \tau \;, \label{eq::TestCases_primal_primal_galerkin_integrator}
			\end{equation}
				with $m=1,\dots,p$ and $i=1,\dots,M$. The numerical time integrator resulting from this discretization will be referred to as mimetic Galerkin integrator, or MGI, since it is obtained from a Galerkin projection.
				
				\begin{remark}
					It is important to note that this formulation turns out to be a reformulation of the energy preserving collocation method introduced in \cite{hairerEnergyPreservingCollocation2010}, which in turn is an extension to higher order of the method introduced in \cite{Quispel2008}. 
				\end{remark}

			\subsection{Properties of the mimetic integrators}	
			Having presented the two mimetic integrators it is essential now to investigate their properties. In this section we will prove that the mimetic canonical integrator is symplectic and that the mimetic Galerkin integrator is energy preserving. 
			
			\subsubsection{Symplecticity of mimetic canonical integrator}
			We start by noting that  the mimetic canonical integrator is a special case of a Runge-Kutta method.
			\begin{definition}[\textbf{Runge-Kutta integrator}]
				\cite{Hairer2006} Take the system of ordinary differential equations in \eqref{eq::TestCases_ode} and let $b_{i}$, $a_{ij}$ ($i,j = 1, \dots, s$) be real numbers. An $s$-stage Runge-Kutta method is given by:
				\begin{align}
					\boldsymbol{k}_{i} &= \boldsymbol{h}\left(\boldsymbol{y}^{0} + \Delta t \sum_{j=1}^{s}a_{ij}\boldsymbol{k}_{j}\right),\quad i=1,\dots,s \label{eq:rk_k} \\
					\boldsymbol{y}^{1} &= \boldsymbol{y}^{0} + \Delta t \sum_{i=1}^{s}b_{i}\boldsymbol{k}_{i}\;. \label{eq:rk_y}
				\end{align}
			\end{definition}
			
			By setting $\boldsymbol{k}_{i} = \frac{\ederiv\boldsymbol{y}}{\ederiv t}\,(\tilde{t}^{i})$, expressing $\frac{\ederiv\boldsymbol{y}}{\ederiv t}\,(t)$ and, by integration, $\boldsymbol{y}(t)$ in terms of $\boldsymbol{k}_{i}$ it follows that the mimetic canonical integrator is a special case of a Runge-Kutta method with a full matrix of coefficients $a_{ij}$, therefore an implicit Runge-Kutta method. We show this explicitly for a first order differential equation only in one variable. For that we start with \eqref{eq:rk_k} and replace the right hand side using \eqref{eq::TestCases_primal_dual_integrator_one_variable}:
			\begin{equation}
				k_{i} = \sum_{l=1,k=0}^{p} \mathsf{E}_{(0,1)}^{k,l}y^{k}\epsilon^{(1)}_{l}(\tilde{\tau}^{i})\quad i=1,\dots,p\,.
			\end{equation}
			We can write this expression in matrix notation as:
			\begin{equation}
				\boldsymbol{k} = \mathsf{A}\,\mathsf{E}^{t}_{(0,1)}\,\boldsymbol{y}\,,
			\end{equation}
			where $\mathsf{A}_{i,l} = \epsilon^{(1)}_{l}(\tilde{\xi}^{i})$. This can be rewritten as:
			\begin{equation}
				\boldsymbol{k} = \mathsf{A}\,\left(\hat{\mathsf{E}}\,\hat{\boldsymbol{y}} + \hat{\boldsymbol{c}}\, y^{0}\right)\,,\label{eq:rk_deduction_k}
			\end{equation}
			where,
			\begin{align}
				\hat{\mathsf{E}}_{i,j} = \mathsf{E}_{(0,1)}^{j+1,i}, &\quad j=1,\dots,p, \quad i=1,\dots,p\,, \\
				c_{i} = \mathsf{E}_{(0,1)}^{0,i}, &\quad i=1,\dots,p\,, \\
				\hat{y}^{i} = y^{i}, &\quad i=1,\dots,p\,. 
			\end{align}
			We can then rewrite \eqref{eq:rk_deduction_k} into:
			\begin{equation}
				\hat{\mathsf{E}}^{-1}\mathsf{A}^{-1}\boldsymbol{k} - \hat{\mathsf{E}}^{-1}\hat{\boldsymbol{c}}\, y^{0}=\hat{\boldsymbol{y}} \,.\label{eq:rk_deduction_k_inverse}
			\end{equation}
			Given the definition of $\hat{\mathsf{E}}$ and $\hat{\boldsymbol{c}}$ we have that:
			\begin{align}
				\left(\hat{\mathsf{E}}^{-1}\right)_{i,j} &= \left\{\begin{array}{c}1\quad\text{if}\quad i\leq j \\ 0\quad\text{if}\quad i> j \\ \end{array}\right. \,,\\
				-\hat{\mathsf{E}}^{-1}\hat{\boldsymbol{c}}\,y^{0} &= \left[\begin{array}{c}y^{0} \\ \vdots \\ y^{0}\end{array}\right]\,.
			\end{align}
			From \eqref{eq:rk_deduction_k_inverse} it is straightforward to see that:
			\begin{equation}
				y^{p} = \sum_{i=1}^{p}\left(\hat{\mathsf{E}}^{-1}\mathsf{A}^{-1}\right)_{p,i}k_{i} + y^{0}\,. \label{eq:rk_b_mimetic}
			\end{equation}
			From \eqref{eq:rk_b_mimetic} and \eqref{eq:rk_y} we can directly see that:
			\[
				b_{i} = \frac{1}{\Delta t}\left(\hat{\mathsf{E}}^{-1}\mathsf{A}^{-1}\right)_{p,i}\,.
			\]
			From \eqref{eq:rk_deduction_k_inverse} and \eqref{eq:rk_k} we have,
			\[
				a_{i,j} = \frac{1}{\Delta t}\left(\hat{\mathsf{E}}^{-1}\mathsf{A}^{-1}\right)_{i,j}\,.
			\]
			
			\begin{theorem}[\textbf{Symplecticity of Runge-Kutta integrators}] \label{theorem::symplecticity_runge_kutta}
				\cite{Hairer2006} If a Runge-Kutta method conserves quadratic first integrals (i.e. $I(\boldsymbol{y}_{1}) = I(\boldsymbol{y}_{0})$ whenever $I(\boldsymbol{y}) = \boldsymbol{y}^{t}\matrixoperator{C}\boldsymbol{y}$, with symmetric matrix $\matrixoperator{C}$, is a first integral of  \eqref{eq::TestCases_ode}), then it is symplectic.
			\end{theorem}
			\begin{proof}
				See \cite{Hairer2006}.
			\end{proof}
			
			\begin{remark}
			It is important to note that  a quadratic first integral $I(\boldsymbol{y}) = \boldsymbol{y}^{t}\mathsf{C}\boldsymbol{y}$, with $\mathsf{C}$ constant, is conserved if:
			\begin{equation}
				\frac{\ederiv I}{\ederiv t} = \frac{\ederiv }{\ederiv t}\, \left(\boldsymbol{y}^{t}\mathsf{C}\boldsymbol{y}\right) = 2\boldsymbol{y}^{t}\mathsf{C}\frac{\ederiv \boldsymbol{y}}{\ederiv t} = 2\boldsymbol{y}^{t}\mathsf{C}\boldsymbol{h}(\boldsymbol{y}) = 0,\quad\forall \boldsymbol{y} \;. \label{invariant_first_integral}
			\end{equation}
			We stress that this equality holds for all points in the phase space.
			\end{remark}
			\begin{theorem}[\textbf{Mimetic canonical integrator, \eqref{eq::TestCases_primal_dual_integrator}, is symplectic}]
				The mimetic canonical integrator, \eqref{eq::TestCases_primal_dual_integrator}, is symplectic.
			\end{theorem}
			\begin{proof}
				Consider a system of ordinary differential equations such as in \eqref{eq::TestCases_ode} and take $\boldsymbol{y}_{h}$ as the approximated polynomial solution obtained by the mimetic canonical projection integrator. Assume $I(\boldsymbol{y}) = \boldsymbol{y}^{t}\mathsf{C}\boldsymbol{y}$ is a quadratic invariant of the system \eqref{eq::TestCases_ode}, see \eqref{invariant_first_integral}. Therefore:
				\[
					I(\boldsymbol{y}_{h}(t_{f})) - I(\boldsymbol{y}_{h}(t_{0})) = \int_{t_{0}}^{t_{f}} \frac{\ederiv I}{\ederiv t}\,\ederiv t = \int_{t_{0}}^{t_{f}}2\boldsymbol{y}_{h}^{t}\mathsf{C}\frac{\ederiv \boldsymbol{y}_{h}}{\ederiv t} \ederiv t\;.
				\]
				The integral term on the right is a polynomial of degree at most $(2p-1)$, which can be exactly integrated by Gauss quadrature, over $p$ collocation points (the Gauss nodes of the dual grid), $\tilde{t}^{\,k}$, using the integration weights $w_{k}$. Therefore, one can write:
				\[
					I(\boldsymbol{y}_{h}(t_{f})) - I(\boldsymbol{y}_{h}(t_{0})) = \sum_{k}\boldsymbol{y}_{h}\left(\tilde{t}^{\,k}\right)^{t}\mathsf{C}\frac{\ederiv \boldsymbol{y}_{h}}{\ederiv t}\left(\tilde{t}^{\,k}\right) w_{k}\;.
				\]
				By construction of the mimetic canonical integrator we have that:
				\[
					\frac{\ederiv \boldsymbol{y}_{h}}{\ederiv t}\left(\tilde{t}^{\,k}\right) = \boldsymbol{h}\left(\boldsymbol{y}_{h}\left(\tilde{t}^{\,k}\right)\right)\;,
				\]
				and consequently:
				\[
					I(\boldsymbol{y}_{h}(t_{f})) - I(\boldsymbol{y}_{h}(t_{0})) = \sum_{k}2\boldsymbol{y}_{h}\left(\tilde{t}^{\,k}\right)^{t}\mathsf{C}\frac{\ederiv \boldsymbol{y}_{h}}{\ederiv t}\left(\tilde{t}^{\,k}\right) w_{k} =  \sum_{k}2\boldsymbol{y}_{h}\left(\tilde{t}^{\,k}\right)^{t}\mathsf{C}\, \boldsymbol{h}\left(\boldsymbol{y}_{h}\left(\tilde{t}^{\,k}\right)\right) w_{k}\;.
				\]
				Equation \eqref{invariant_first_integral} states that  if the continuous system of equations conserves a quadratic first integral then:
				\[
					\boldsymbol{y}^{t}\mathsf{C}\boldsymbol{h}(\boldsymbol{y}) = 0, \quad \forall \boldsymbol{y}\,.
				\]
				Therefore all the terms in the summation term on the right hand side are zero:
				\[
					\boldsymbol{y}_{h}\left(\tilde{t}^{\,k}\right)^{t}\mathsf{C}\, \boldsymbol{h}\left(\boldsymbol{y}_{h}\left(\tilde{t}^{\,k}\right)\right)  = 0\,.
				\]
				This directly implies that:
				\[
					I(\boldsymbol{y}_{h}(t_{f})) - I(\boldsymbol{y}_{h}(t_{0})) = 0\;.
				\]
				Therefore quadratic first integrals are conserved by the mimetic canonical projection integrator. Since it is a Runge-Kutta method, by \theoremref{theorem::symplecticity_runge_kutta} it is symplectic.
			\end{proof}
		\subsubsection{Energy conservation of mimetic Galerkin integrator}
			Consider the case of autonomous Hamiltonian systems, where $\boldsymbol{y} = (\boldsymbol{p},\boldsymbol{q})$, $\boldsymbol{q}\in\mathbb{R}^{m}$, $\boldsymbol{p}\in\mathbb{R}^{m}$ and,

\begin{equation}
	\frac{\ederiv \boldsymbol{y}}{\ederiv t} = \mathsf{J}^{-1}\nabla H(\boldsymbol{y}),\quad \boldsymbol{y}(t_{0}) = \boldsymbol{y}^{0},\quad \boldsymbol{y}\in\mathcal{M}\subset\mathbb{R}^{2m}\quad \text{and} \quad t\in I\subset\mathbb{R} \quad \text{with}\quad \mathsf{J} = \left[\begin{array}{cc} \mathsf{0} & \mathsf{I}_{m} \\ -\mathsf{I}_{m}  & \mathsf{0}\end{array}\right]\;,\label{eq::hamiltonian_system_proof}
\end{equation}
where $\mathsf{I}_{m}$ is the $m\times m$ identity matrix. The solution obtained with the mimetic Galerkin integrator, \eqref{eq::TestCases_primal_primal_galerkin_integrator}, exactly preserves the energy of Hamiltonian systems. To prove that, we will follow a similar approach to \cite{hairerEnergyPreservingCollocation2010,Quispel2008}.
			
			\begin{theorem}
				The mimetic Galerkin integrator of \eqref{eq::TestCases_primal_primal_galerkin_integrator} exactly preserves the energy of Hamiltonian systems.
			\end{theorem}
			\begin{proof}
				From the fundamental theorem of calculus:
				\[
					H(\boldsymbol{y}_{h}(t_{f})) - H(\boldsymbol{y}_{h}(t_{0})) =\int_{\gamma(t)} \ederiv H(\boldsymbol{y}_{h}(t))\;,
				\]
				where $\gamma(t)$ is the path along which the line integral is computed. Using the chain rule and expanding the terms on the right hand side this can written as:
				\[
					H(\boldsymbol{y}_{h}(t_{f})) - H(\boldsymbol{y}_{h}(t_{0})) =\sum_{i=1}^{2m}\int_{t_{0}}^{t_{f}} \left(\frac{\ederiv \boldsymbol{y}_{i,h}}{\ederiv t}\right)^{t} \, \frac{\partial H}{\partial y_{i}}(\boldsymbol{y}_{h}(t))\,\ederiv t\;,
				\]
				In differential form notation this can be written as:
				\begin{equation}
					H(\boldsymbol{y}_{h}(t_{f})) - H(\boldsymbol{y}_{h}(t_{0})) =\sum_{i=1}^{2m}\int_{\manifold{T}} \star\ederiv y^{(0)}_{i,h}\wedge \star\left(\frac{\partial H}{\partial y_{i}}(\boldsymbol{y}_{h})\right)^{(0)}\;, \label{eq::ode_fundamental_theorem_calculus_diff_geom}
				\end{equation}
				where $\manifold{T} = [-1,1]$. Since:
				\begin{equation}
					\sum_{j=1,l=1,k=0}^{p} \incidencederivative{0}{1}_{k,l}y_{i}^{k}e_{l}(\tilde{\tau}^{j})\, \kdifform{\tilde{\epsilon}_{j}}{0}(\tau) =  \star\ederiv y^{(0)}_{i,h},\quad i=1,\dots,2m\;,\label{eq::ode_galerking_proof_01}
				\end{equation}
			      using \eqref{eq::TestCases_primal_primal_galerkin_integrator} and noting that for Hamiltonian systems we have that $\boldsymbol{h}^{(0)} = \left(\mathsf{J}^{-1}\,\nabla H\right)^{(0)}$, and therefore,
			      \[
			      	h_{i}^{(0)} = \left(\mathsf{J}^{-1}\,\nabla H\right)^{(0)}_{i} = \left(\sum_{j=1}^{2m}\mathsf{J}^{-1}_{ij}\,\frac{\partial H}{\partial y_{j}}\right)^{(0)}\,,
				\]
				it is possible to write:
			      \begin{equation}
				\sum_{j=1,l=1,k=0}^{p} \incidencederivative{0}{1}_{k,l}y_{i}^{k} e_{l}(\tilde{\tau}^{j}) \int_{\manifold{T}}\tilde{\epsilon}_{j}^{(0)}\wedge\star\tilde{\epsilon}_{r}^{(0)}= \int_{\manifold{T}}\left(\mathsf{J}^{-1}\,\nabla H\right)^{(0)}_{i}\wedge\star\kdifform{\tilde{\epsilon}_{r}}{0}\;, \label{eq::TestCases_primal_primal_galerkin_integrator_proof_energy}
			\end{equation}
			where $r=1,\dots,p$ and $i=1,\dots,2m$. We then notice that:
			\begin{equation}
				\mathsf{N}_{i,j} := \int_{\manifold{T}}\kdifform{\tilde{\epsilon}_{i}}{0}\wedge\star\kdifform{\tilde{\epsilon}_{j}}{0} = \int_{-1}^{1} \tilde{l}_{i}^{g}(\tau)\tilde{l}_{j}^{g}(\tau)\sqrt{g}\,\ederiv\tau \;. \label{eq::ode_matrix_N_before}
			\end{equation}
			This integral can be exactly evaluated by Gaussian quadrature. Since $\tilde{l}^{g}_{i}(\tau)$ is a polynomial of degree $(p-1)$, the product will be a polynomial of degree at most $2(p-1)$. Using Gaussian quadrature over the $p$ Gauss nodes associated to the 0-forms $\kdifform{\tilde{\epsilon}_{k}}{0}$, $\tilde{\tau}^{j}$, the quadrature is exact for polynomials up to degree $(2p-1)$ and therefore exactly evaluates this integral. In this way, the matrix $\mathsf{N}$ becomes:
			\begin{equation}
				\mathsf{N}_{i,j} = \int_{-1}^{1} \tilde{l}_{i}^{g}(\tau)\tilde{l}_{j}^{g}(\tau)\sqrt{g}\,\ederiv\tau = \sum_{k=1}^{p}\tilde{l}^{g}_{i}(\tilde{\tau}^{k})\tilde{l}^{g}_{j}(\tilde{\tau}^{k})\,w_{k} = \sum_{k=1}^{p}\delta_{i,k}\delta_{j,k}\,w_{k} =\delta_{i,j}\,w_{j}\;. \label{eq::ode_matrix_N}
			\end{equation}

			Which means that $\mathsf{N}$ is a diagonal matrix and therefore $\mathsf{N}^{-1}$ is also a diagonal matrix whose elements are given by $\left(\mathsf{N}^{-1}\right)_{i,j}=\delta_{i,j}\,w_{j}^{-1}$. If we substitute this into \eqref{eq::TestCases_primal_primal_galerkin_integrator_proof_energy}, we get:
			 \begin{equation}
				\sum_{l=1,k=0}^{p} \incidencederivative{0}{1}_{k,l}y_{i}^{k} e_{l}(\tilde{\tau}^{j}) = \sum_{r=1}^{p}\delta_{j,r}\,w_{r}^{-1}\int_{\mathcal{T}}\left(\mathsf{J}^{-1}\,\nabla H\right)^{(0)}_{i}\wedge\star\kdifform{\tilde{\epsilon}_{r}}{0}\;, \label{eq::TestCases_primal_primal_galerkin_integrator_proof_energy_2}
			\end{equation}
			with $j=1,\dots,p$ and $i=1,\dots,2m$.
						
			Now, substituting \eqref{eq::TestCases_primal_primal_galerkin_integrator_proof_energy_2} in \eqref{eq::ode_galerking_proof_01} yields:
			\begin{equation}
				\sum_{j=1}^{p} \left(\sum_{r=1}^{p}\delta_{j,r}\,w_{r}^{-1}\int_{\mathcal{T}}\left(\mathsf{J}\,\nabla H\right)^{(0)}_{i}\wedge\star\kdifform{\tilde{\epsilon}_{r}}{0}\right)\, \kdifform{\tilde{\epsilon}_{j}}{0} =  \star\ederiv y_{h,i},\quad i=1,\dots,2m\,.
			\end{equation}
			Substituting this expression into the right hand side of \eqref{eq::ode_fundamental_theorem_calculus_diff_geom} yields:
			\begin{equation}
				H(\boldsymbol{y}_{h}(t_{f})) - H(\boldsymbol{y}_{h}(t_{0}))  = \sum_{i=1,k=1}^{2m}\int_{\mathcal{T}} \left[\sum_{j=1,r=1}^{p} \left(\delta_{jr}\,w_{r}^{-1}\int_{\mathcal{T}}\left(\mathsf{J}^{-1}\nabla H\right)^{(0)}_{i}\wedge\star\kdifform{\tilde{\epsilon}_{r}}{0}\right)\, \kdifform{\tilde{\epsilon}_{j}}{0}\right]\wedge \star\left(\nabla H\right)^{(0)}_{i}\,.
			\end{equation}
			Rearranging the terms and summing in $j$, the following expression is obtained:
			\begin{equation}
				H(\boldsymbol{y}_{h}(t_{f})) - H(\boldsymbol{y}_{h}(t_{0}))  = \sum_{i=1,k=1}^{2m}\sum_{m=1}^{p} w_{m}^{-1}\,\left(\int_{\mathcal{T}}\left(\nabla H\right)^{(0)}_{k}\wedge\star\kdifform{\tilde{\epsilon}_{m}}{0}\right)\mathsf{J}_{ik}\left(\int_{\mathcal{T}}\kdifform{\tilde{\epsilon}_{m}}{0}\wedge \star\left(\nabla H\right)^{(0)}_{i}\right)\;.
			\end{equation}
			The right hand side of this expression is equal to zero due to the skew-symmetry of matrix $\mathsf{J}$, this equality becomes:
			\[
				H(\boldsymbol{y}_{h}(t_{f})) - H(\boldsymbol{y}_{h}(t_{0})) = 0\;,
			\]
			and consequently the mimetic Galerkin integrator is exactly conserves the energy of the Hamiltonian system.
		\end{proof}
			
			\begin{definition}[\textbf{Symmetric integrator}]
				Take the system of ordinary differential equations in \eqref{eq::TestCases_ode}. The numerical method $\Phi_{\Delta t}:\boldsymbol{y}_{0}\in\mathbb{R}^{n}\mapsto\boldsymbol{y}_{1}\in\mathbb{R}^{n}$, which approximates this system of equations, maps points $\boldsymbol{y}_{0}$ in the instant $t_{0}$ to points $\boldsymbol{y}_{1}$ in the time instant $t_{1}=t_{0}+\Delta t$. The numerical method is \emph{symmetric} if $\Phi_{-\Delta t}\circ\Phi_{\Delta t} = I$.
			\end{definition}
			One can see that the canonical and Galerkin integrators are symmetric. For that recall \eqref{eq::TestCases_primal_dual_integrator_one_variable} and \eqref{eq::TestCases_primal_primal_galerkin_integrator_one_variable} and note that substituting $y_{h}(-t)$ in \eqref{eq::TestCases_primal_dual_integrator_one_variable} and \eqref{eq::TestCases_primal_primal_galerkin_integrator_one_variable} it solves both systems of equations for $-\Delta t$.

		\section{Numerical tests}\label{section::numerical_results}
				In order to illustrate the properties of the mimetic integrators just presented  we apply them to the solution of four test cases. We have chosen the circular motion problem (a Hamiltonian linear problem), the Lotka-Volterra problem (a non-Hamiltonian, non-linear polynomial problem), the pendulum problem (a Hamiltonian, non-linear and non-polynomial problem) and finally the two-body Kepler problem (a higher order Hamiltonian problem).

			\subsection{Circular motion problem}
				The mimetic integrators were applied to a circular motion problem, a linear problem,
			\begin{equation}
				\left\{
					\begin{array}{l}
						\frac{\ederiv p}{\ederiv t} = q\\
						\frac{\ederiv q}{\ederiv t} = -p
					\end{array}
				\right.\;, \label{eq::circle}
			\end{equation}
			whose Hamiltonian is given by $H(p,q)=\frac{1}{2}\left(q^{2}+p^{2}\right)$.

			In \figref{fig:time_circle_phase} one can see the numerical solution of the circular motion problem, \eqref{eq::circle}, using the two mimetic integrators, with polynomial order in time $p_{t}=2$, a Runge-Kutta of 4th order and the simple explicit Euler. From this figure one can see that both mimetic integrators are able to perform a long time simulation, as they retain the behavior of the analytical solution. On the other hand, as expected, the explicit Euler scheme rapidly spirals outwards and the Runge-Kutta of 4th order quickly spirals inwards reaching the stable point $(0,0)$. This is an important example that stresses the relevance of symplectic and energy preserving integrators. Although the Runge-Kutta integrator has the same convergence rate of the mimetic integrators, see \figref{fig:time_circle_error} (right), it completely fails to capture the circular motion of this system of equations. Both mimetic integrators exactly (up to machine precision) preserve the constant of motion $R^{2} = q^{2} + p^{2}$. In the case of  the mimetic canonical integrator this happens because this quantity is a quadratic expression of the variables $q$ and $p$, and this integrator is symplectic. For the mimetic Galerkin integrator this happens because this quantity differs from the Hamiltonian of the system by only a multiplicative constant. An interesting point to note is that, in this case, due to the linearity of the system, both mimetic integrators yield the same result, up to machine accuracy, as can be seen in \figref{fig:time_circle_mgi2_mci2}. Another relevant point to note is that, although the trajectory computed from both mimetic integrators is of order $p_{t}=2$ they show a convergence rate of $2p_{t}$.

			\begin{figure}[ht]
				\centering
					\subfigure{
				\includegraphics[width=0.35\textwidth]{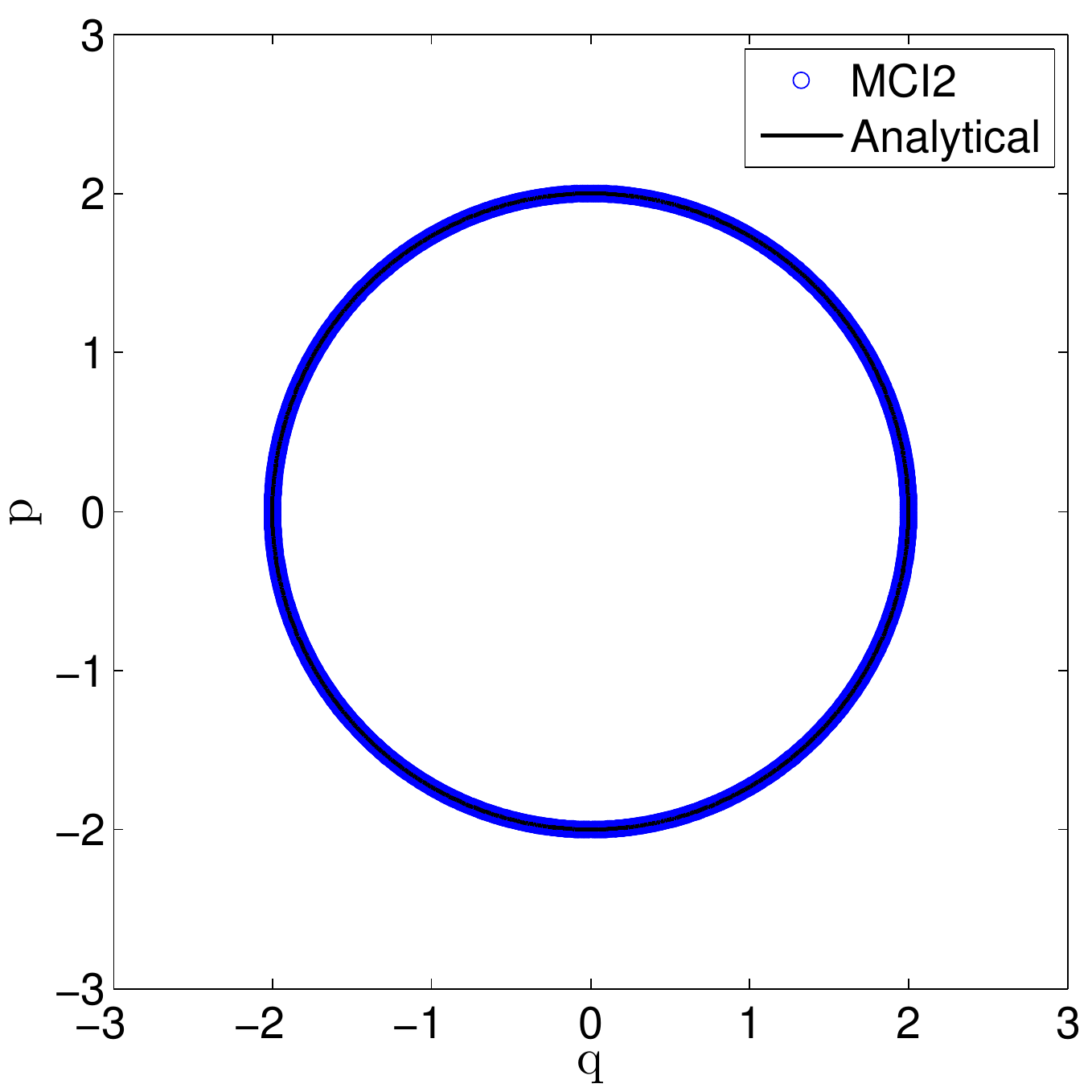}
						\label{fig:time_circle_phase_mci2}
					}
					\subfigure{
				\includegraphics[width=0.35\textwidth]{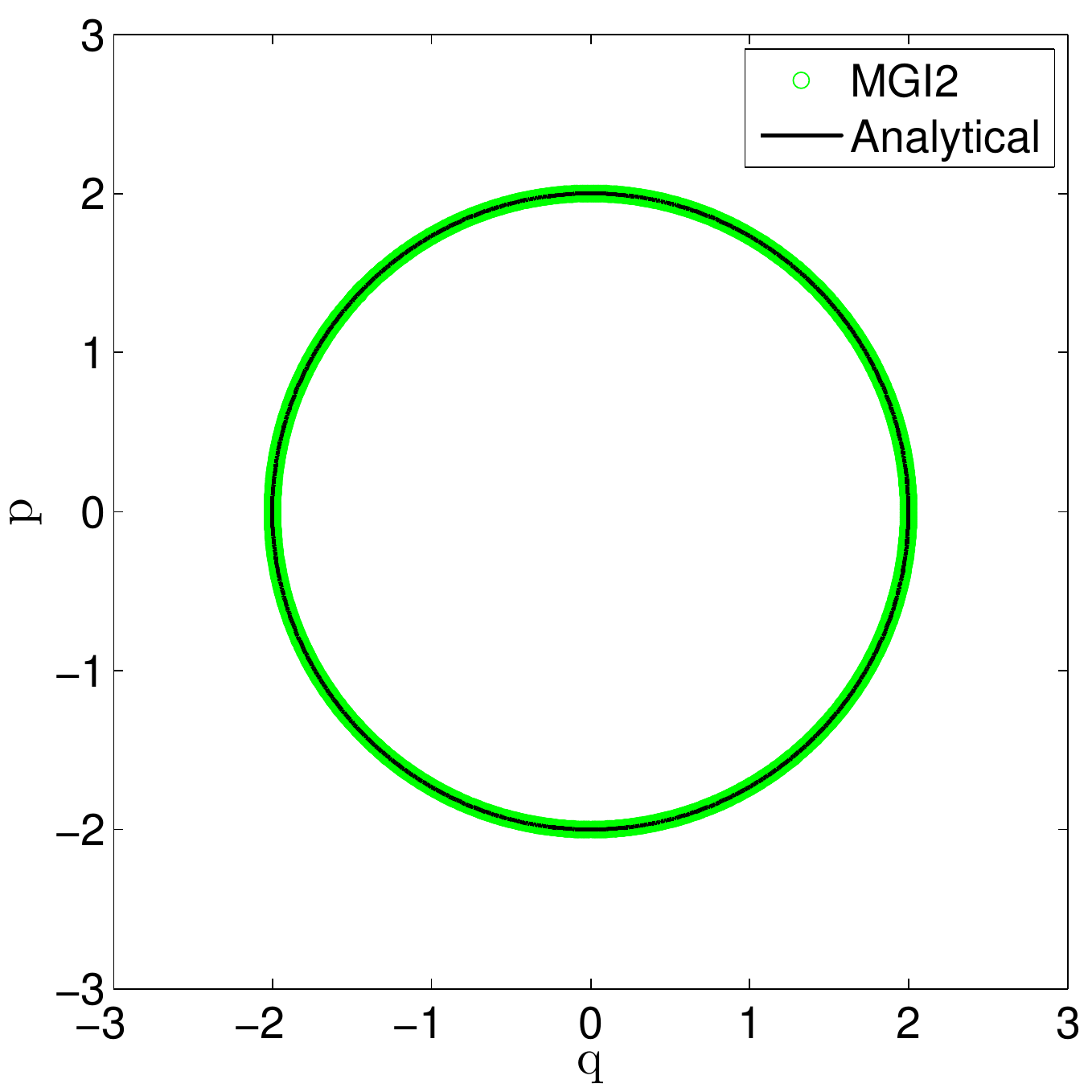}
						\label{fig:time_circle_phase_mgi2}
					}
					\subfigure{
				\includegraphics[width=0.35\textwidth]{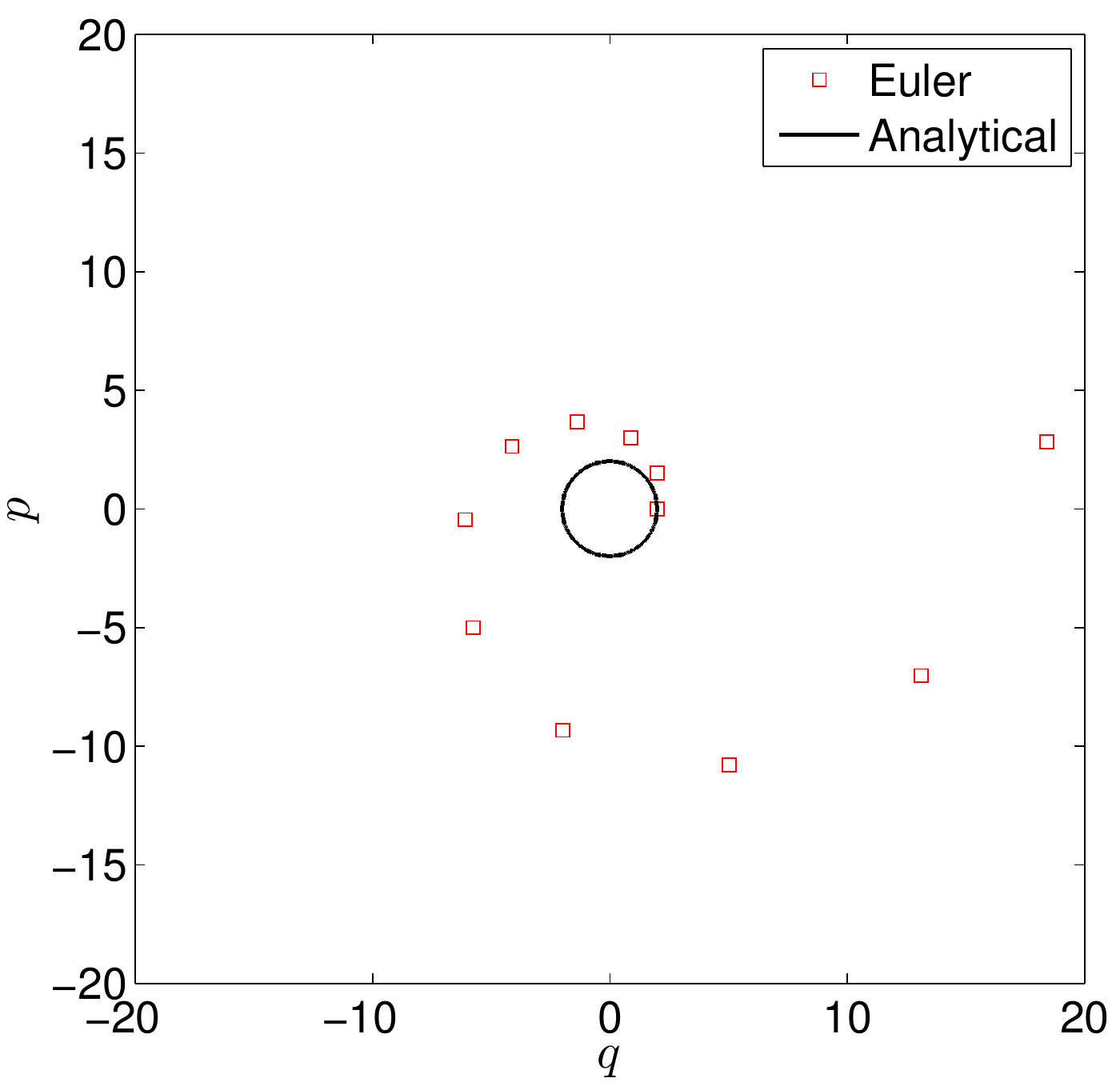}
						\label{fig:time_circle_phase_euler}
					}
					\subfigure{
				\includegraphics[width=0.35\textwidth]{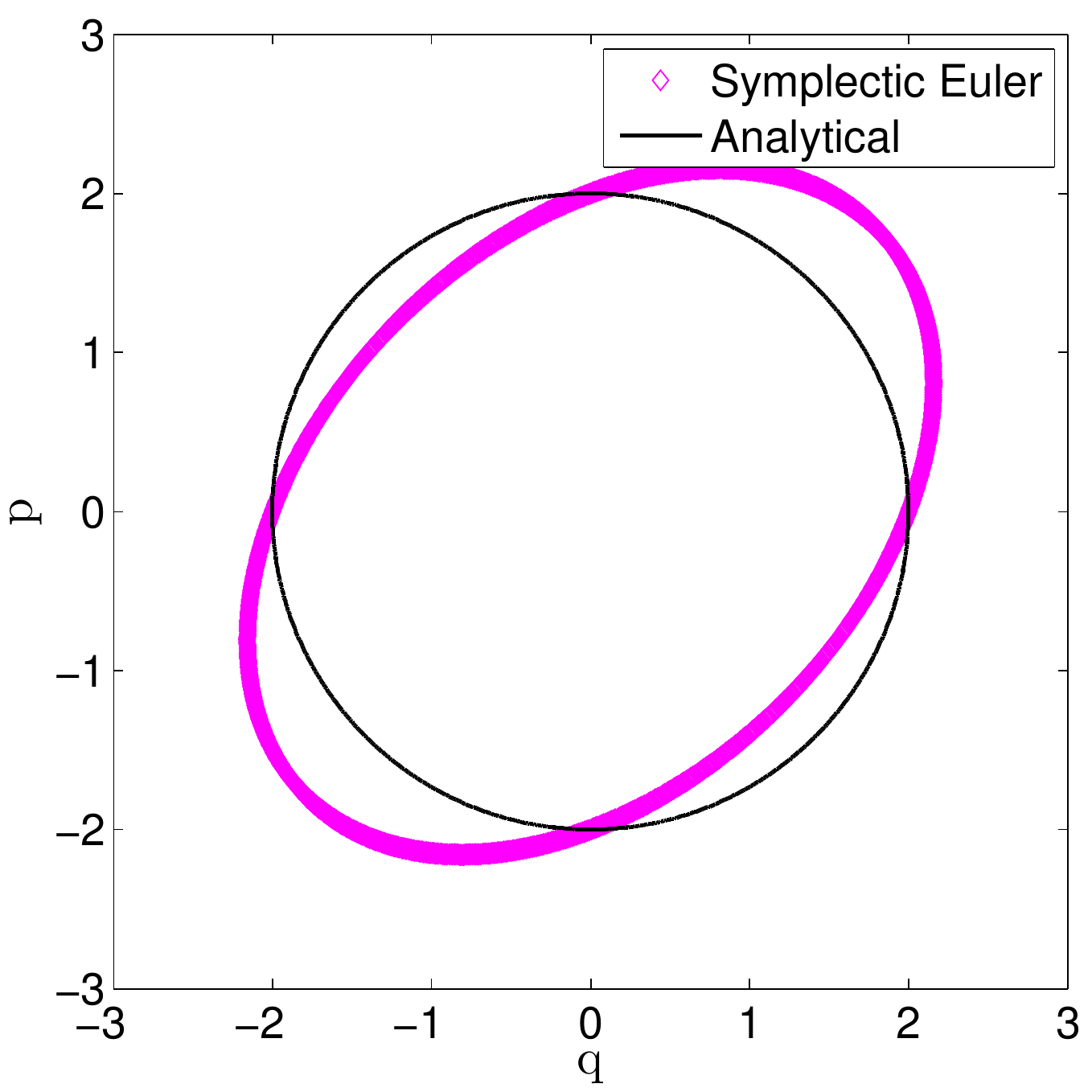}
						\label{fig:time_circle_phase_SE}
					}
					\subfigure{
				\includegraphics[width=0.35\textwidth]{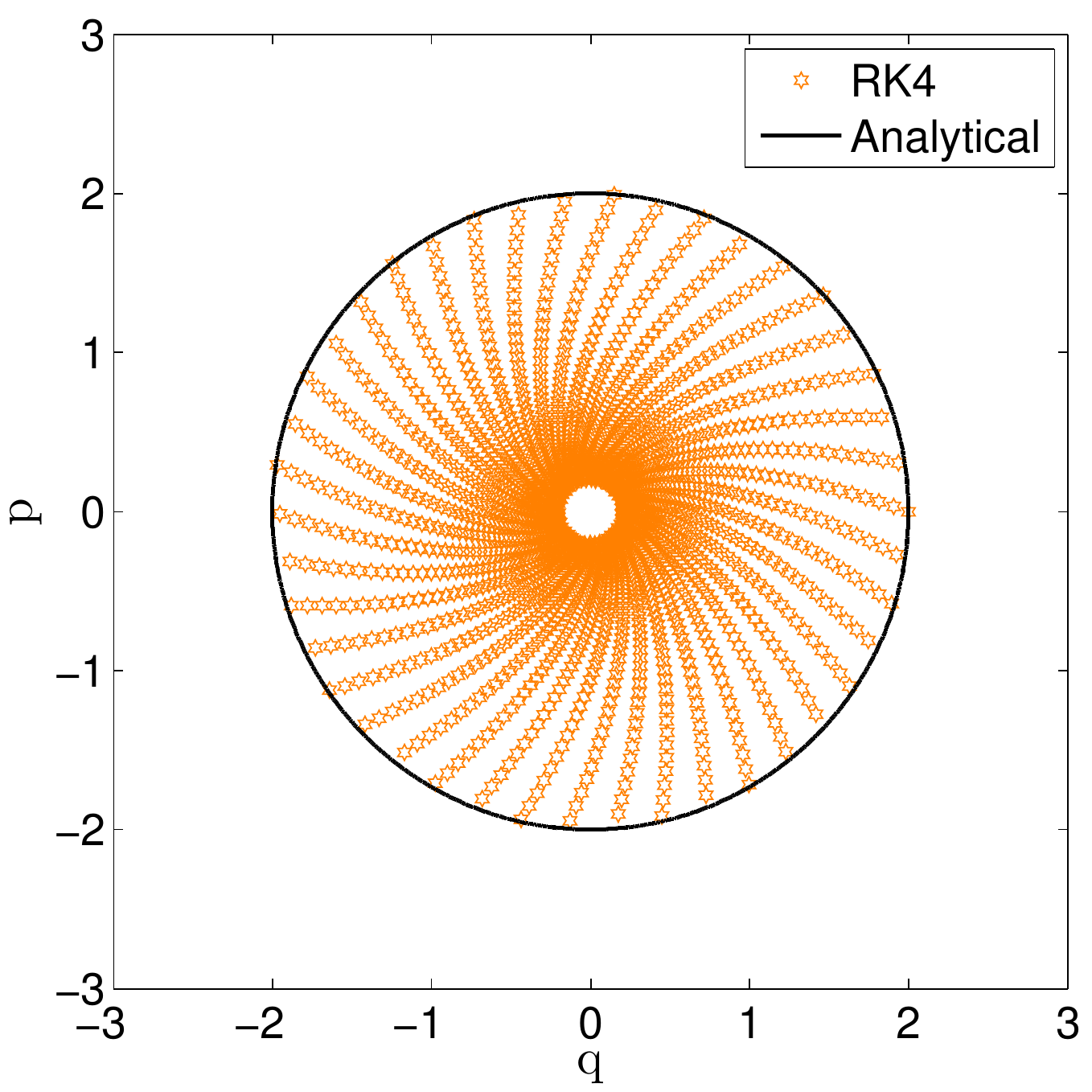}
						\label{fig:time_circle_phase_RK4}
					}
					\caption{Phase-space numerical (in color) and analytical (in black) solutions of circular motion, \eqref{eq::circle}, with initial condition $y_{1}=2$, $y_{2}=0$ and time step $\Delta t = 1s$. Top left: mimetic canonical integrator with poynomial order in time $p_{t}=2$ (MCI2). Top right: mimetic Galerkin integrator with polynomial order in time $p_{t}=2$ (MGI2). Center left: explicit Euler scheme, which clearly diverges. Center right: symplectic Euler scheme, which clearly diverges. Bottom: explicit Runge-Kutta of 4th order, which also clearly diverges.}
					\label{fig:time_circle_phase}
				\end{figure}
				
				\begin{figure}[ht]
					\centering
					\subfigure{
								\includegraphics[width=0.44\textwidth]{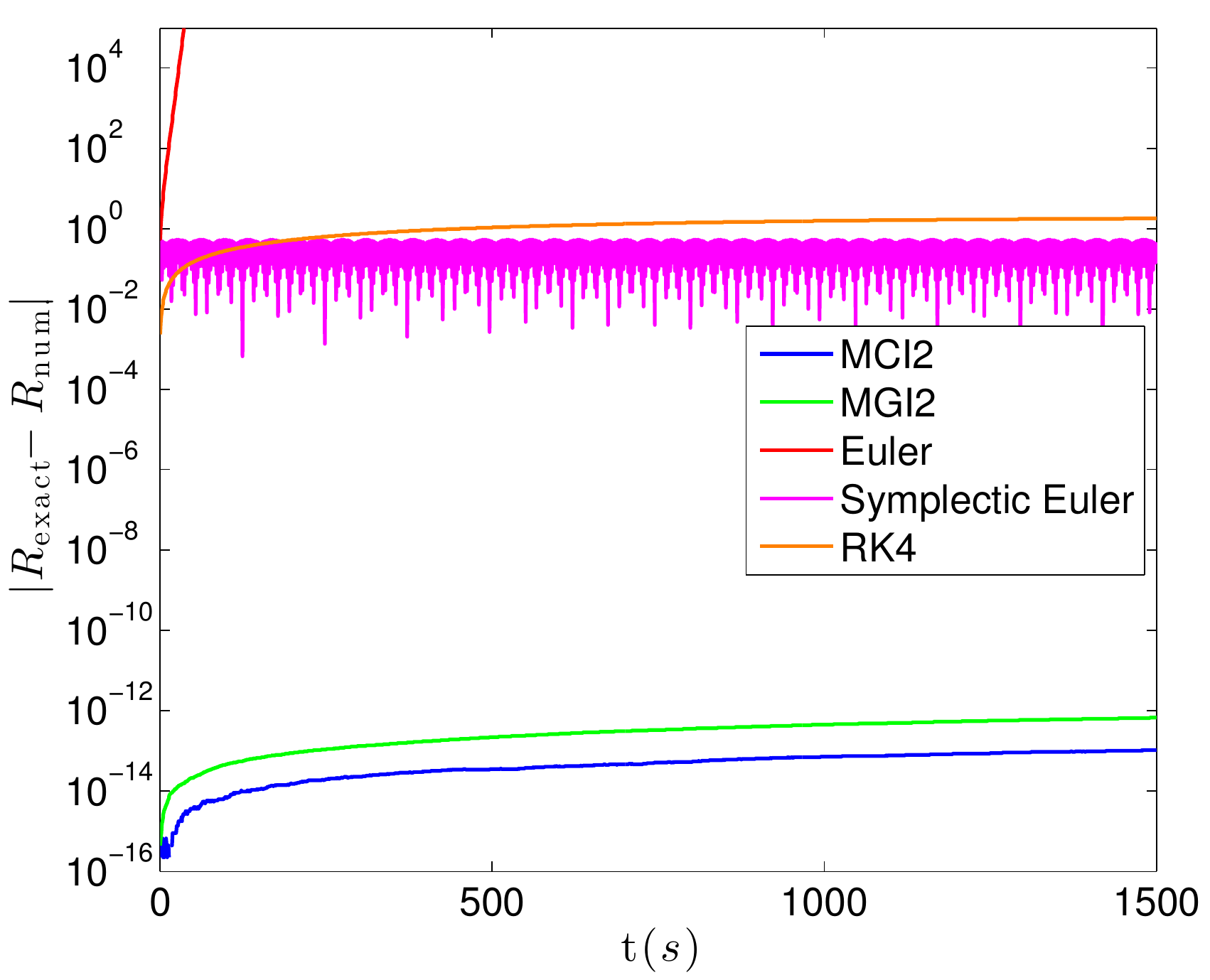}
						\label{fig:time_circle_error_time}
					}
					\subfigure{
								\includegraphics[width=0.44\textwidth]{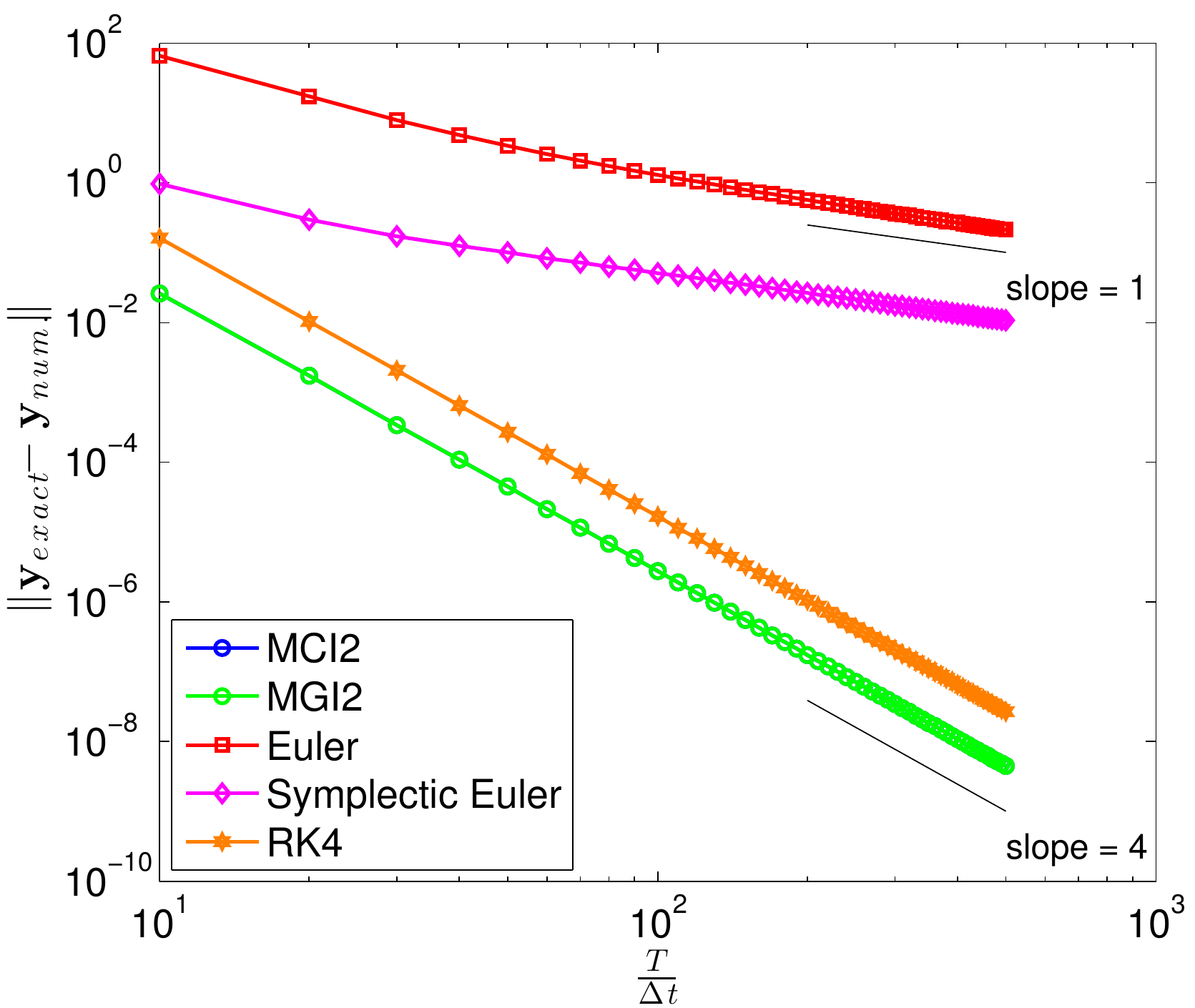}

						\label{fig:time_circle_error_convergence}
					}
					\caption{Left: Error in time of  $R=(y_{1}^{2} + y_{2}^{2})^{\frac{1}{2}}$ for the numerical solution of the circular motion, \eqref{eq::circle}, for $\Delta t = 1s$, with the mimetic canonical integrator of polynomial order in time $p_{t}=2$ (MCI2), mimetic Galerkin integrator of polynomial order in time $p_{t}=2$ (MGI2), Runge-Kutta of 4th order, explicit Euler scheme and symplectic Euler method. Since the quantity $R$ is a function (square root) of a quadratic term, the mimetic canonical integrator exactly (up to machine accuracy) preserves this quantity along the trajectory, as can be seen. The same can be seen for the mimetic Galerkin integrator, since $\frac{1}{2}\left(y_{1}^{2} + y_{2}^{2}\right)$ is also the Hamiltonian of the system. Right: Convergence of the error of the path in phase space as a function of the time step ($\frac{T}{\Delta t}$) for mimetic canonical integrator  of polynomial order in time $p_{t}=2$ (MCI2), mimetic Galerkin integrator of polynomial order in time $p_{t}=2$ (MGI2), Runge-Kutta of 4th order,  explicit Euler and symplectic Euler. It is possible to observe the 4th order rate of convergence of the mimetic integrators.}
					\label{fig:time_circle_error}
					\end{figure}

				\begin{figure}[ht]
					\centering
					\includegraphics[width=0.45\textwidth]{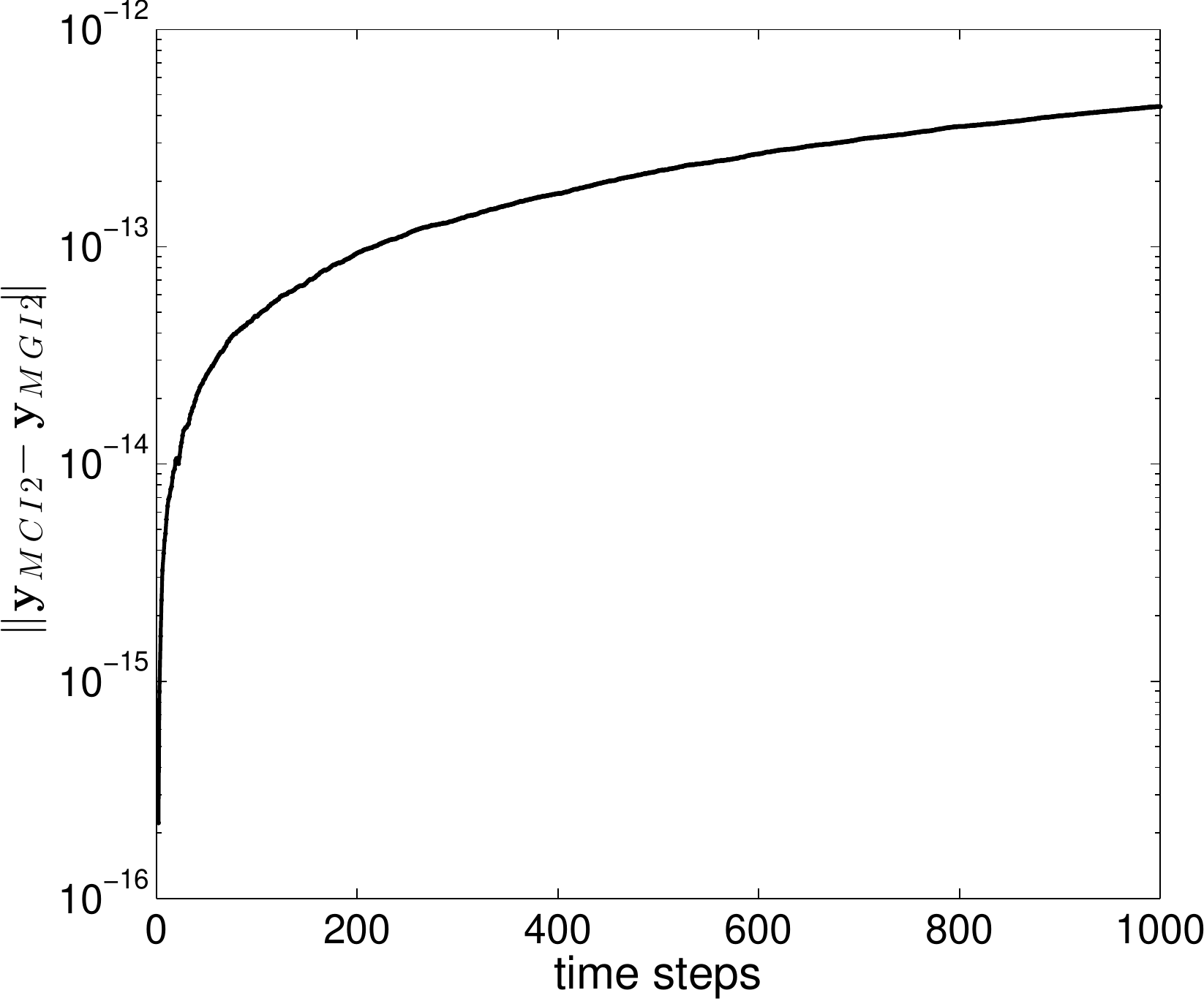}
					\caption{Difference between the numerical solution of the circular motion, \eqref{eq::circle}, for $\Delta t = 1s$, obtained with the mimetic canonical integrator of polynomial order in time $p_{t}=2$ (MCI2) and the mimetic Galerkin integrator of polynomial order in time $p_{t}=2$ (MGI2). For linear systems of ordinary differential equations, both solutions are equal up to machine accuracy. In this particular case, both integrators give the same result. This is due to the fact that the in system of equations \eqref{eq::circle} is linear.}
					\label{fig:time_circle_mgi2_mci2}
				\end{figure}

				\FloatBarrier

			\subsection{Lotka-Volterra problem}
				The Lotka-Volterra problem is governed by the following equations,
			\begin{equation}
				\left\{
					\begin{array}{l}
						\frac{\ederiv y_{1}}{\ederiv t} = y_{1}(y_{2}-2)\\
						\frac{\ederiv y_{2}}{\ederiv t} = y_{2}(1-y_{1})
					\end{array}
				\right. \;, \label{eq::lotka_volterra}
			\end{equation}
			a non-linear polynomial problem, which is not a Hamiltonian system but conserves the quantity $V(y_{1},y_{2})=-y_{1}+\log y_{1}-y_{2}+2\log y_{2}$ along its trajectories.

			The Lotka-Volterra system of equations, \eqref{eq::lotka_volterra}, is an interesting problem to analyze since it is neither Hamiltonian (it is a Poisson system) nor has a conserved quadratic quantity. Here one can see that the trajectories obtained with both mimetic integrators and the Runge-Kutta of 4th order are very similar, see \figref{fig:time_lotkaVolterra_phase}, and the error as a function of time of the conserved quantity $V=-y_{1}+\log y_{1}-y_{2}+2\log y_{2}$, see \figref{fig:time_lotkaVolterra_error} (left). A similar behaviour is seen for the convergence plots as a function of time step size, see \figref{fig:time_lotkaVolterra_error} (right). In this case the advantage of using the mimetic integrators when compared to Runge-Kutta of 4th order does not show itself.

			\begin{figure}[ht]
				\centering
					\subfigure{
				\includegraphics[width=0.35\textwidth]{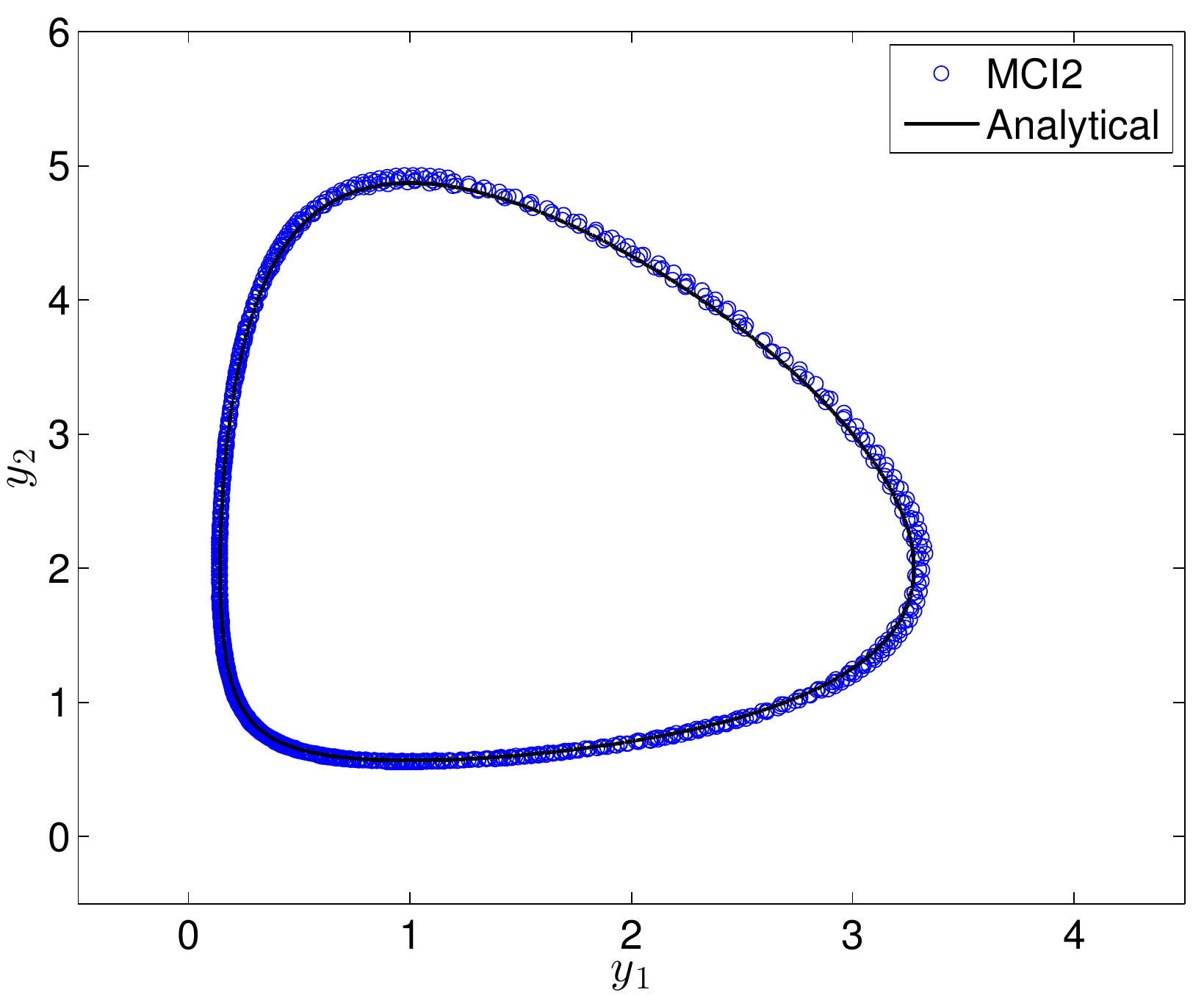}
						\label{fig:time_lotkaVolterra_phase_mci2}
					}
					\subfigure{
				\includegraphics[width=0.35\textwidth]{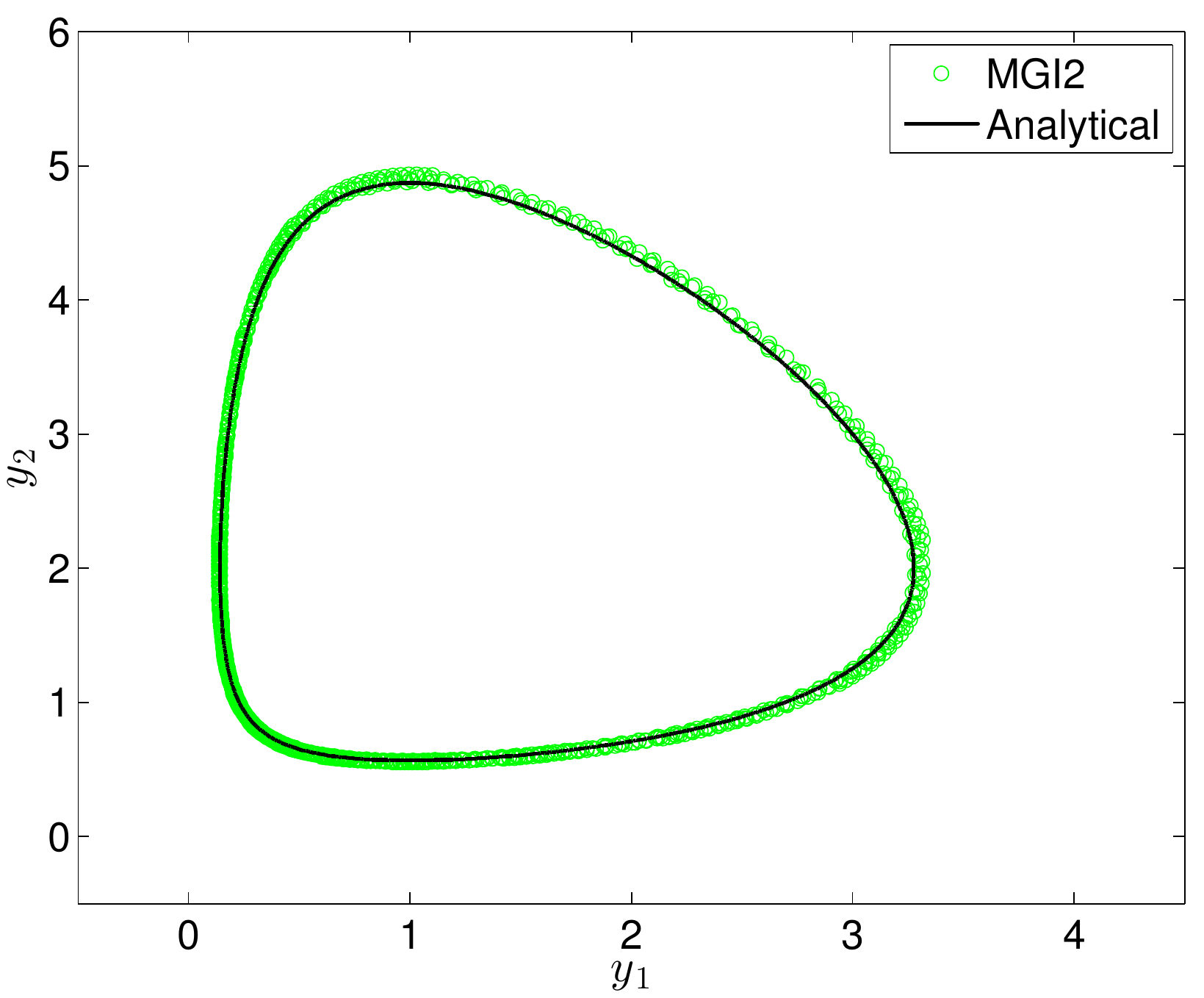}
						\label{fig:time_lotkaVolterra_phase_mgi2}
					}
					\subfigure{
				\includegraphics[width=0.35\textwidth]{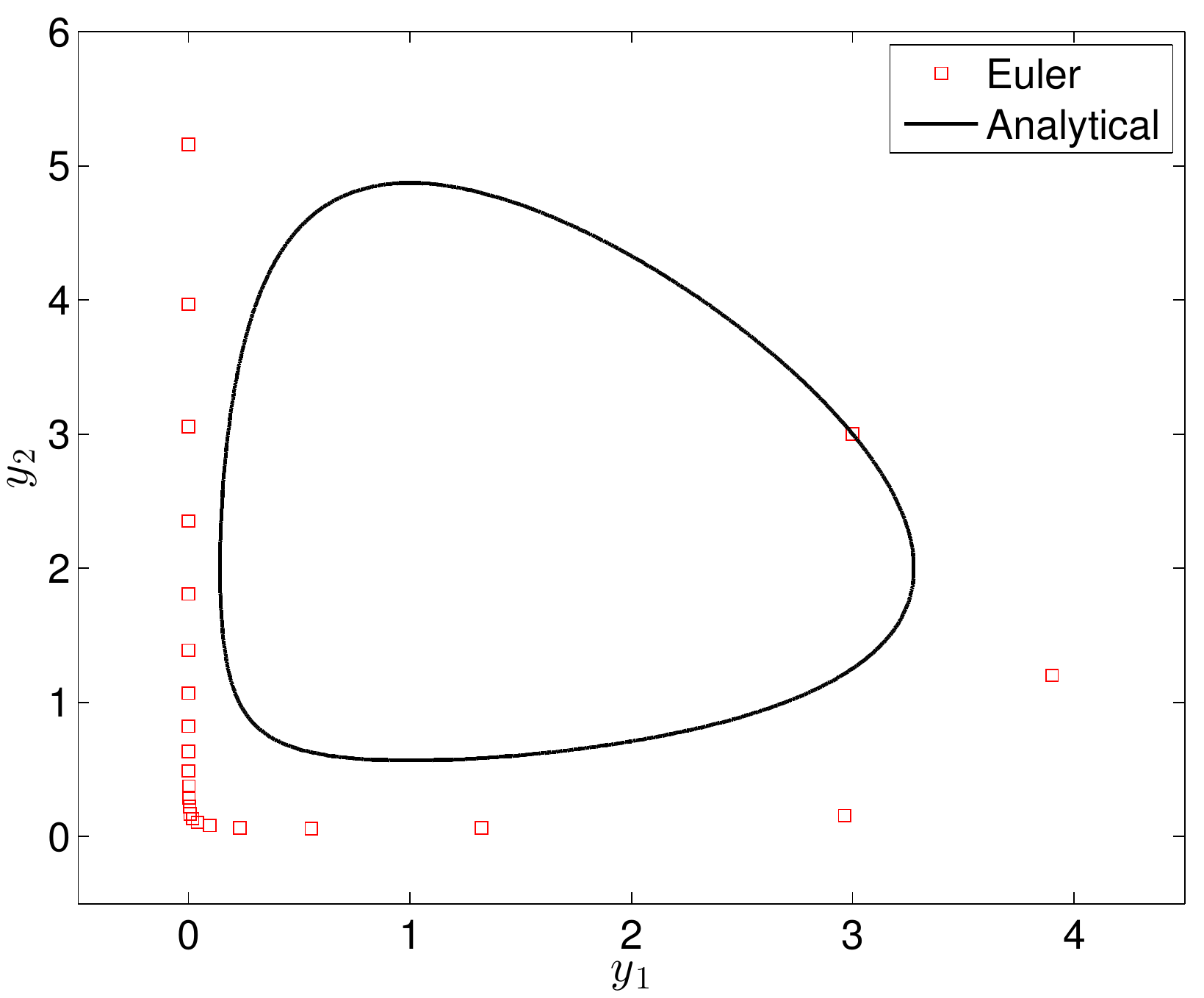}
						\label{fig:time_lotkaVolterra_phase_euler}
					}
					\subfigure{
				\includegraphics[width=0.35\textwidth]{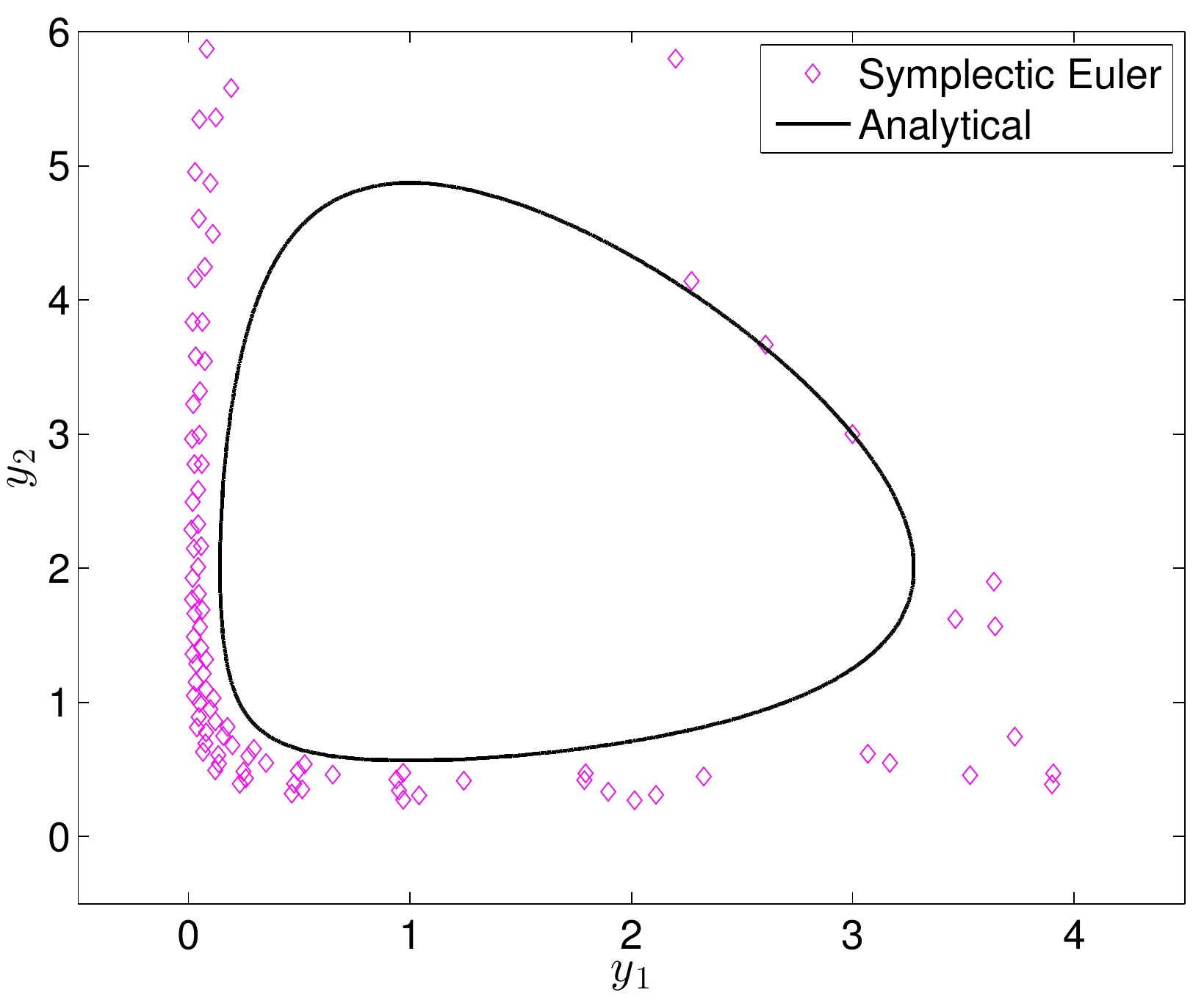}
						\label{fig:time_lotkaVolterra_phase_symplectic_euler}
					}
					\subfigure{
				\includegraphics[width=0.35\textwidth]{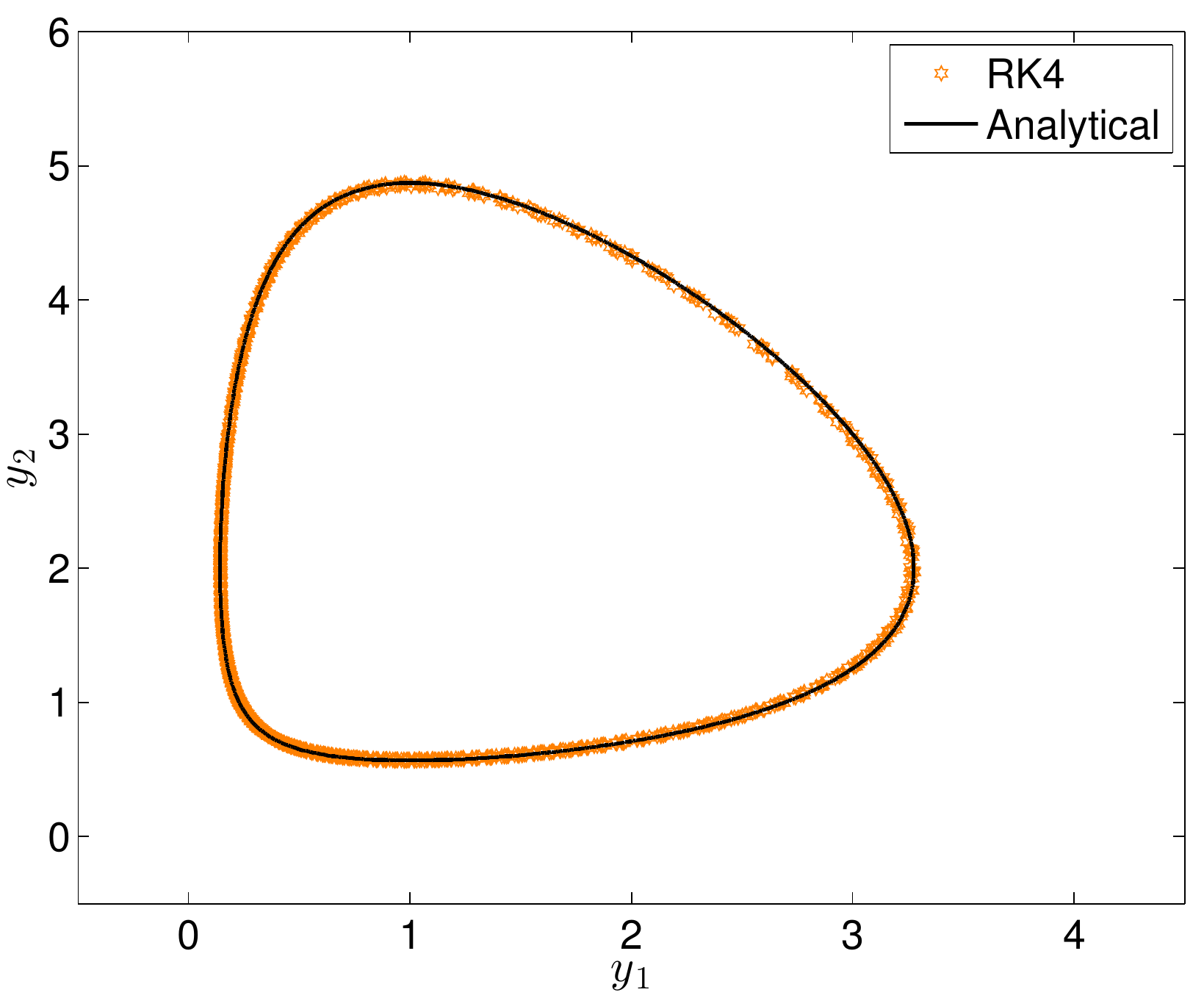}
						\label{fig:time_lotkaVolterra_phase_RK4}
					}
					\caption{Phase-space numerical (in color) and analytical (in black) solutions of Lotka-Volterra system, \eqref{eq::lotka_volterra} with initial condition $y_{1}=3$, $y_{2}=3$ and time step $\Delta t = 0.3s$. Top left: mimetic canonical integrator with polynomial order in time $p_{t}=2$ (MCI2). Top right: mimetic Galerkin integrator with polynomial order in time $p_{t}=2$ (MGI2). Center left: explicit Euler scheme, where one can clearly see the fast divergence of the solution. Center right: symplectic Euler scheme, where one can also see the divergence of the solution. Bottom: explicit Runge-Kutta of 4th order.}
					\label{fig:time_lotkaVolterra_phase}
				\end{figure}

				\begin{figure}[ht]
					\centering
					\subfigure{

								\includegraphics[width=0.44\textwidth]{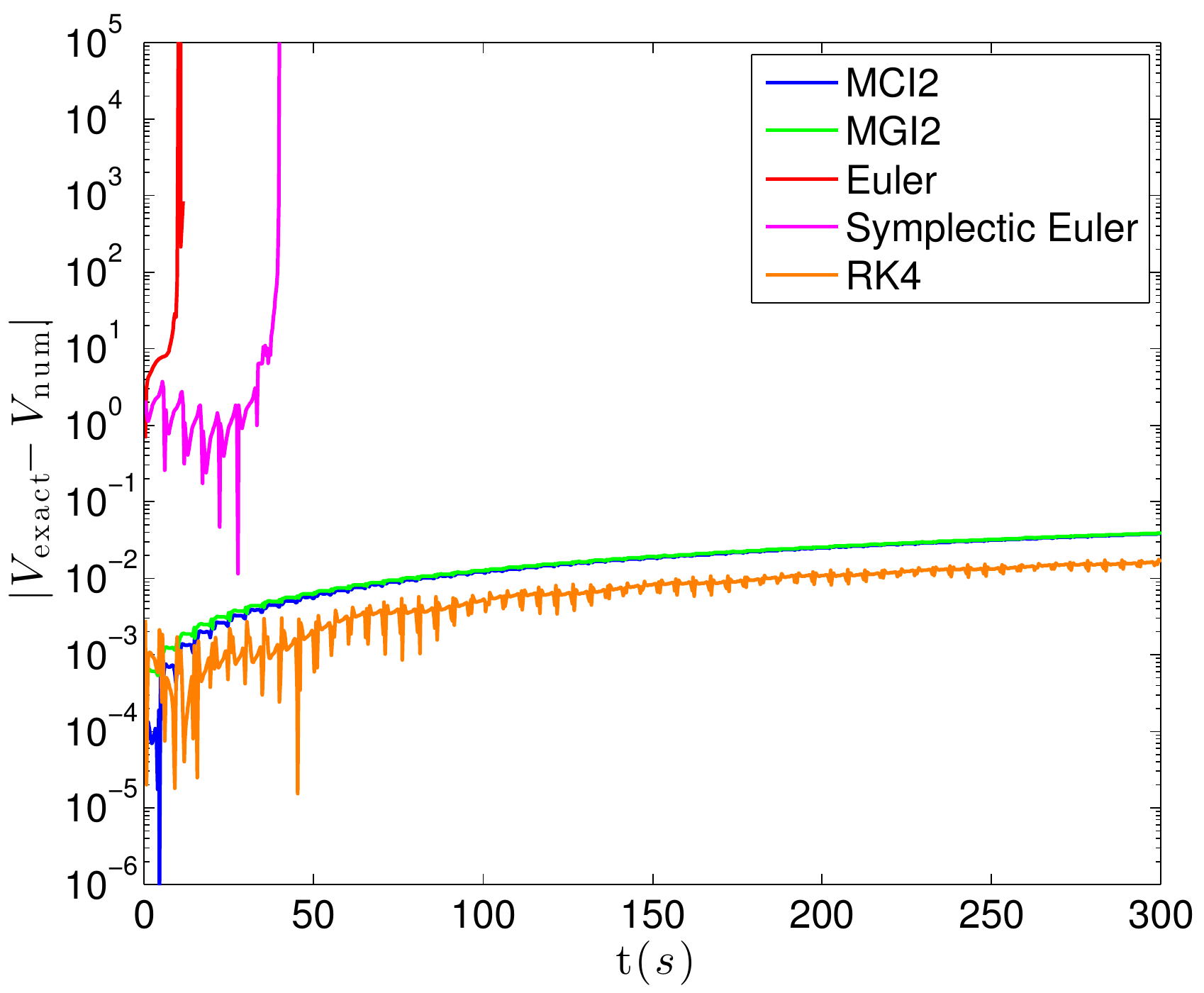}
						\label{fig:time_lotkaVolterra_error_time}
					}
					\subfigure{
								\includegraphics[width=0.44\textwidth]{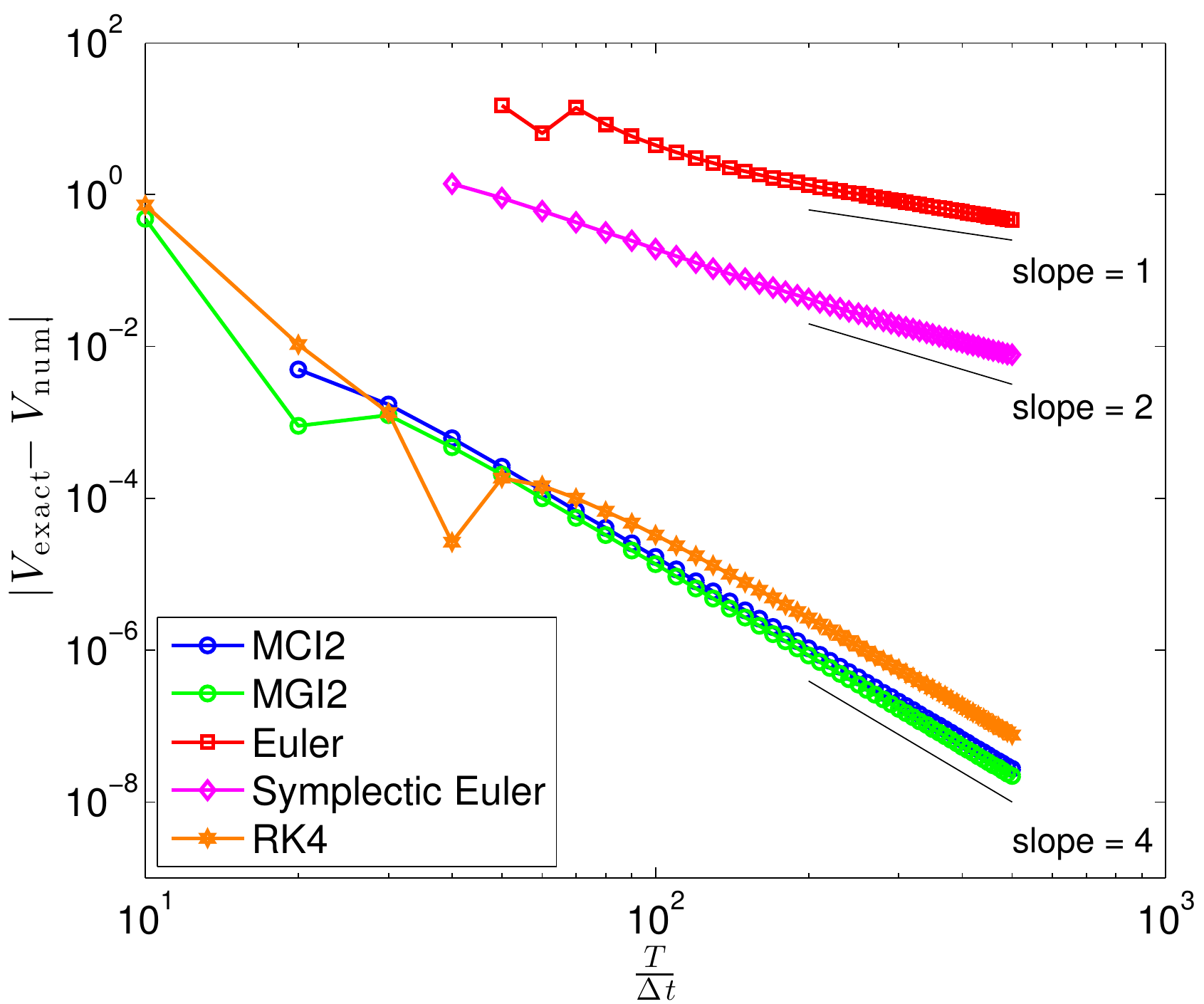}
						\label{fig:time_lotkaVolterra_error_convergence}
					}
					\caption{Left: Error in time of the conserved quantity  $V = -y_{1}+\log y_{1} - y_{2} + 2\log y_{2}$ for the numerical solution of the Lotka-Volterra system of equations, \eqref{eq::lotka_volterra}, for $\Delta t = 0.3s$, with the mimetic canonical integrator of polynomial order in time $p_{t}=2$ (MCI2), mimetic Galerkin integrator of polynomial order in time $p_{t}=2$ (MGI2), Runge-Kutta of 4th order, explicit Euler scheme and symplectic Euler method. Since this quantity $V$ is not a quadratic term, the mimetic canonical integrator is unable to exactly preserve this quantity along the trajectory, as can be seen by the increase in the error with time. The same can be seen for the mimetic Galerkin projection integrator, since $V$ is not the Hamiltonian of the system. Right: Convergence of the error of the path in phase space as a function of $\frac{T}{\Delta t}$ for mimetic canonical integrator  of polynomial order in time $p_{t}=2$ (MCI2), mimetic Galerkin integrator of polynomial order in time $p_{t}=2$ (MGI2), Runge-Kutta of 4th order, explicit Euler and symplectic Euler. It is possible to observe the 4th order rate of convergence of the mimetic integrators.}
					\label{fig:time_lotkaVolterra_error}
					\end{figure}

					\FloatBarrier

			\subsection{Pendulum problem}
				The pendulum problem is a non-linear, non-polynomial problem,
			\begin{equation}
				\left\{
					\begin{array}{l}
						\frac{\ederiv p}{\ederiv t} = -10\sin q \\
						\frac{\ederiv q}{\ederiv t} = p
					\end{array}
				\right. \;, \label{eq::pendulum}
			\end{equation}
			 with Hamiltonian $H(p,q)=\frac{p}{2}-10\cos q$. 
			 
			 The pendulum problem, \eqref{eq::pendulum} is a Hamiltonian system and here both mimetic integrators show again a clear advantage when compared to the Runge-Kutta integrator. As can be seen in \figref{fig:time_pendulum_phase}, both mimetic integrators are capable of keeping the numerical trajectories close to the analytical solution, therefore retaining the qualitative behavior of the problem, whilst the Runge-Kutta integrator shows an inward spiraling behavior. This numerical solution is characteristic of a damped pendulum problem (with friction) but the problem being solved is a frictionless pendulum. This issue is relevant since the numerical method interferes with the physics of the problem being solved.

			 \begin{figure}[ht]
				\centering
					\subfigure{
				\includegraphics[width=0.35\textwidth]{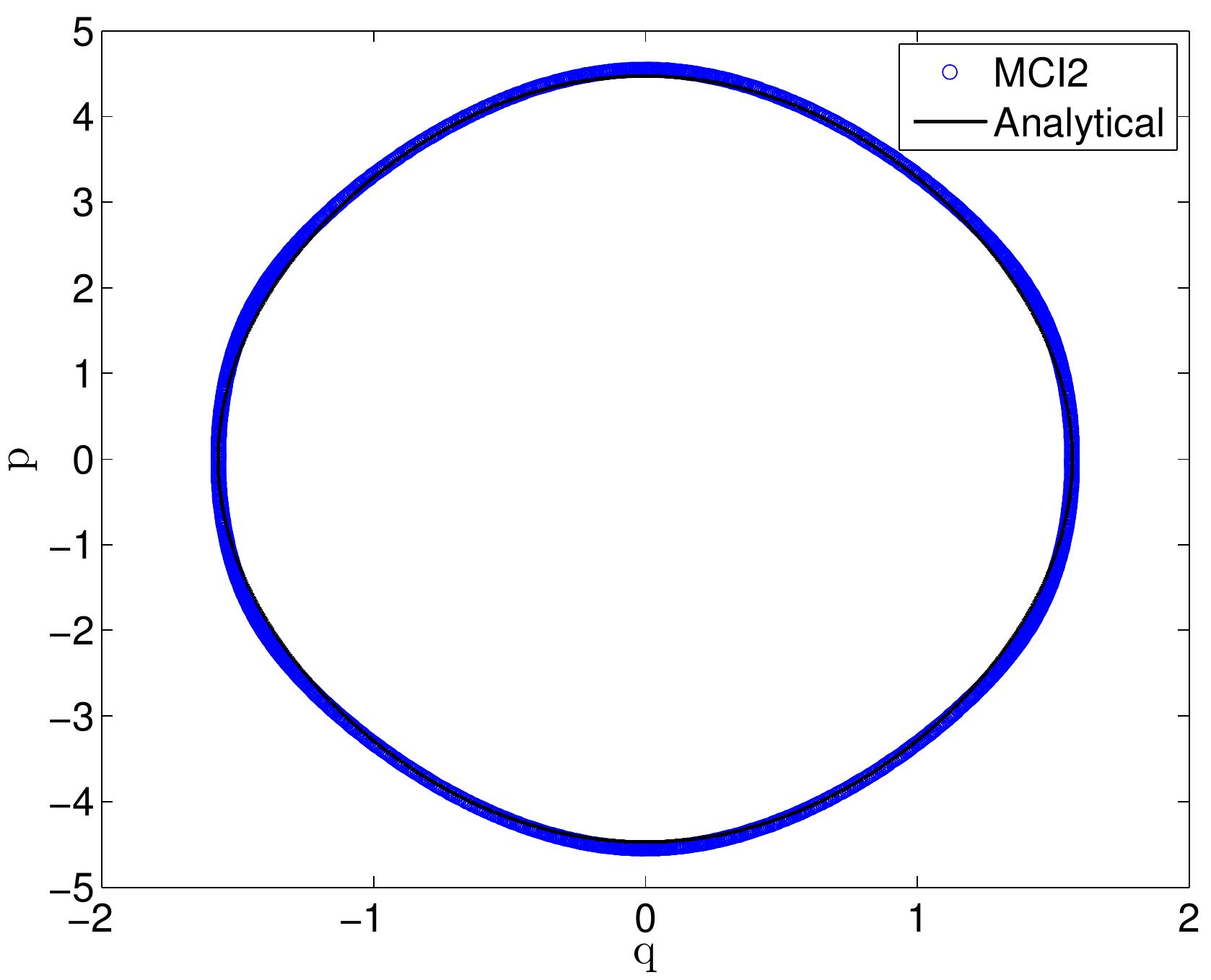}
						\label{fig:time_pendulum_phase_mci2}
					}
					\subfigure{
				\includegraphics[width=0.35\textwidth]{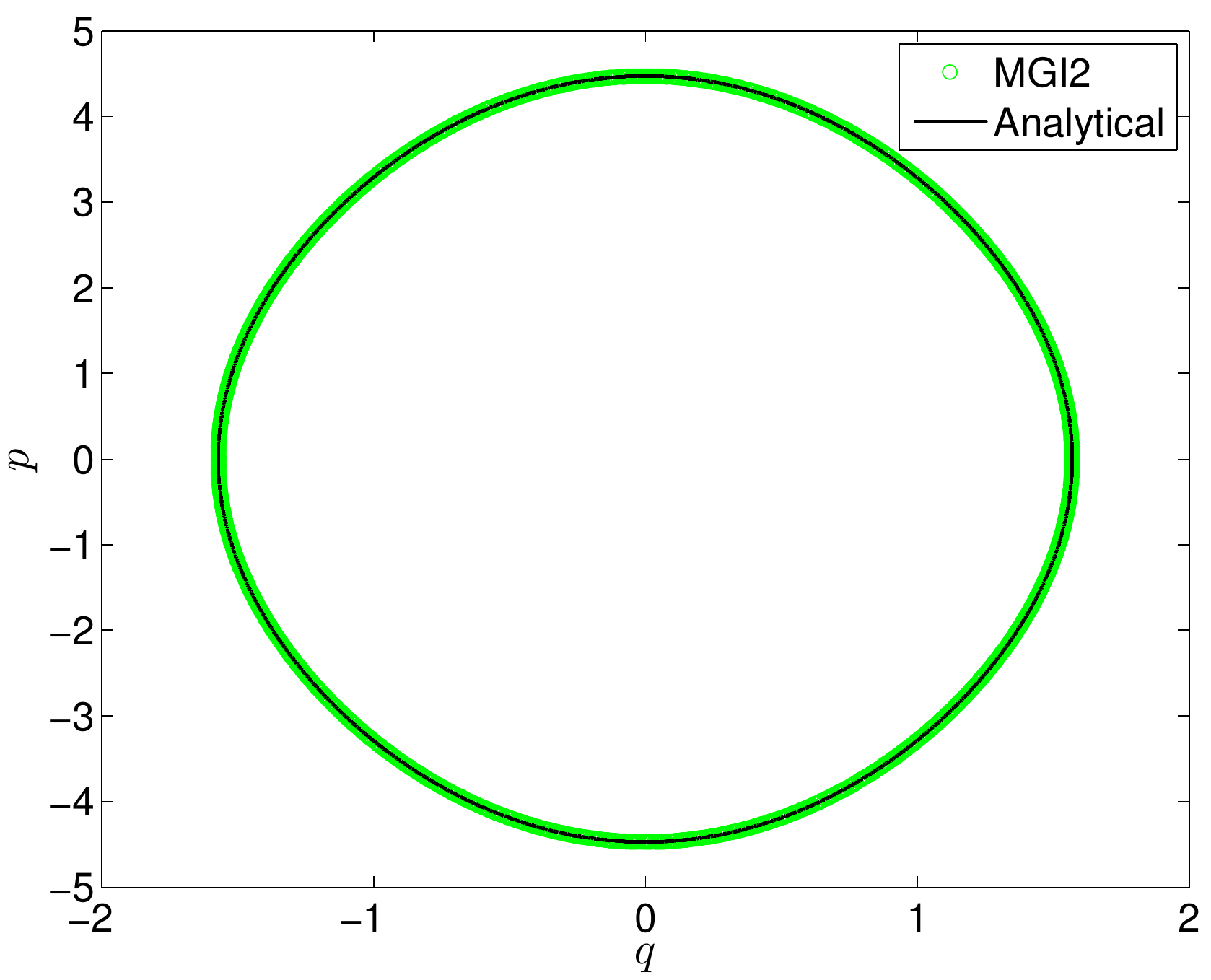}
						\label{fig:time_pendulum_phase_mgi2}
					}
					\subfigure{
				\includegraphics[width=0.35\textwidth]{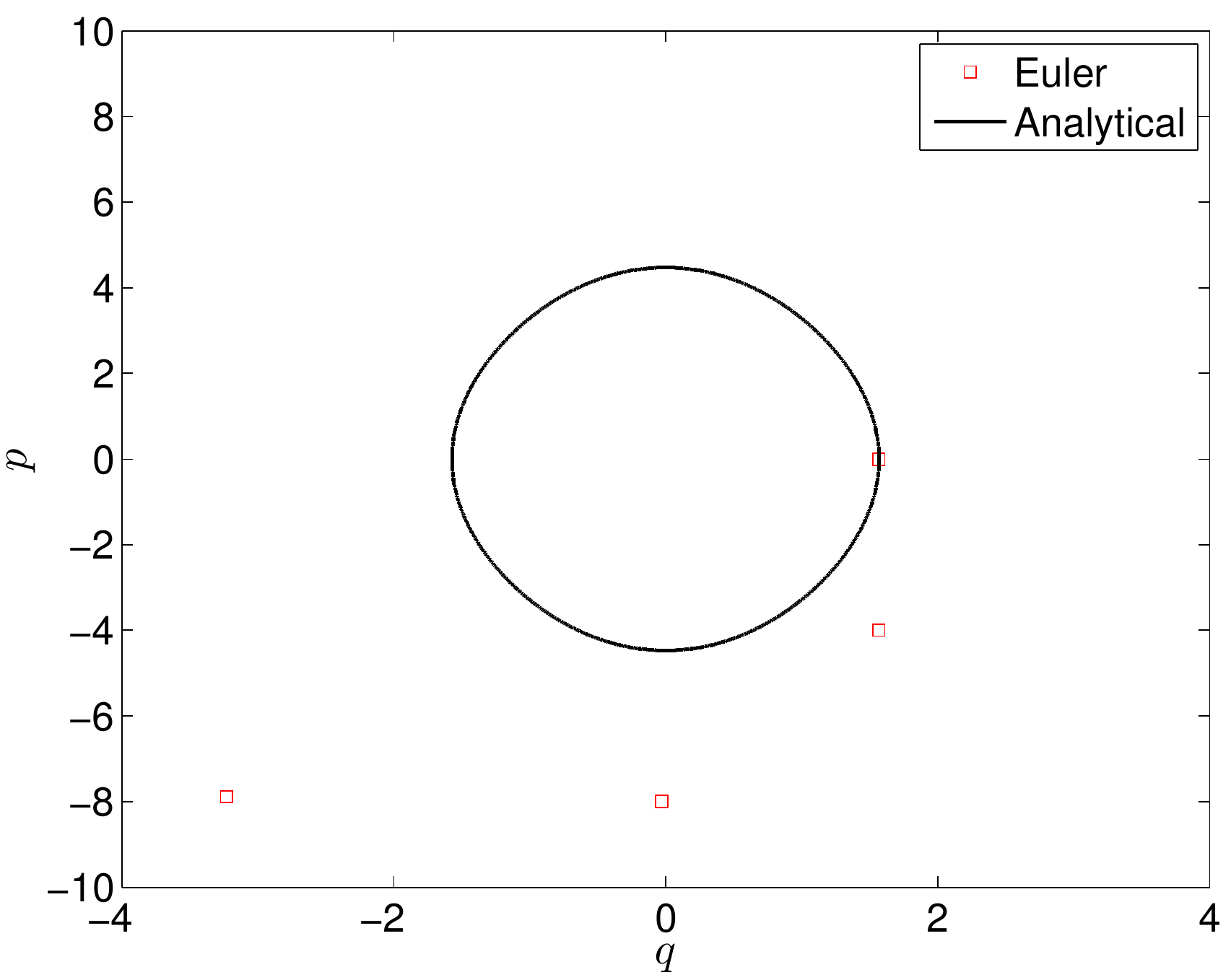}
						\label{fig:time_pendulum_phase_euler}
					}
					\subfigure{
				\includegraphics[width=0.35\textwidth]{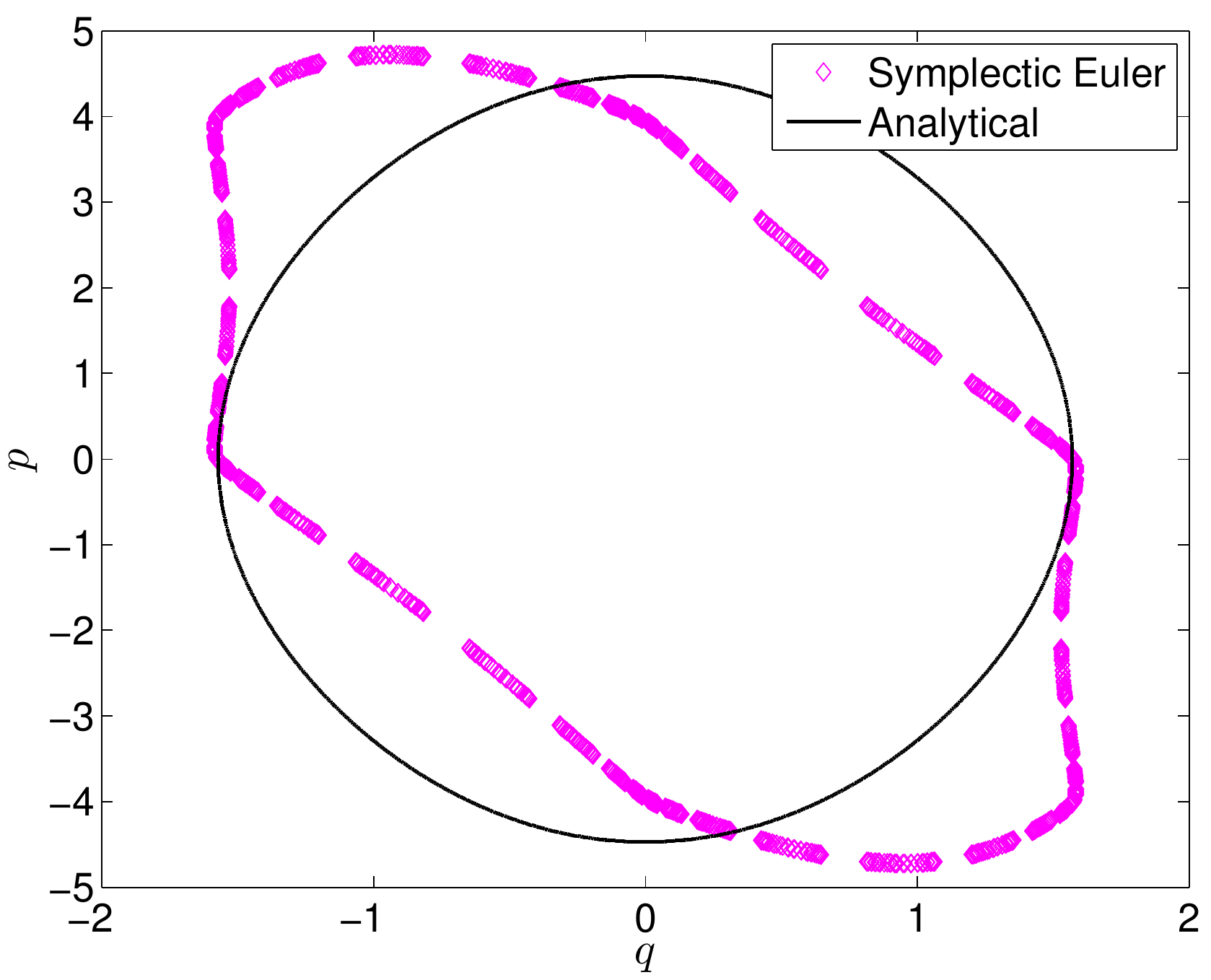}
						\label{fig:time_pendulum_phase_symplectic_euler}
					}
					\subfigure{
				\includegraphics[width=0.35\textwidth]{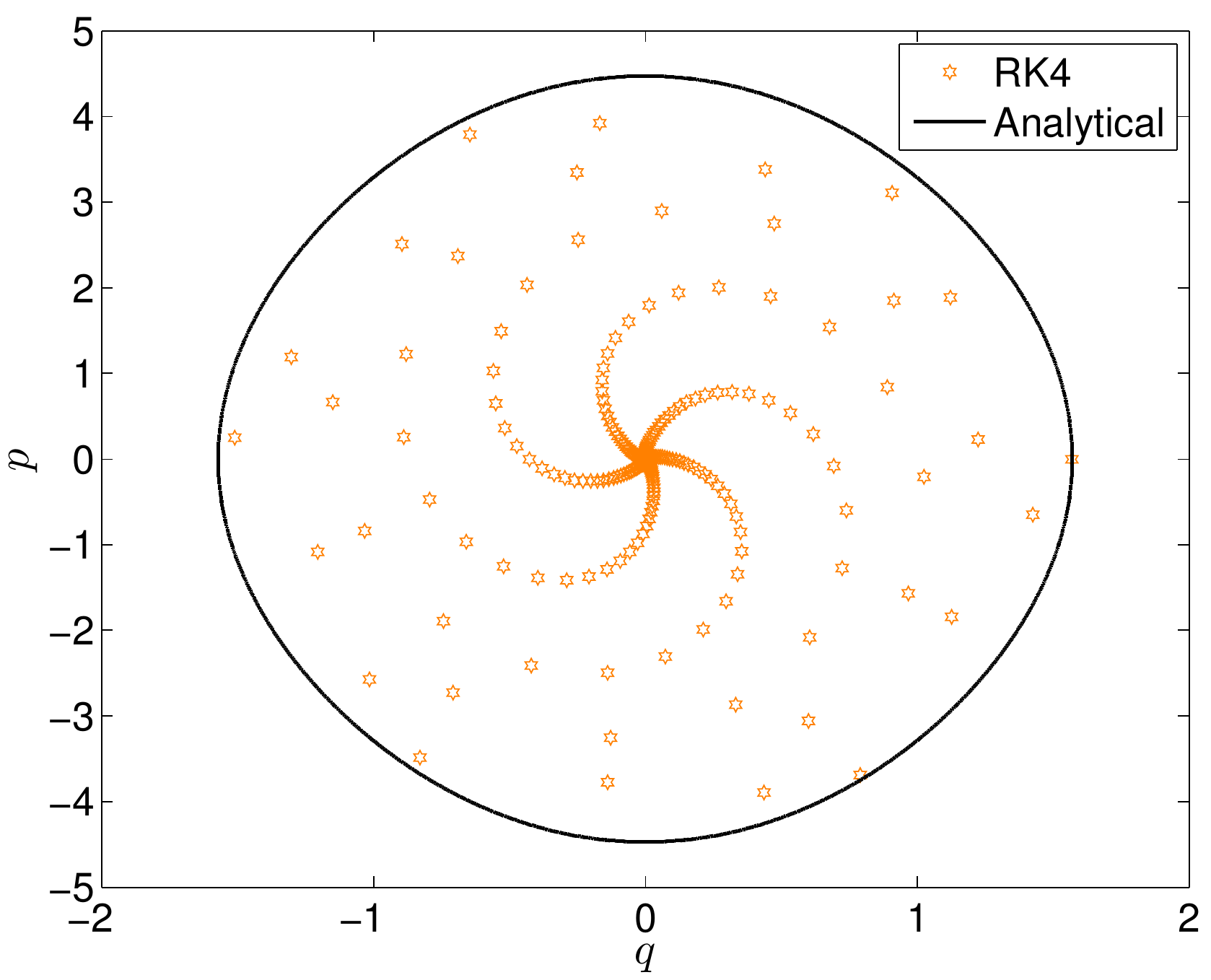}
						\label{fig:time_pendulum_phase_RK4}
					}
					\caption{Phase-space numerical (in color) and analytical (in black) solutions of pendulum system, \eqref{eq::pendulum}, with initial condition $q=\frac{\pi}{2}$, $p=0$ and time step $\Delta t = 0.4s$. Top left: mimetic canonical integrator with polynomial order in time $p_{t}=2$ (MCI2). Top right: mimetic Galerkin integrator with polynomial order in time $p_{t}=2$ (MGI2). Center left: explicit Euler scheme, which clearly diverges, with the system rapidly gaining energy. Center right: symplectic Euler scheme, which keeps a stable orbit but with a large difference in comparison to the analytical trajectory. Bottom: explicit Runge-Kutta of 4th order, which clearly is dissipative.}
					\label{fig:time_pendulum_phase}
				\end{figure}
				
				\begin{figure}[ht]
					\centering
					\subfigure{
								\includegraphics[width=0.44\textwidth]{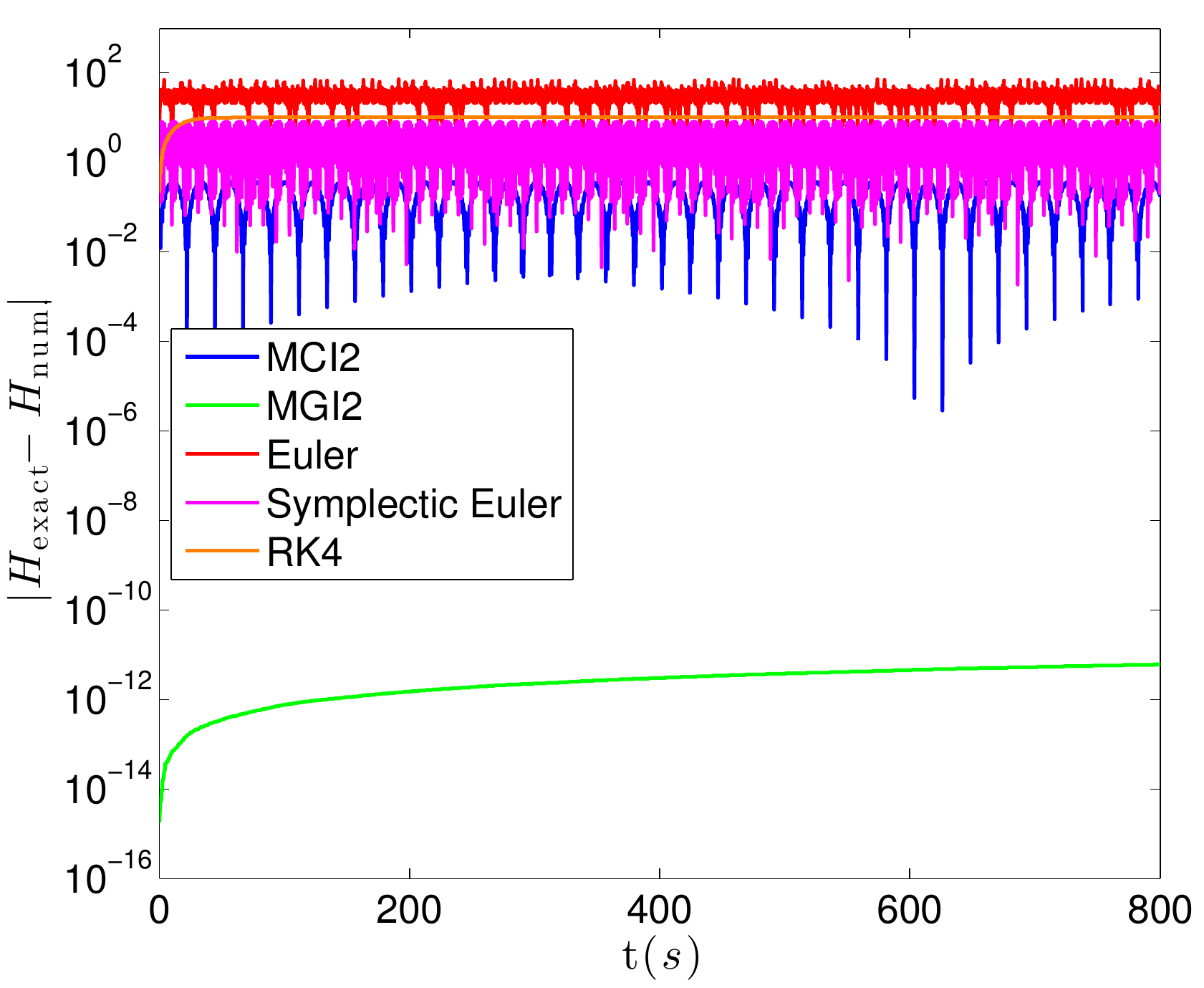}
						\label{fig:time_pendulum_error_time}
					}
					\subfigure{
								\includegraphics[width=0.44\textwidth]{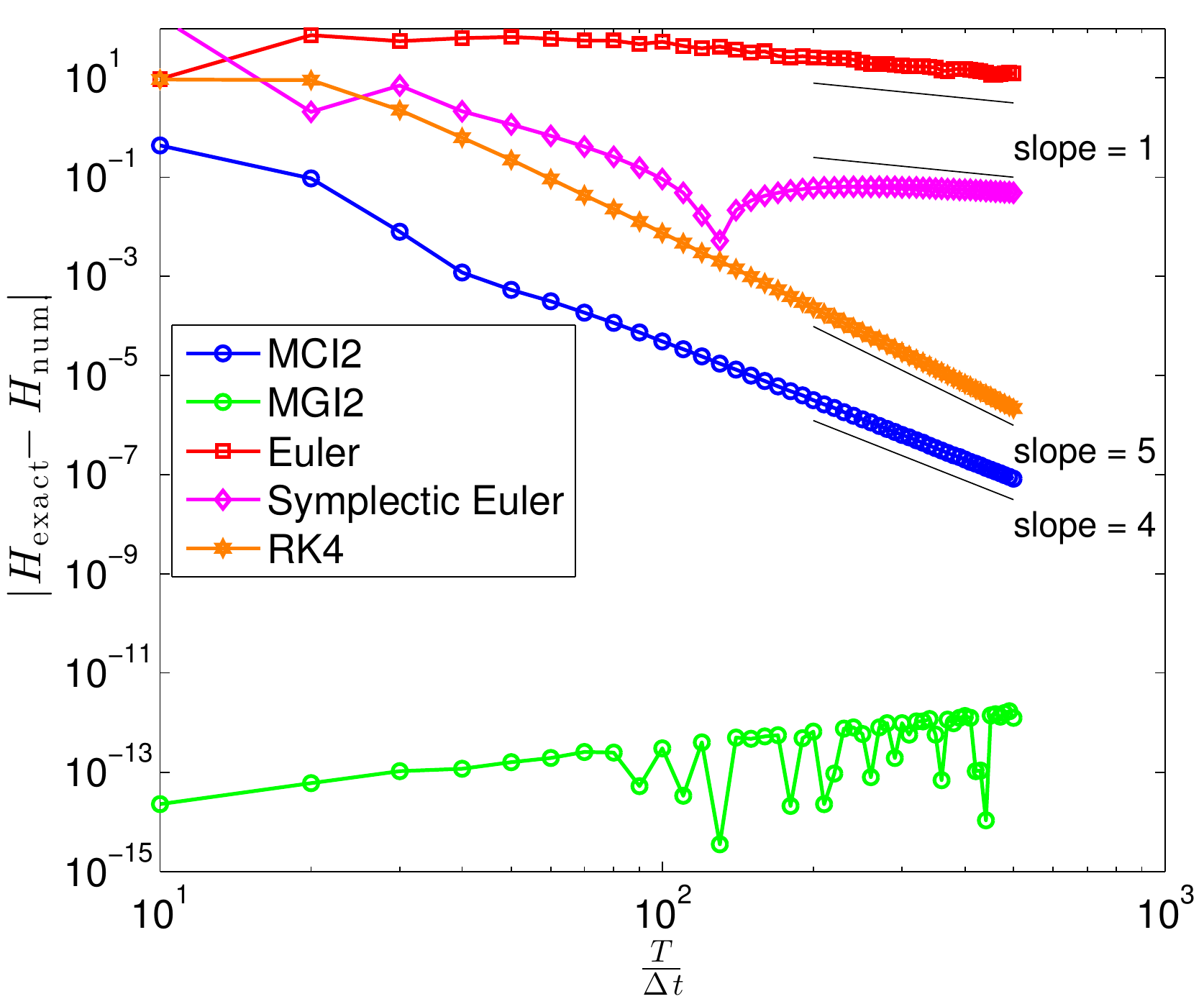}
						\label{fig:time_pendulum_error_convergence}
					}
					\caption{Left: Error in time in the Hamiltonian  $H = \frac{1}{2}p^{2} -10\cos q$ for the numerical solution of the pendulum problem, \eqref{eq::pendulum}, for $\Delta t=0.1s$, with the mimetic canonical integrator of polynomial order in time $p_{t}=2$ (MCI2), mimetic Galerkin integrator of polynomial order in time $p_{t}=2$ (MGI2), Runge-Kutta of 4th order, explicit Euler scheme and symplectic Euler method. Since the Hamiltonian, $H$, is not a quadratic term, the mimetic canonical integrator cannot exactly preserve this quantity along the trajectory, as can be seen. Nevertheless, the error remains bounded. On the other hand the mimetic Galerkin integrator exactly (up to machine precision) preserves it, since it is an exact energy preserving integrator. Both the explicit Euler and the Runge-Kutta methods, show an increase of the error with time. The symplectic Euler method, shows also that the error in the energy of the system is kept bounded. Right: Convergence of the error in $H$ as a function of $\frac{T}{\Delta t}$ for mimetic canonical integrator  of polynomial order in time $p_{t}=2$ (MCI2), mimetic Galerkin integrator of polynomial order in time $p_{t}=2$ (MGI2), Runge-Kutta of 4th order, explicit Euler and symplectic Euler. It is possible to observe the 4th order rate of convergence of the mimetic canonical integrator and the exact solution for the mimetic Galerkin integrator. The Runge-Kutta integrator shows a 5th order rate of convergence.}
					\label{fig:time_pendulum_error}
					\end{figure}

					\FloatBarrier

			\subsection{Two-body Kepler problem}
				Finally, the mimetic time integrators were applied to the solution of the two-body Kepler problem,
			\begin{equation}
				\left\{
					\begin{array}{l}
						\frac{\ederiv q_{1}}{\ederiv t} = p_{1}\\
						\frac{\ederiv q_{2}}{\ederiv t} = p_{2}\\
						\frac{\ederiv p_{1}}{\ederiv t} = -\frac{q_{1}}{\left(q_{1}^{2} + q_{2}^{2}\right)^{\frac{3}{2}}} \\
						\frac{\ederiv p_{2}}{\ederiv t} = -\frac{q_{2}}{\left(q_{1}^{2} + q_{2}^{2}\right)^{\frac{3}{2}}}
					\end{array}
				\right. \;,  \label{eq::kepler}
			\end{equation}
			a higher order Hamiltonian system with Hamiltonian $H(p_{1},p_{2},q_{1},q_{2}) = \frac{1}{2}\left(p_{1}^{2}+p_{2}^{2}\right) - \frac{1}{\sqrt{q_{1}^{2}+q_{2}^{2}}}$.
			
			This problem is also a Hamiltonian system and again both mimetic integrators show a clear advantage when compared to the Runge-Kutta integrator. As can be seen in \figref{fig:time_kepler_phase}, both mimetic integrators are capable of keeping the numerical trajectories close to the analytical solution, therefore retaining the qualitative behavior of the problem, whilst the Runge-Kutta integrator shows a slow inward spiraling behavior, once more indicating an artificial damping effect. The symplectic Euler shows a precession in the trajectory.

			Looking at the error in the Hamiltonian with time,\figref{fig:time_kepler_error} (left), one can clearly see that the mimetic Galerkin integrator exactly preserves, up to machine accuracy, the Hamiltonian. The mimetic canonical integrator, although not preserving the Hamiltonian exactly, keeps its error bounded. On the other hand, the Runge-Kutta and the Euler integrators quickly build up the error in the Hamiltonian. In \figref{fig:time_kepler_error} (right) one can see the convergence rates of the different integrators. The Runge-Kutta shows a convergence rate close to order five and the mimetic canonical integrator shows a convergence rate of order four. In both cases the mimetic Galerkin integrator exactly conserves the Hamiltonian, therefore the convergence rate is shown as zero since the error is always zero.

				\begin{figure}[ht]
				\centering
					\subfigure{
				\includegraphics[width=0.35\textwidth]{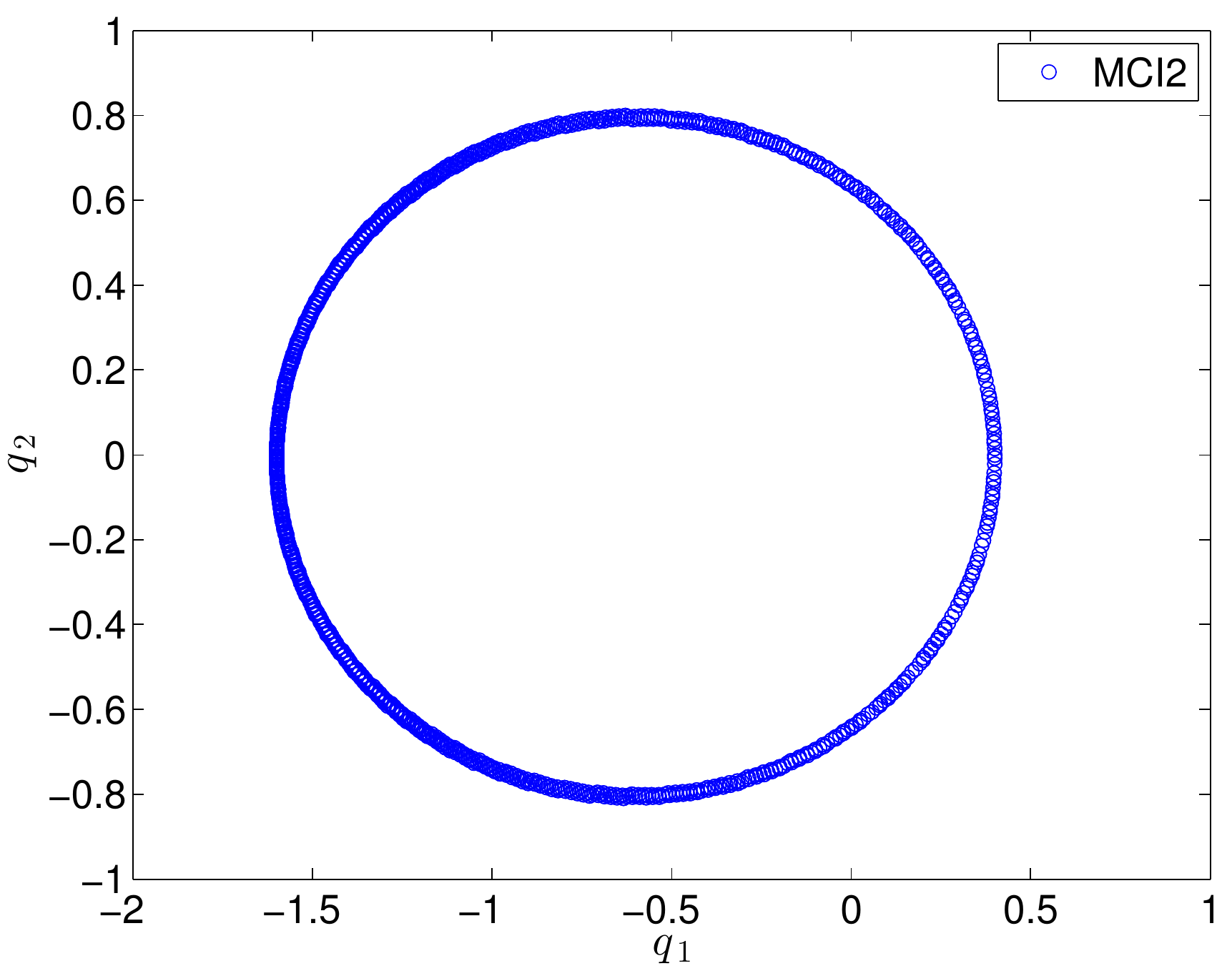}
						\label{fig:time_kepler_phase_mci2}
					}
					\subfigure{
				\includegraphics[width=0.35\textwidth]{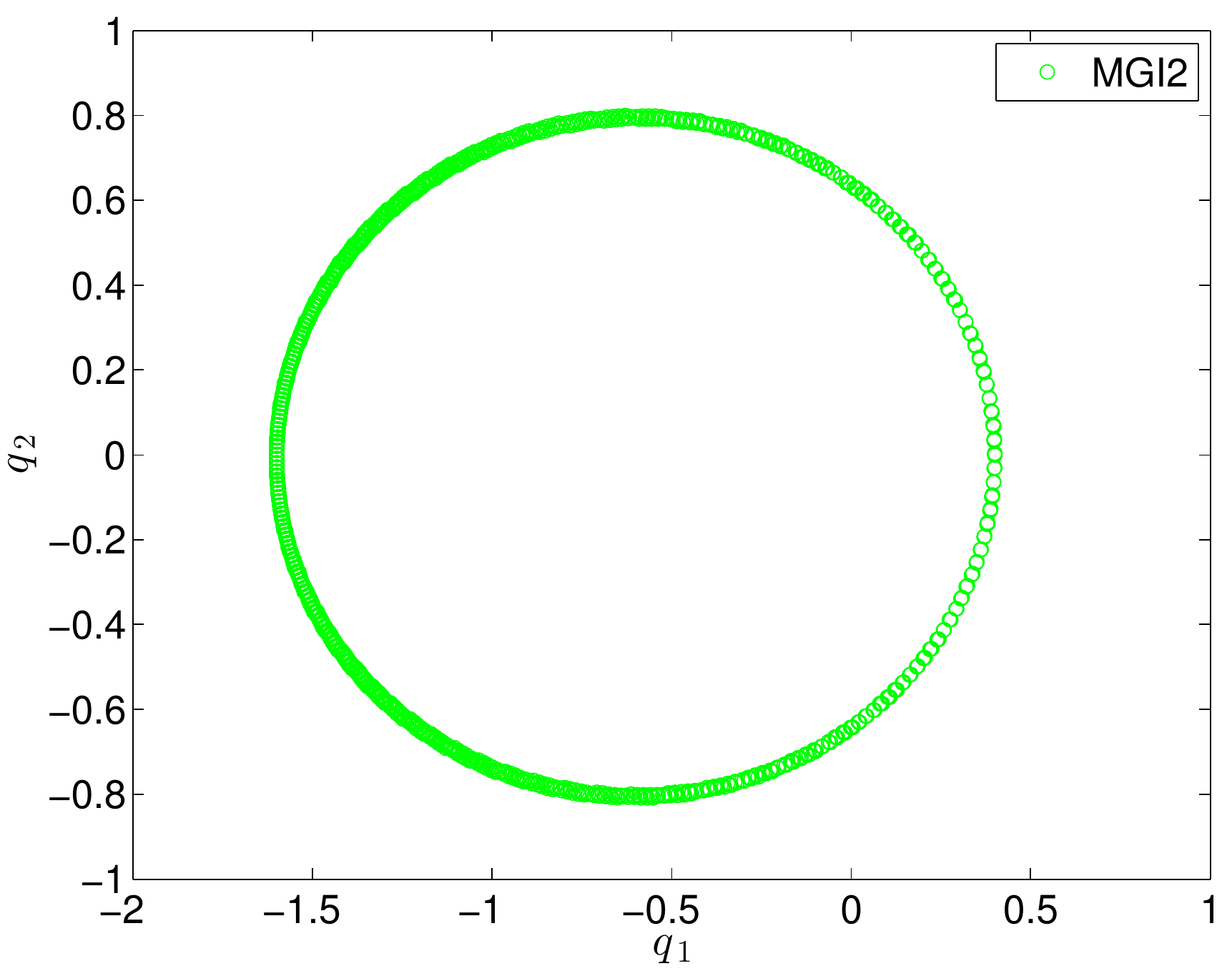}
						\label{fig:time_kepler_phase_mgi2}
					}
					\subfigure{
				\includegraphics[width=0.35\textwidth]{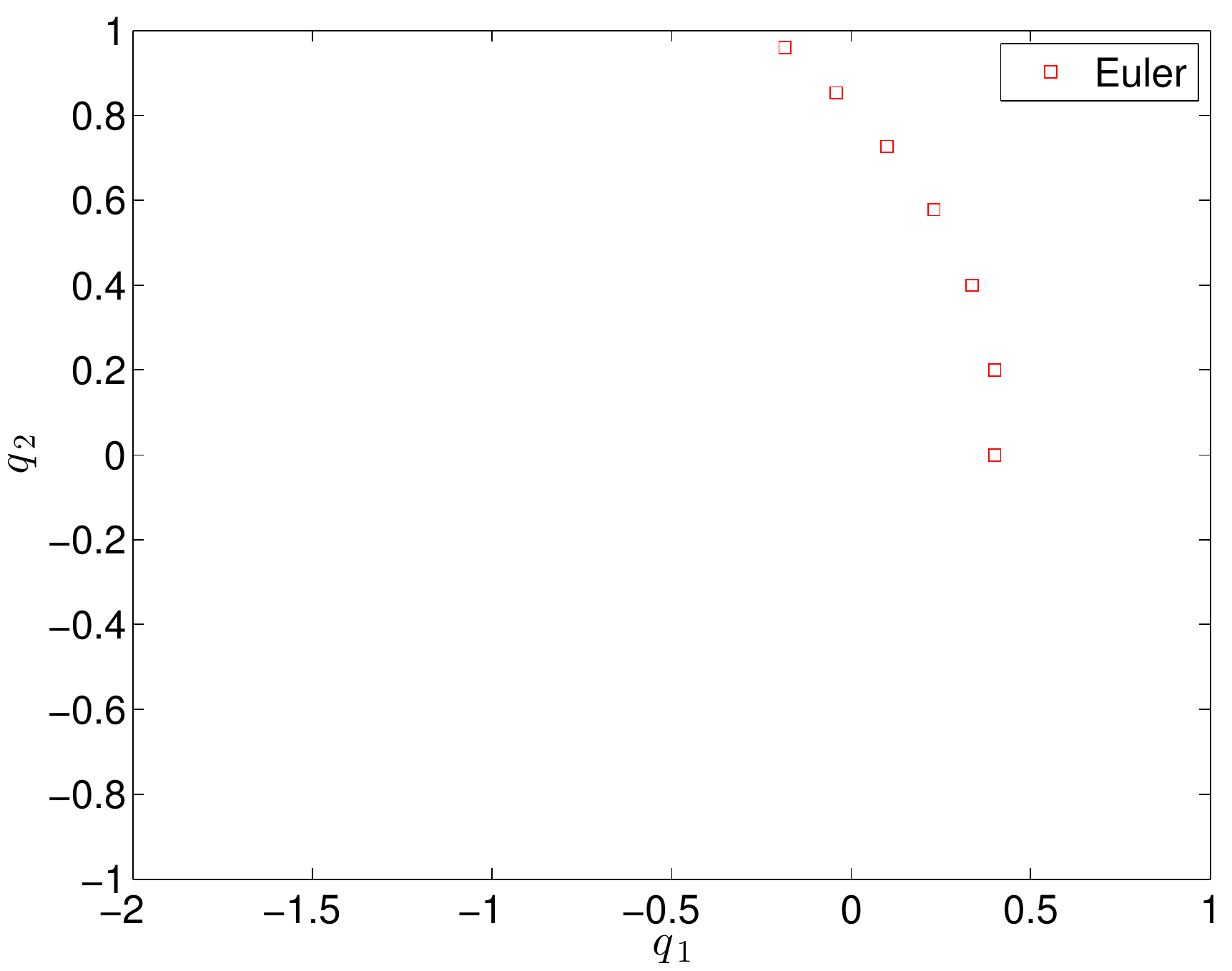}
						\label{fig:time_kepler_phase_euler}
					}
					\subfigure{
				\includegraphics[width=0.35\textwidth]{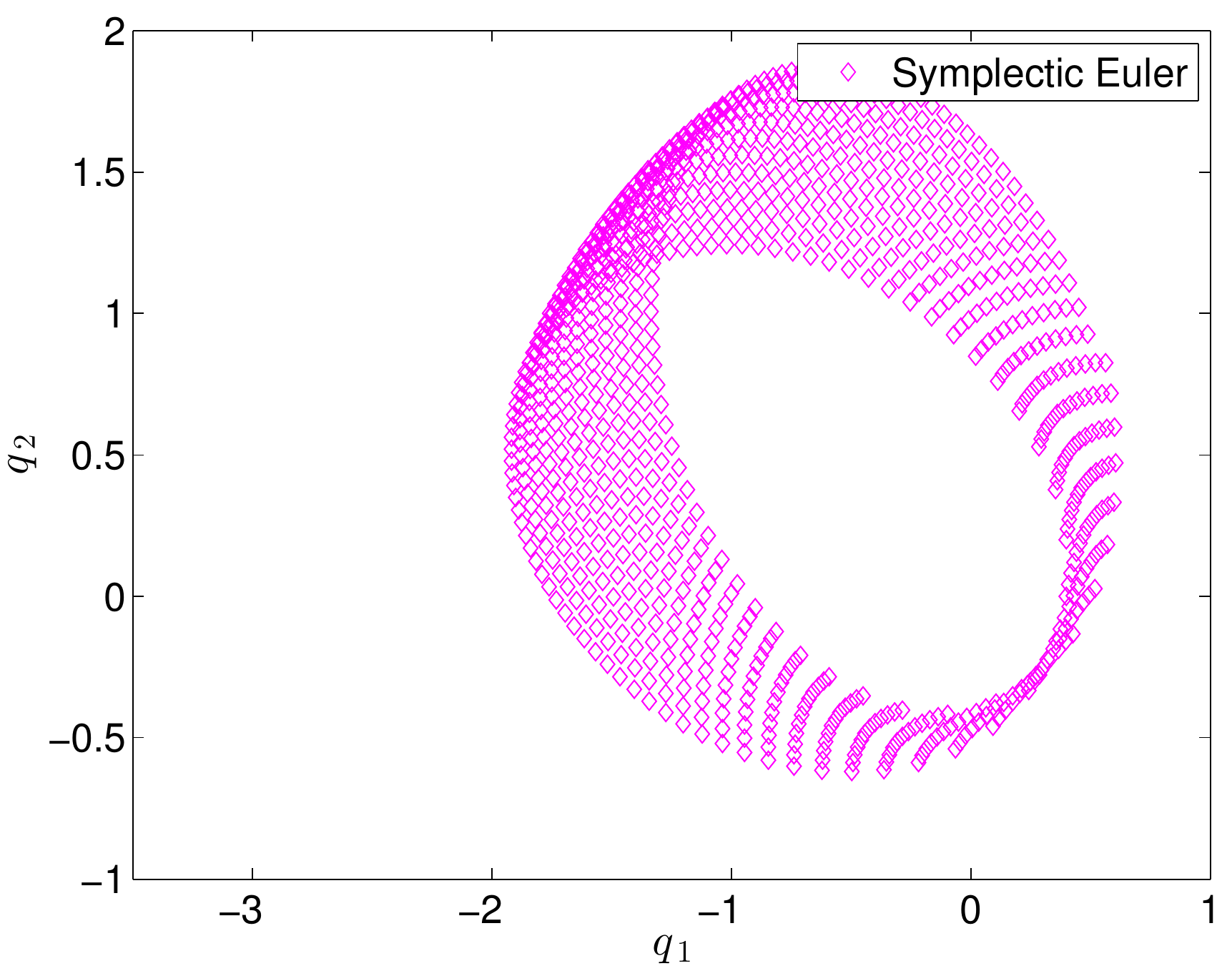}
						\label{fig:time_kepler_phase_symplectic_euler}
					}
					\subfigure{
				\includegraphics[width=0.35\textwidth]{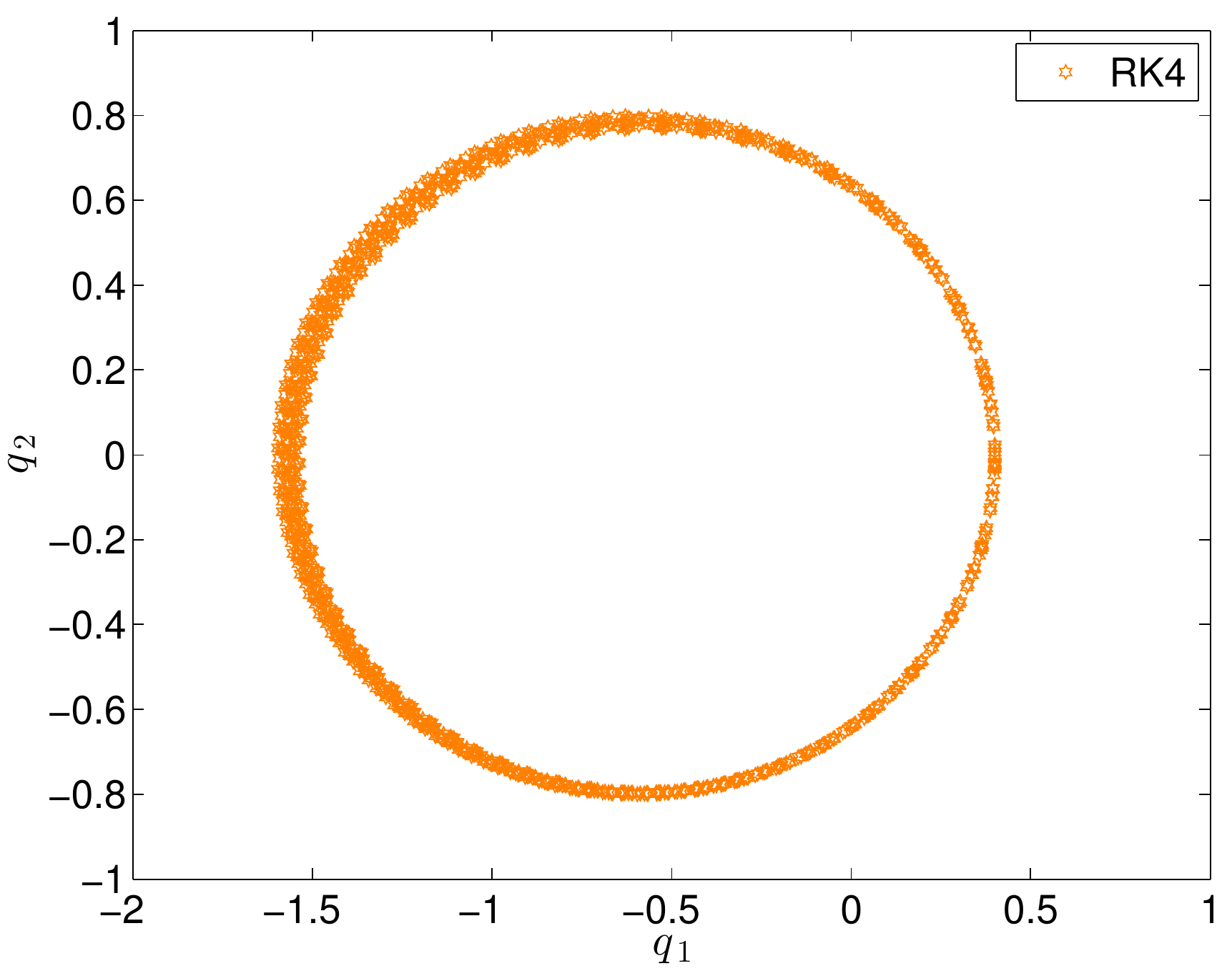}
						\label{fig:time_kepler_phase_RK4}
					}
					\caption{Phase-space numerical solution (in color) of Kepler two body system, \eqref{eq::kepler} with initial condition $q_{1}=0.4$,  $p_{1}=0$, $q_{2}=0$ and $p_{2}=2$ and time step $\Delta t = 0.1s$. Top left: mimetic canonical integrator with polynomial order in time $p_{t}=2$ (MCI2). Top right: mimetic Galerkin integrator with polynomial order in time $p_{t}=2$ (MGI2). Center left: explicit Euler scheme, where clearly the trajectory diverges, with the system rapidly gaining energy. Center right: symplectic Euler scheme, where one can see the precession in the trajectory. Bottom: explicit Runge-Kutta of 4th order, which diverges, loosing energy.}
					\label{fig:time_kepler_phase}
				\end{figure}

				\begin{figure}[ht]
					\centering
					\subfigure{
								\includegraphics[width=0.44\textwidth]{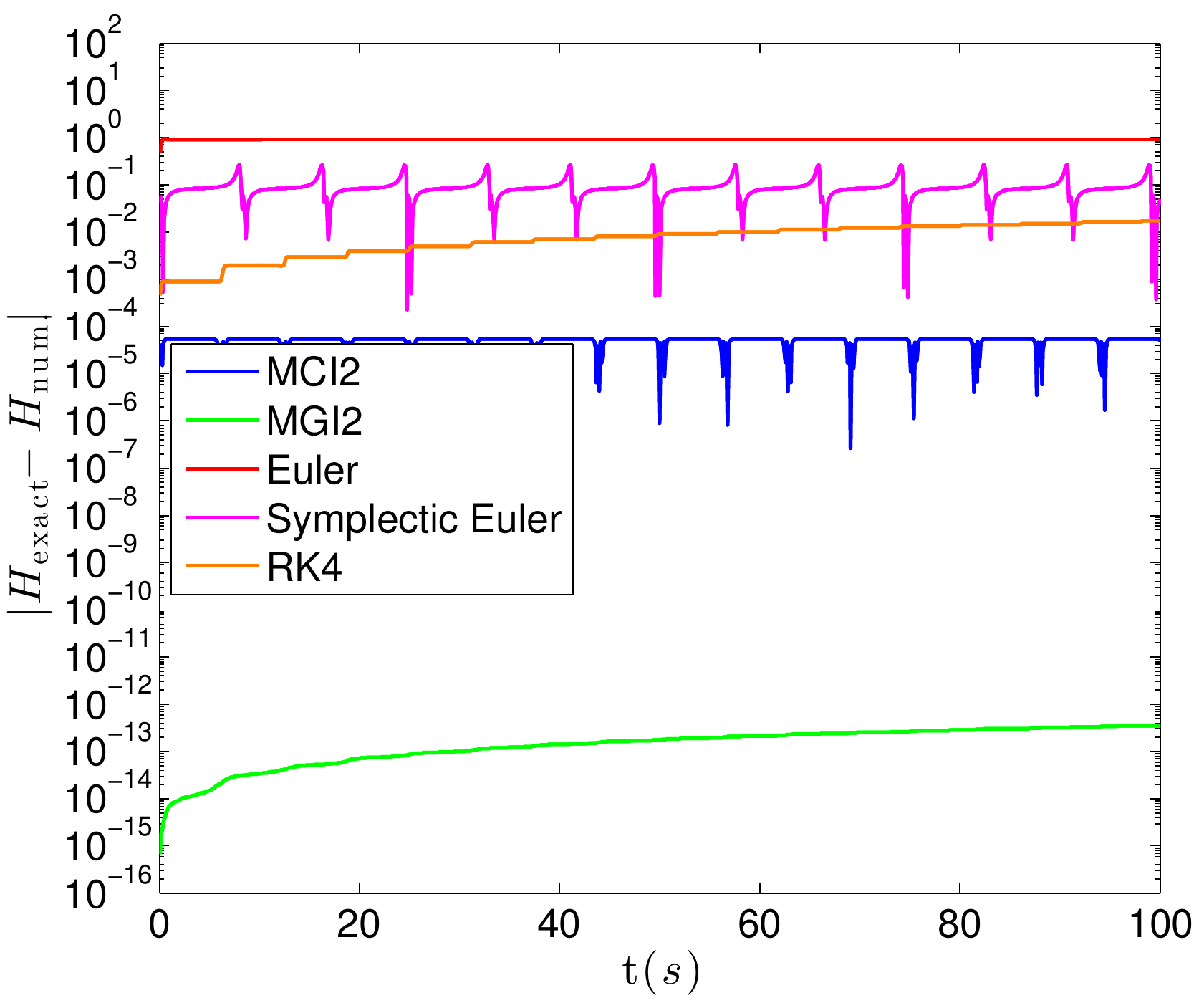}
						\label{fig:time_kepler_error_time}
					}
					\subfigure{
								\includegraphics[width=0.44\textwidth]{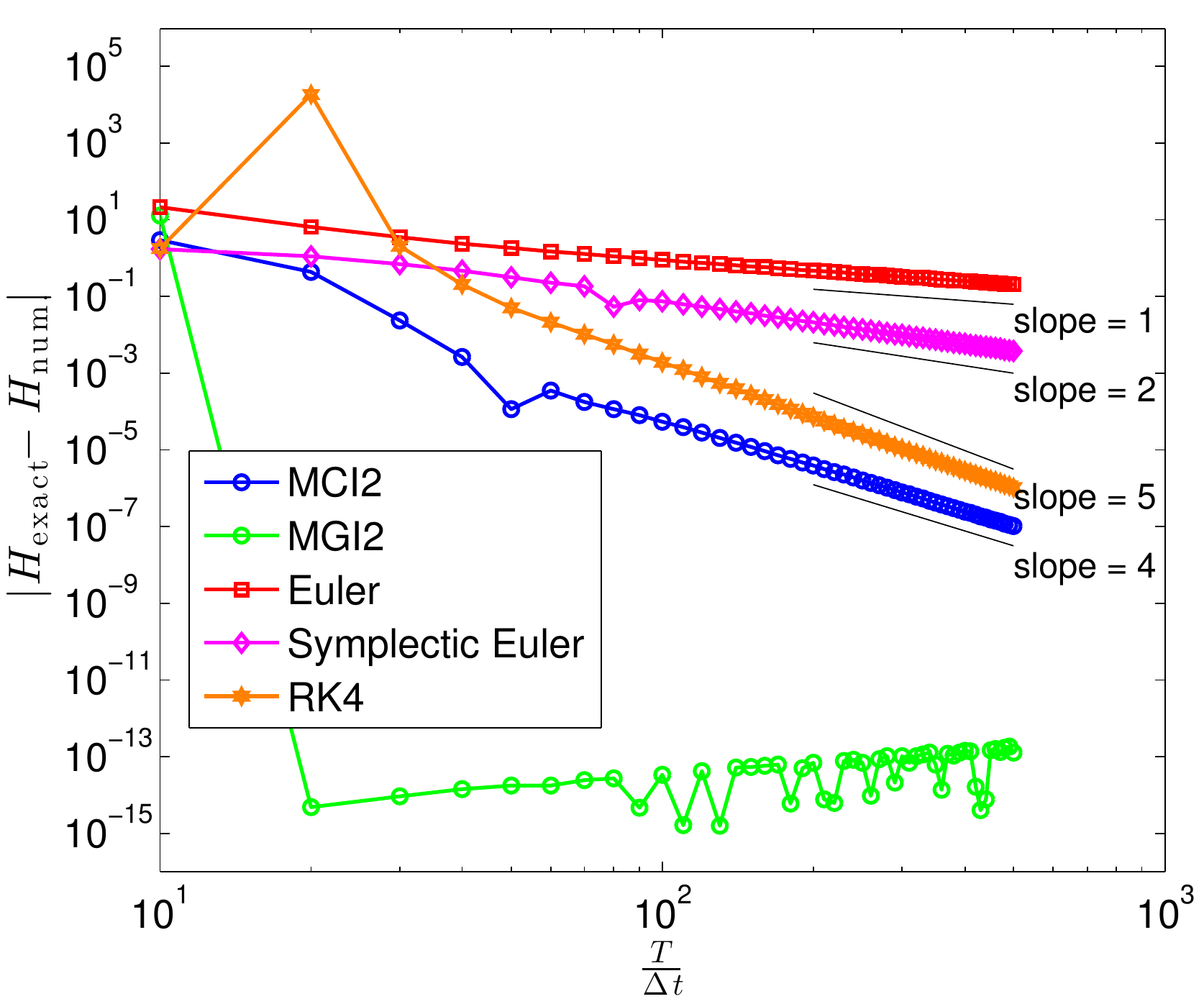}
						\label{fig:time_kepler_error_convergence}
					}
					\caption{Left: Error in time in the Hamiltonian $H = \frac{1}{2}\left(p_{1}^{2}+p_{2}^{2}\right) - \frac{1}{\sqrt{q_{1}^{2}+q_{2}^{2}}}$ for the numerical solution of the Kepler two body problem, \eqref{eq::kepler}, for $\Delta t = 0.1s$ with the mimetic canonical integrator of polynomial order in time $p_{t}=2$ (MCI2), mimetic Galerkin integrator of polynomial order in time $p_{t}=2$ (MGI2), Runge-Kutta of 4th order, explicit Euler scheme and symplectic Euler method. Since the Hamiltonian, $H$, is not a quadratic term, the mimetic canonical integrator cannot exactly preserve this quantity along the trajectory, as can be seen. Nevertheless, the error remains bounded. On the other hand the mimetic Galerkin integrator exactly (up to machine precision) preserves it, since it is an exact energy preserving integrator. Both the explicit Euler and the Runge-Kutta methods, show an increase of the error with time. The symplectic Euler manages to keep the error in the energy bounded. Right: Convergence of the error in $H$ as a function of the number of time steps for mimetic canonical integrator  of polynomial order in time $p_{t}=2$ (MCI2), mimetic Galerkin integrator of polynomial order in time $p_{t}=2$ (MGI2), Runge-Kutta of 4th order, explicit Euler and symplectic Euler. It is possible to observe the 4th order rate of convergence of the mimetic canonical integrator and the exact solution for the mimetic Galerkin integrator. The Runge-Kutta integrator shows a 5th order rate of convergence.}
					\label{fig:time_kepler_error}
					\end{figure}

	\FloatBarrier

\section{Summary and outlook}
\label{Section::Conclusions}

In this work we took as a starting point the mimetic framework established in \cite{Kreeft2011,Palha2014} and applied it to the numerical solution of systems of ordinary differential equations. This framework is based on two well established mathematical constructs: differential geometry (for the continuous formulation) and algebraic topology (for the discrete formulation). The connection between the two is established by a projection operator. With this framework, the topological relations can be exactly represented at the discrete level and all approximations lie in the metric dependent operators (e.g. Hodge-$\star$). Moreover, these metric dependent operators do not have a unique discrete representation. Therefore, for example, depending on the representation of the Hodge-$\star$ operator, it is possible to construct discretizations with different properties.

When this framework is applied to the solution of systems of ordinary differential equations, we have shown that two different time integrators can be obtained, one symplectic and the other exact energy preserving. The two integrators are obtained simply by choosing a specific discrete representation of the Hodge-$\star$ operator. For a discrete Hodge-$\star$ based on a canonical Hodge the integrator is symplectic, for a discrete Hodge-$\star$ based on a Galerkin Hodge the integrator is exact energy preserving.

In future work, we intend to build a physically accurate framework for time dependent systems of partial differential equations. Furthermore, the symmetry property that characterizes both integrators presented can play an important role, for instance, in the application of the mimetic framework to the advection equation. Moreover, given the properties of the integrators presented in this work, and the properties obtained for pure spatial discretizations, \cite{Palha2014,kreeft::stokes}, we also intend to explore a space-time formulation of the mimetic framework.


\section{Acknowledgments}
	Artur Palha was partially supported by the FCT (Foundation for Science and Technology, Portugal) grant SFRH/ BD/36093/2007.


\bibliographystyle{elsart-num-sort} 
\def\url#1{}
\bibliography{library_clean}

\end{document}